\documentclass[letterpaper, 11pt, oneside]{amsart}
\usepackage[left=0.6in,
            right=0.6in,
            top=0.7in,
            bottom=0.6in,]{geometry}
\usepackage{mathtools,verbatim, multirow}
\usepackage{amsmath, amsthm, fdsymbol}
\usepackage{silence} 
\usepackage{chessfss} %get chess symbols
\usepackage{amsfonts,siunitx}
\usepackage{amssymb}
\usepackage{graphicx}
\usepackage[shortlabels]{enumitem}
\usepackage{tcolorbox}
\usepackage[backend=biber, style=alphabetic-verb]{biblatex}
\usepackage{amscd, amstext, hyperref, xcolor, float, tabularray,bm, tikz-cd}
\usepackage[section]{placeins} % prevent floats from crossing sections
\usepackage{tablefootnote,utfsym}
\theoremstyle{definition}
\newtheorem{theorem}{Theorem}[section]
\newtheorem*{mainthm*}{Main Theorem}
\newtheorem{prop}[theorem]{Proposition}
\newtheorem{remark}[theorem]{Remark}
\newtheorem{ex}[theorem]{Example}
\newtheorem{definition}[theorem]{Definition}

\newtheorem{lemma}[theorem]{Lemma}

\newtheorem{thm}[theorem]{Theorem}

\newcommand{\Mod}[1]{\ (\mathrm{mod}\ #1)}
\newcommand{\q}{q^{\times}}
\newcommand{\Z}{\mathbb{Z}}
\newcommand{\R}{\mathbb{R}}
\newcommand{\Q}{\mathbb{Q}}
\newcommand{\F}{\mathbb{F}}
\newcommand{\C}{\mathbb{C}}

\newcommand{\PP}{\mathbb{P}}

\newcommand{\lcd}{\text{lcd}}

\newcommand{\Fsf}{\mathsf F}
\newcommand{\Lsf}{\mathsf L}
\newcommand{\Csf}{\mathsf C}

\title{Hypergeometric decomposition of Delsarte K3 pencils}
\date{\today}
\begin{document}

\author{Rachel Davis}
\address{Department of Mathematics,
University of Wisconsin-Madison, 480 Lincoln Drive
Madison WI 53706}
\email{rachel.davis@wisc.edu}
\thanks{The authors thank Rethinking Number Theory, an AIM Research Community, for convening and supporting this project. We thank ICERM for allowing us to meet in person through Collaborate@ICERM. We are indebted to John Voight and Tyler L. Kelly for thoughtful and inspirational conversations. We thank Tristan H\"ubsch for helpful comments. We appreciate the meticulous reading of an anonymous referee. We thank Asem Abdelraouf for the insight he provided on finite field hypergeometric functions. CoCalc, Edgar Costa, and Brian J. Ferguson provided key computational support.}

\author{Jessamyn Dukes}
\address{University of Denver, 2199 University Blvd, Denver CO 80210}
\email{jessamyn.dukes@du.edu}

\author{Thais Gomes Ribeiro}
\address{School of Mathematics, University of Birmingham, Ring Road N B152TT, UK}
\email{txg306@student.bham.ac.uk}

\author{Eli Orvis}
\address{University of Colorado at Boulder, Boulder CO 80309}
\email{eli.orvis@colorado.edu}

\thanks{The fourth author gratefully acknowledges the School of Mathematics at the University of Birmingham for the PhD funding that enabled her to dedicate time to this project and further emphasizes the invaluable advice she received from Tyler L. Kelly. She also thanks the University of Sydney for its support.}

\author{Adriana Salerno}
\address{National Science Foundation, 2415 Eisenhower Ave, Alexandria, VA 22314}
\email{asalerno@nsf.gov}
\thanks{For the fifth author, this material is based upon work supported by and while serving at the National Science Foundation. Any opinion, findings, conclusions or recommendations expressed in this material are those of the authors and do not necessarily reflect the views of the National Science Foundation. }

\author{Leah Sturman}
\address{Southern Connecticut State University, 501 Crescent Street New Haven CT 06515}
\email{sturmanL1@southernct.edu}

\author{Ursula Whitcher}
\address{Mathematical Reviews, American Mathematical Society,
535 W William St Ste 210, Ann Arbor MI 48103}
\email{uaw@umich.edu}
\thanks{The last author thanks the Simons Center for Geometry and Physics for hosting the thought-provoking workshop Murmurations in Arithmetic Geometry and Related Topics, where an early version of this work was presented.}

\begin{abstract}
We study five pencils of projective quartic Delsarte K3 surfaces. Over finite fields, we give explicit formulas for the point counts of each family, written in terms of hypergeometric sums. Over the complex numbers, we match the periods of the corresponding family with hypergeometric differential operators and series. We also obtain a decomposition of the incomplete $L$-function of each pencil in terms of hypergeometric $L$-series and Dedekind zeta functions. This gives an explicit description of the hypergeometric motives geometrically realised by each pencil.   
\end{abstract}
\maketitle
\markboth{\shorttitle}{\shorttitle}
\tableofcontents

\section{Introduction}

In \cite{I58}, Igusa noticed an intriguing relation between the periods and the point-counts of the Legendre family of elliptic curves $$C_{\lambda}: y^2z=x(x-z)(x-\lambda z).$$

\vspace{0.2cm}

He showed that for a prime $p,$ the trace of Frobenius $a_p(C_{\lambda})=p+1-\#C_{\lambda}(\mathbb{F}_p)$ satisfies $$a_p(C_{\lambda}) \equiv (-1)^{\frac{p-1}{2}} {_2F}_1\left(\tfrac{1}{2},\tfrac{1}{2};1,1 \mid \lambda\right) \Mod{p}.$$ Clemens referred to Igusa's finding as an instance of `Manin's unity of Mathematics' (see \autocite[Section 2.10, 2.11 and 2.12]{C80}). In \cite{O98}, using a \emph{finite field hypergeometric function}, Ono made this congruence into a precise formula for the number of points of $C_{\lambda}.$ For $q=p^r,$ he proved that  \begin{align*}
\#C_{\lambda}(\mathbb{F}_q)=q+1-(-1)^{\frac{q-1}{2}}H_q\left(\tfrac{1}{2},\tfrac{1}{2};1,1 \mid \lambda\right). 
\end{align*}

The relationship between periods and point counts has been further explored in many cases. A few examples are pencils of quartics in $\mathbb{P}^3$ (\cite{dwork, Katz2009, HD20}) and pencils of quintics in $\mathbb{P}^4$ (\autocite{C00},\autocite{C04}). 

A geometric property of the examples above is that they are Calabi-Yau varieties, that is, elliptic curves, K3 surfaces, and their higher-dimensional generalizations. These are varieties with trivial canonical class (simply connected in dimensions 2 and above).  They form a fascinating arena for the Langlands program (see for example \cite{C25} and \cite{S07}) and a natural testbed for generalizations of the murmurations phenomena observed in \cite{He_2024}. In the context of mathematical physics, computational experiments with Calabi-Yau threefolds have yielded examples of rank two attractor points, demonstrating unexpected relationships to classical modular forms and yielding new evidence for the Hodge conjecture (see \cite{elmiEtAl, attractor}). 

Computing $L$-functions is a problem of inherent great interest. K3 surfaces, and Calabi-Yau varieties more generally, provide a natural setting for such exploration. Inspired by Deligne's conjecture on the $L$-value at zero (see \cite{De79}), we pursue the philosophy that the $L$-function encodes the relation between periods and point counts mentioned above. 

Many of the best-understood examples of K3 and Calabi-Yau zeta and $L$-functions are hypergeometric, as are the Picard-Fuchs differential equations satisfied by their holomorphic forms. We may view this phenomenon as an instance of the hypergeometric motives philosophy, which holds that a single underlying object connects classical hypergeometric functions to geometry and arithmetic. The study of finite field hypergeometric function manifestations of this philosophy is an active area of research. For example, the authors of \cite{LTYZ} used the theory of hypergeometric motives to establish supercongruences associated to rigid Calabi-Yau varieties. 
The present project provides concrete higher-dimensional realizations of hypergeometric motives in the sense of \cite{RRV}.

Any smooth quartic hypersurface in $\PP^3$ defines a K3 surface. We focus on \emph{Delsarte} or \emph{invertible} pencils of quartic K3 hypersurfaces. These are given by one-parameter deformations of certain quartics defined by invertible matrices of nonnegative integers. The most famous such family is the Dwork pencil given by $x_0^4+x_1^4+x_2^4+x_3^4-4 \psi x_0x_1x_2x_3=0$; in this case, the matrix is the $4 \times 4$ identity matrix multiplied by 4. The arithmetic and geometry of the Dwork pencil is a subject of longstanding interest; see, for example, \cite{dwork, Katz2009, BINI2014173} and \cite{DUAN2020389}. As generalizations of Fermat hypersurfaces, Delsarte hypersurfaces are also of great interest.

There are ten Delsarte (or invertible) quartics in $\PP^3$ up to isomorphism. Deforming by $-d^T \psi x_0x_1x_2x_3$, where $d^T$ is a constant determined by the initial quartic and $\psi$ is our parameter, we obtain ten invertible pencils, as described in Table~\ref{table: introDelsarteList} below. See \S~\ref{S:invertible} for more details.

 In \cite{ZF18}, two of the present authors and their collaborators showed that the four invertible pencils $\Lsf_3\Fsf_1$, $\Lsf_2\Fsf_2$, $\Lsf_2\Lsf_2$, and $\Lsf_4$ share a zeta function factor with the Dwork pencil $\Fsf_4$, and that a similar phenomenon holds in higher dimensions. They identified another pair of pencils, labeled $\Csf_2\Fsf_2$ and $\Csf_2\Lsf_2$ in Table~\ref{table: introDelsarteList}, that share a zeta function factor with each other. We indicate these families with shared interesting zeta function factors using the Suit column in Table~\ref{table: introDelsarteList}. In \cite{HD20}, the same group of authors studied the collection of five pencils including the Dwork pencil in more detail, giving a complete hypergeometric decomposition of their $L$-functions.

\begin{table}[ht]
    \centering
    \caption{Delsarte quartic pencils}
    \begin{tblr}{colspec ={|c|c|c|c|c|},cell{2}{1} = {r=5}{c},cell{7}{1} = {r=2}{c}}
    \hline
   Suit & Polynomial & Label & $d^T$ & $p$ bad\\
    \hline 
    \hline 
$\diamondsuit$ & $x_0^4+x_1^4+x_2^4+x_3^4$ & $\Fsf_4$ & 4 & $2$\\
& $x_0^3x_1+x_1^3x_2+x_2^3x_0+x_3^4$ & $\Lsf_3\Fsf_1$ & 4 & $2,7$\\
& $x_0^3x_1+x_1^3x_0+x_2^4+x_3^4$ & $\Lsf_2\Fsf_2$ & 4 & $2$ \\ 
& $x_0^3x_1+x_1^3x_0+x_2^3x_3+x_3^3x_2$ & $\Lsf_2\Lsf_2$ & 4 & $2$\\ 
& $x_0^3x_1+x_1^3x_2+x_2^3x_3+x_3^3x_0$& $\Lsf_4$ & 4 & $2,5$\\ 
    \hline 
$\clubsuit$ &    $x_0^3x_1+x_1^4+x_2^4+x_3^4$ & $\Csf_2\Fsf_2$ & 12 & $2,3$\\
&    $x_0^3x_1+x_1^4+x_2^3x_3+x_3^3x_2$ & $\Csf_2\Lsf_2$ & 12 & $2,3$\\
    \hline     
$\symking$ &    $x_0^3x_1+x_1^4+x_2^3x_3+x_3^4$ & $\Csf_2\Csf_2$ & 6 & $3$\\
    \hline 
$\spadesuit$ & $x_0^3x_1+x_1^3x_2+x_2^4+x_3^4$& $\Csf_3\Fsf_1$ & 36& $2,3$\\ 
    \hline 
$\heartsuit$ & $x_0^3x_1+x_1^3x_2+x_2^3x_3+x_3^4$ & $\Csf_4$ & 27 & $3$\\
    \hline
    \end{tblr}
    \label{table: introDelsarteList}
\end{table}
 
 Meanwhile, in \cite{KM17} Kloosterman studied Delsarte pencils $X_\psi$ in $\PP^n$ from a geometric viewpoint, giving an alternative description of the common factors and obtaining a Frobenius-stable decomposition of the cohomology. Building on results of Shioda in \cite{Shioda}, Kloosterman used a covering map from a one-parameter deformation $Y_\psi$ of a Fermat variety of degree $d_{\Fsf}$; this degree is always divisible by $d^T$ and may be significantly higher than the degree of the original pencil. (See \cite{K13} for another interesting application of Shioda maps to invertible hypersurfaces.) Kloosterman showed that, working over a field $\F_q$ whose order satisfies an appropriate divisibility condition, the middle cohomology $H^{n-1}(X_\psi)$ of the Delsarte pencil decomposes as a direct sum of Frobenius-stable subspaces. Combining with results from his earlier work \cite{kloostermanFermat}, Kloosterman observed that this decomposition will yield hypergeometric factors of the zeta function of $X_\psi$ when $q \equiv 1 \pmod{d_{\Fsf}}$. In particular, when $X_\psi$ is a K3 surface or a Calabi-Yau variety, the decomposition will always yield a hypergeometric factor corresponding to the holomorphic form. In the K3 case, the remaining subspace of middle cohomology is generated by algebraic curves; generators are given in \cite{KM17} for many of these cases.

The hypergeometric decomposition obtained for five Delsarte pencils in \cite{HD20} is consistent with the decomposition of $H^{n-1}(X_\psi)$ obtained in \cite{KM17} but does not make the restriction that $q \equiv 1 \pmod{d_{\Fsf}}$. In the current work, we determine $L$-functions for the remaining five K3 Delsarte pencils, $\Csf_2\Fsf_2$, $\Csf_2\Lsf_2$, $\Csf_2\Csf_2$, $\Csf_3\Fsf_1$, and $\Csf_4$. Our methods are explicit: we obtain formulas for the point counts in terms of finite field hypergeometric functions. Away from a finite set of bad primes, as in \cite{HD20}, we describe the $L$-functions of our pencils in terms of hypergeometric $L$-functions. 

\begin{table}
    \centering
     \caption{Hypergeometric parameters}
    \begin{tblr}{|c|c|c|c|}
    \hline
    Suit & $\bm{\alpha}$ & $\bm{\beta}$ & $t$ \\
    \hline  
    \hline 
    $\clubsuit_0$ & $\scriptstyle{\frac{1}{12},\frac{1}{6},\frac{5}{12},\frac{7}{12},\frac{5}{6},\frac{11}{12}}$ & $\scriptstyle{\frac{1}{2},\frac{1}{3},\frac{2}{3},1,1,1}$ & $\scriptstyle{2^{-10}3^{-6}\psi^{-12}}$\\
    
    $\clubsuit_1$ &  $\scriptstyle{\frac{1}{12},\frac{5}{12},\frac{7}{12}, \frac{11}{12}}$ & $\scriptstyle{\frac{1}{4}, \frac{1}{2}, \frac{3}{4},1}$& $\scriptstyle{2^{-10}3^{-6}\psi^{-12}}$\\
    $\clubsuit_2$ & $\scriptstyle{\frac{1}{12},\frac{1}{6},\frac{5}{12},\frac{5}{6}}$ & $\scriptstyle{\frac{1}{8},\frac{5}{8},\frac{1}{4},1}$ & $\scriptstyle{-2^{-10}3^{-6}\psi^{-12}}$\\
    $\clubsuit_3$ & $\scriptstyle{\frac{7}{12},\frac{1}{6},\frac{11}{12},\frac{5}{6}}$ & $\scriptstyle{\frac{3}{8},\frac{7}{8},\frac{3}{4},1}$& $\scriptstyle{-2^{-10}3^{-6}\psi^{-12}}$\\
    $\clubsuit_4$ & $\scriptstyle{\frac{1}{24},\frac{5}{24},\frac{7}{24},\frac{11}{24},\frac{13}{24},\frac{17}{24},\frac{19}{24},\frac{23}{24}}$ &
    $\scriptstyle{\frac{1}{6},\frac{1}{4},\frac{1}{3},\frac{1}{2},\frac{2}{3},\frac{3}{4},\frac{5}{6},1}$ & $\scriptstyle{2^{-10}3^{-6}\psi^{-12}}$\\
    \hline 
      
    $\symking_0$& $\scriptstyle{\frac{1}{6},\frac{1}{3},\frac{2}{3},\frac{5}{6}}$& $\scriptstyle{\frac{1}{2},1,1,1}$ &$\scriptstyle{2^{-4}\psi^{-6}}$\\
    $\symking_1$     & $\scriptstyle{\frac{1}{12},\frac{5}{12},\frac{7}{12},\frac{11}{12}}$ & $\scriptstyle{\frac{1}{4},\frac{1}{2},\frac{3}{4},1}$ & $\scriptstyle{2^{-4}\psi^{-6}}$\\
    $\symking_2$     & $\scriptstyle{\frac{5}{6},\frac{1}{3}}$ & $\scriptstyle{\frac{2}{3},1}$ &$\scriptstyle{2^{-4}\psi^{-6}}$\\
    $\symking_3$& $\scriptstyle{\frac{1}{6},\frac{2}{3}}$ & $\scriptstyle{\frac{1}{3},1}$ &$\scriptstyle{2^{-4}\psi^{-6}}$\\
    $\symking_4$     & $\scriptstyle{\frac{5}{12},\frac{11}{12}}$ & $\scriptstyle{\frac{5}{6},1}$&$\scriptstyle{2^{-4}\psi^{-6}}$\\ 
    $\symking_5$& $\scriptstyle{\frac{1}{12},\frac{7}{12}}$ & $\scriptstyle{\frac{1}{6},1}$ &$\scriptstyle{2^{-4}\psi^{-6}}$\\
    \hline  
     $\spadesuit_0$& $\scriptstyle{\frac{1}{36},\frac{1}{18},\frac{5}{36},\frac{7}{36},\frac{5}{18},\frac{11}{36},\frac{13}{36},\frac{7}{18},\frac{17}{36},\frac{19}{36},\frac{11}{18},\frac{23}{36},\frac{25}{36},\frac{13}{18},\frac{29}{36},\frac{31}{36},\frac{17}{18},\frac{35}{36}}$ & $\scriptstyle{\frac{1}{8},\frac{1}{4},\frac{3}{8},\frac{1}{2},\frac{5}{8},\frac{3}{4},\frac{7}{8},\frac{1}{7},\frac{2}{7},\frac{3}{7},\frac{4}{7},\frac{5}{7},\frac{6}{7},\frac{1}{3},\frac{2}{3},1,1,1}$&$\scriptstyle{2^{-48}3^{-30}7^{-7}\psi^{-36}}$\\
    \hline 
    $\heartsuit_0$ & $\scriptstyle{\frac{1}{27},\frac{2}{27},\frac{4}{27},\frac{5}{27},\frac{7}{27},\frac{8}{27},\frac{10}{27},\frac{11}{27},\frac{13}{27},\frac{14}{27},\frac{16}{27},\frac{17}{27},\frac{19}{27},\frac{20}{27},\frac{22}{27},\frac{23}{27},\frac{25}{27},\frac{26}{27}}$ & $\scriptstyle{\frac{1}{3},\frac{2}{3},\frac{1}{2},\frac{1}{5},\frac{2}{5},\frac{3}{5},\frac{4}{5},\frac{1}{6},\frac{5}{6},\frac{1}{7},\frac{2}{7},\frac{3}{7},\frac{4}{7},\frac{5}{7},\frac{6}{7},1,1,1}$ & $\scriptstyle{2^{-6}3^{-24}5^{-5}7^{-7}\psi^{-27}}$\\
    \hline
    \end{tblr}
    \label{table: alphaBetaSubscripts}
\end{table}

We develop multiple new techniques in order to address challenges that arise in our analysis. The first challenge we faced in this work appears when dealing with the point counts of the two families $\Csf_2\Fsf_2$ and $\Csf_2\Csf_2.$ For $\Csf_2\Fsf_2,$ for example, if $q$ is a power of a prime satisfying $q \equiv 1 \Mod{4},$ then some summands (e.g. Prop \ref{prop: propc2f2}(c)) that naturally appear in the point counts are not defined over $\Q$ (in the sense of \cite{BCM15}). In \cite{HD20}, a similar situation led the authors to introduce a new definition of hypergeometric functions over finite fields, called \textit{splittable}. 

Here we must go one step further, as even the splittable hypergeometrics do not capture some of the subtleties of our computations. Our approach builds on the work of Asem Abdelraouf in \cite{AA25} and his newly defined finite hypergeometric function associated to a \textit{gamma triple} of parameters $(\bm{\gamma},\bm{\delta},N) \in \Z^d\times \Z^d \times \Z_{>0}.$ Abdelraouf's new definition not only generalizes \cite{BCM15}, but also expresses exactly the type of sums that we were unable to connect to the previously known finite hypergeometric functions. Thanks to the new approach, we define $L$-functions for \textit{gamma triples}. 

We note that $\Csf_2\Csf_2,$ one of the the cases where we require the new definition of a finite field hypergeometric function, is also among the families where \cite{KM17} was not able to give a full list of generators for algebraic curves. In this sense, our need for $L$-functions not defined over $\Q$ detects the fact that the K3 surfaces in question have algebraic curves defined over more elaborate number fields.

With our new framework, we can factorize the $L$-functions of our K3 families as shown in our main theorem. We use the notation for hypergeometric parameters and the parameter of our pencil established in Table~\ref{table: alphaBetaSubscripts}.

\begin{mainthm*}\label{T:main}
Let $S=S(\bullet,\psi)$ denote the set of bad primes together with the primes dividing the numerator or denominator of $t$ or $t-1,$ where $\bullet$ denotes one of our pencils. The $L$-functions of the Delsarte pencils of K3 surfaces $\Csf_2\Fsf_2$, $\Csf_2\Lsf_2$, $\Csf_2\Csf_2$, $\Csf_3\Fsf_1$, and $\Csf_4$ decompose as products of hypergeometric and Dirichlet $L$-functions in the following way. 

\begin{itemize}[label={}]
\item \begin{align*}
&L_S(X_{\Csf_4,\psi},s)=\zeta_{S,\Q(i\sqrt{3})}(s-1)\zeta_{S}(s-1)\cdot L_S\left(H(\bm{\alpha}_{\heartsuit_0},\bm{\beta}_{\heartsuit_0} \mid t_{\heartsuit_0}),s\right).
\end{align*}
\begin{align*}
L_S(X_{\Csf_2\Fsf_2,\psi},s)&=L_S(\Q(\zeta_8)|\Q,s)\cdot L_S(H(\bm{\alpha}_{\clubsuit_0},\bm{\beta}_{\clubsuit_0}|t_{\clubsuit_0}),s)\cdot L_S(H(\bm{\alpha}_{\clubsuit_1},\bm{\beta}_{\clubsuit_1}|t_{\clubsuit_1}),s-1,\phi_{-1}) \\&\cdot L_S(F((4,2,3,3,-12,1,-1),(0,-1,0,1,0,0,0),4 \mid t_{\clubsuit_2}),\Q(i),s-1,\phi_{\sqrt{-1}}). 
\end{align*}
\begin{align*}
L_S(X_{\Csf_3\Fsf_1,\psi},s)&=L_{S}(\Q(\zeta_8)\mid \Q,s-1)\cdot L_S\left(H(\bm{\alpha}_{\spadesuit_0},\bm{\beta}_{\spadesuit_0}\mid t_{\spadesuit_0}),s\right). 
\end{align*}
\begin{align*}
L_S(X_{\Csf_2\Lsf_2,\psi},s)&=\zeta_{S,\Q(i\sqrt{3})}(s-1)^2\zeta_{S,\Q(i)}(s-1)\zeta_S(s-1)\cdot L_S\left(H(\bm{\alpha}_{\clubsuit_0},\bm{\beta}_{\clubsuit_0}\mid t_{\clubsuit_0}),s\right)\\&\cdot L_{S}\left(H\left(\bm{\alpha}_{\clubsuit_4},\bm{\beta}_{\clubsuit_4}\mid t_{\clubsuit_4}\right),s-1,\phi_{-12}\phi_{\psi}\right). 
\end{align*}
\begin{align*}
L_S(X_{\Csf_2\Csf_2,\psi},s)=\zeta_S(s-1) &\zeta_{S,\Q(i\sqrt{3})}(s-1)^2 \cdot  L_S(H(\bm{\alpha}_{\symking_0},\bm{\beta}_{\symking_0}\mid t_{\symking_0}),s)\cdot L_S(H(\bm{\alpha}_{\symking_1},\bm{\beta}_{\symking_1}\mid t_{\symking_1}),s-1,\phi_{-6\psi}) \\&\cdot L_S(F((2,1,2,1,-6),(0,0,1,-1,0),3\mid t_{\symking_3}),\Q(\sqrt{-3}),s-1)\\&\cdot L_S(F((2,1,2,1,-6),(3,0,-1,1,3),6\mid t_{\symking_3}),\Q(\zeta_6),s-1,\phi_{-6\psi}).    
\end{align*}
\end{itemize}
\end{mainthm*}

In each case, we have one hypergeometric factor associated to the holomorphic form. The appearance of additional hypergeometric and Dirichlet $L$-function factors aligns with the decomposition of $H^{2}(X_\psi)$ identified by Kloosterman. 

The hypergeometric parameters are consistent with the parameters arising in hypergeometric Picard-Fuchs equations satisfied by elements of $H^2(X_\psi, \C)$.  
When computing the Picard-Fuchs equations, we face another challenge that was not present in \cite{HD20}. The Griffiths-Dwork diagrammatic method used by those authors to compute the Picard-Fuchs equations is not as effective in our situation. In order to address this challenge, we use the recent work of Adolphson and Sperber in \cite{AS23} to obtain the hypergeometric parameters associated with the Picard-Fuchs differential operators.

The plan of the paper is as follows. 
In \S~\ref{S:invertible}, after a brief discussion of the properties of K3 surfaces, we review the classification of Delsarte and invertible polynomials, explaining our choice of deformation constant $d^T$ and identifying groups of symmetries for each pencil. In \S~\ref{S:PFequations}, we apply the Adolphson--Sperber method for computing hypergeometric differential operators to the Picard-Fuchs equations of K3 surfaces in $\PP^3$, and compute the resulting collections of Picard-Fuchs equations in our cases of interest. In \S~\ref{S:finiteHyp}, we discuss finite field hypergeometric functions and the associated hypergeometric $L$-functions; in particular, we show how to define hypergeometric $L$-functions for collections of parameters that may not be defined over $\Q$, building on work of \cite{AA25}. In \S~\ref{S:numPoints}, we obtain explicit formulas for the point counts. We validate each of our formulas computationally. In \S~\ref{S:Lseries}, we first prove that the hypergeometric parameters in our finite field formulas match the parameters of the hypergeometric Picard-Fuchs equations (see Theorem~\ref{T:hypergeometricMatch}). We then prove our \hyperref[T:main]{Main Theorem} by assembling the point count formulas into $L$-series.

\section{Invertible pencils and Delsarte K3 surfaces}\label{S:invertible}

In the present work, we study a collection of K3 surface pencils that arise naturally in the setting of invertible polynomials. We recall some background about K3 surfaces and establish notation for the examples of interest.

Let $k$ be a field. A K3 surface over $k$ is a complete non-singular variety $X$ of dimension two such that $\omega_X=\Omega^2_{X/k} \simeq \mathcal{O}_X$ and $H^1(X,\mathcal{O}_X)=0.$ Over $\C$, the Hodge diamond of a K3 surface has the following form: \begin{center}\begin{tabular}{ccccc}
 &  & $1$ &  &  \\
 & $0$ &  & $0$ &  \\
$1$ &  & $20$ &  & $1$ \\
 & $0$ &  & $0$ &  \\
 &  & $1$ &  &  \\
\end{tabular}
\end{center}

Smooth quartic hypersurfaces in $\mathbb{P}_k^3$ are K3 surfaces. 

Let $A=(a_{ij})_{0 \leq i,j \leq n}$ be an $(n+1)\times (n+1)$ matrix with integer coefficients. We say that the polynomial $$F_A=\sum_{i=0}^n\prod_{j=0}^n x_j^{a_{ij}} \in \Z[x_0,\dots,x_n]$$ is \emph{invertible} when the following conditions are satisfied:

\begin{enumerate}
    \item $\det(A) \neq 0.$
    \item There exist positive integers called \emph{weights} $q_j$ satisfying the condition that $d : = \sum_{j=0}^n q_j a_{ij}$ is the same constant for all $i$.
    \item $F_A: k^{n+1} \to k$ has a unique critical point at the origin. 
\end{enumerate}
 
We say an invertible polynomial is \emph{Delsarte} when the weights are all 1. A Delsarte polynomial defines a smooth hypersurface in $\PP^n$; more generally, any invertible polynomial defines a quasismooth hypersurface in $W\mathbb{P}^n(q_0,\dots,q_n)$. Note that a Delsarte polynomial defined by an $(n+1)\times(n+1)$ matrix gives a smooth degree $n+1$ hypersurface in $\mathbb{P}^n$, that is, a Calabi-Yau hypersurface. 

The \textit{transposed polynomial} of $F_A$ is the polynomial determined by $A^T.$ Since $A^T$ has nonnegative entries, $F_{A^T}$ is well defined. The transposed polynomial is a quasihomogeneous polynomial that defines a hypersurface in $W\mathbb{P}^n(q_0,\dots,q_n).$ The weights $q_0,\dots,q_n$ are called the \textit{dual weights} of $F_A$ and we denote their sum by $d^T.$ One of the reasons for interest in invertible polynomials and their transposes is that they are essential ingredients for Berglund--H\"ubsch--Krawitz (BHK) mirror symmetry; see \cite{K08} and \cite{C18} for further discussion and examples, and \cite{ZF18} for an application of these ideas to zeta functions. 

Let $F_A \in \Z[x_0,\dots,x_n]$ be a Delsarte polynomial. The \textit{invertible pencil} associated to it is the pencil of hypersurfaces \begin{align}\label{eq: pencil}
X_{A,\psi}=V(F_A- d^T \psi x_0 \cdots x_n) \subset \mathbb{P}_k^n.    
\end{align}

Our choice of constant $d^T$ in the deformation term follows the convention of previous work in the math and physics literature on quintic Calabi-Yau threefolds (cf. \cite{CX00}). 

Invertible polynomials and the corresponding varieties admit several interesting abelian symmetry groups. The group
$$\text{Aut}(F_A)=\{(\lambda_0,\dots,\lambda_n) \in (\C^{\times})^n:F_A((\lambda_0x_0,\dots,\lambda_nx_n))=F_A(x_0,\dots,x_n)\}$$ is called the group of \textit{diagonal automorphisms} of $F_A.$ Its subgroup satisfying $\lambda_0\cdots \lambda_n=1$ will be denoted by $\text{SL}(F_A).$ Finally, the subgroup of $\text{SL}(F_A)$ generated by $$\rho=(\exp(\tfrac{2\pi i}{n+1}),\dots, \exp(\tfrac{2\pi i}{n+1}))$$ will be denoted by $J(F_A)$ and called the \textit{trivial diagonal symmetries.} Since the elements of $\text{Aut}(F_A)$ that preserve the pencil $X_{A,\psi}$ are the elements of $\text{SL}(F_A),$ we define the \textit{symmetry group} $G_{F_A}$ of $X_{A,\psi}$ by $\text{SL}(F_A)/J(F_A)$. One can compute the automorphism group of each polynomial using \cite[Proposition 2]{ABS11}.

There is an important result concerning the classification of invertible polynomials over $\mathbb{C}$, which extends to a classification of their corresponding invertible pencils:

\begin{theorem}[\cite{KS92}]\label{T: invertibleKS}
Any invertible polynomial can be written up to permutation of the variables as a disjoint sum of the following atomic types: 
\begin{enumerate}
    \item Fermat: $x_0^{a_1}+x_1^{a_2}+\cdots+x_k^{a_k}$
    \item Loop: $x_0^{a_0}x_1+x_1^{a_1}x_2+\cdots+x_{k-1}^{a_{k-1}}x_k+x_k^{a_k}x_0$
    \item Chain: $x_0^{a_0}x_1+x_1^{a_1}x_2+\cdots+x_{k-1}^{a_{k-1}}x_k+x_k^{a_k}$
\end{enumerate}
\end{theorem}

In the Delsarte case, where the degree of each monomial is clear, we use the notational convention that $\Fsf_k$ represents a Fermat of length $k$, $\Lsf_k$ represents a loop of length $k$, and $\Csf_k$ represents a chain of length $k$. For example, $\Lsf_3\Fsf_1$ corresponds to the invertible polynomial $x_0^3x_1+x_1^3x_2+x_2^3x_0+x_3^4$ and, by abuse of notation, to the corresponding invertible pencil. For the case of Delsarte quartics, the classification in Theorem~\ref{T: invertibleKS} gives us the 10 possibilities presented in Table~\ref{table: classif}.

\begin{table}[ht]
    \centering
    \caption{Delsarte quartics}
    \begin{tblr}{colspec ={|c|c|c|c|c|},cell{2}{1} = {r=5}{c},cell{7}{1} = {r=2}{c}}
    \hline
Suit   & Polynomial & Label & Dual weights & Symmetries of the pencil\\
    \hline 
    \hline 
$\diamondsuit$ & $x_0^4+x_1^4+x_2^4+x_3^4$ & $\Fsf_4$ & 1, 1, 1, 1 & $(\Z/4\Z)^2$\\
& $x_0^3x_1+x_1^3x_2+x_2^3x_0+x_3^4$ & $\Lsf_3\Fsf_1$ & 1, 1, 1, 1& $\Z/7\Z$\\
& $x_0^3x_1+x_1^3x_0+x_2^4+x_3^4$ & $\Lsf_2\Fsf_2$ & 1, 1, 1, 1 &$\Z/8\Z$\\ 
& $x_0^3x_1+x_1^3x_0+x_2^3x_3+x_3^3x_2$& $\Lsf_2\Lsf_2$ & 1, 1, 1, 1 &$(\Z/4\Z) \times (\Z/2\Z)$\\ 
& $x_0^3x_1+x_1^3x_2+x_2^3x_3+x_3^3x_0$& $\Lsf_4$ & 1, 1, 1, 1 & $\Z/5\Z$\\ 
    \hline 
$\clubsuit$ &     $x_0^3x_1+x_1^4+x_2^4+x_3^4$ & $\Csf_2\Fsf_2$ & 4, 2, 3, 3 & $\Z/4\Z$\\
&     $x_0^3x_1+x_1^4+x_2^3x_3+x_3^3x_2$ & $\Csf_2\Lsf_2$ & 4, 2, 3, 3 & $\Z/2\Z$\\
    \hline     
$\symking$ &   $x_0^3x_1+x_1^4+x_2^3x_3+x_3^4$&$\Csf_2\Csf_2$ & 2, 1, 2, 1 & $\Z/6\Z$ \\
    \hline 
$\spadesuit$  & $x_0^3x_1+x_1^3x_2+x_2^4+x_3^4$& $\Csf_3\Fsf_1$ & 12, 8, 7, 9&\text{trivial}\\ 
    \hline 
 $\heartsuit$   & $x_0^3x_1+x_1^3x_2+x_2^3x_3+x_3^4$ & $\Csf_4$ & 9, 6, 7, 5 &\text{trivial} \\
    \hline
    \end{tblr}
    \label{table: classif}
\end{table}

Kloosterman studied Delsarte pencils in \cite{KM17} and identified several interesting subspaces of cohomology. We now recall this construction in the cases of interest to us. Let $X_{\psi}$ be a Delsarte pencil with defining matrix $A$, and $d_{\Fsf}$ the smallest integer such that $B:=d_{\Fsf}A^{-1}$ has integer entries and define $b=(1,\dots,1)^TB.$ Let $$G=\bigg\{(g_0,\dots,g_n) \in (\Z/d_{\Fsf} \Z)^n : \sum_{i=1}^ng_ib_i=0\bigg\} \quad \text{ and } \quad Y_{\psi}: \sum_{i=0}^ny_i^d-d^T\psi y_0 \cdots y_n=0.$$ 
 
Then $X_{\psi}$ is birational to $Y_{\psi}/G$ via the rational map induced by \begin{align*}
\mathbb{P}^n &\to W\mathbb{P}^n(w_0,\dots,w_n) \\
 (y_0: \cdots: y_n)  &\mapsto \left(\prod_{i=1}^ny_i^{b_{0,i}}: \cdots :\prod_{i=1}^ny_i^{b_{n,i}}\right).
\end{align*} 

Suppose that we are working over a finite field $\F_q$ such that $q \equiv 1 \Mod{d_{\Fsf}}.$ Kloosterman proved in \cite{KM17} that $H^{n-1}(X_\psi)=H^{n-1}(Y_\psi)^{G_{\text{tor}}} \oplus W_\psi \oplus C,$ where $G_{\text{tor}}$ is the maximal torus group of symmetries on the Fermat pencil $Y_\psi,$ $W_{\psi}$ and $C$ are Frobenius-stable subspaces and $H^{n-1}(Y_\psi)^{G_{\text{tor}}} \oplus W_\psi = H^{n-1}(Y_\psi)^G.$ Using his previous work in (\cite{kloostermanFermat}), Kloosterman observed that $H^{n-1}(Y_\psi)^{G_{\text{tor}}}$ and $W_\psi$ are associated to hypergeometric factors of the zeta function of $X_\psi$ when appropriate divisibility conditions (which we relax in the present work) are satisfied. Precisely, $H^{n-1}(Y_\psi)^{G_{\text{tor}}}$ corresponds to the holomorphic piece and $W_\psi$ to the other hypergeometric pieces.

Table \ref{table: kloost} shows the dimensions of the Frobenius-stable subspaces for our five Delsarte pencils. The column ``PF order" refers to the dimension of the subspace spanned by the holomorphic form and its derivatives with respect to the parameter $\psi$, and, thus, to the order of the Picard-Fuchs equation associated to the holomorphic form. Note that in each case the sum of the last three columns is 21, the dimension of primitive cohomology for a quartic in $\PP^3$.

\begin{table}[ht!]
    \centering
    \caption{Frobenius-stable subspaces}
    \begin{tblr}{colspec ={|c|c|c|c|c|c|c|},cell{2}{1} = {r=2}{c}}
    \hline
Suit   & Label & PF Order & $\dim W_{\psi}$ & $\dim C$\\
    \hline 
    \hline 
$\clubsuit$ & $\Csf_2\Fsf_2$ & 6 & 12 & 3\\
& $\Csf_2\Lsf_2$ & 6 & 8 & 7\\
    \hline     
$\symking$ & $\Csf_2\Csf_2$ & 4 & 12 & 5\\
    \hline 
$\spadesuit$  & $\Csf_3\Fsf_1$ & 18 & 0 & 3\\ 
    \hline 
 $\heartsuit$   & $\Csf_4$ & 18 & 0 & 3\\
    \hline
    \end{tblr}
    \label{table: kloost}
\end{table}

In our \hyperref[T:main]{Main Theorem}, we obtain extra hypergeometric factors of the $L$-function precisely when $W_{\psi}$ is nontrivial, and we obtain Dirichlet factors precisely when $C$ is nontrivial.

\section{Picard-Fuchs equations}\label{S:PFequations}

The goal of this section is to identify hypergeometric periods associated with the five Delsarte families. We give a brief overview of the notions of periods, Picard-Fuchs equations, hypergeometric differential operators and hypergeometric series and then apply computational results of \cite{G11} and \cite{AS23}.

\subsection{Periods and Picard-Fuchs equations}
Let $X$ be a smooth hypersurface in $\mathbb{P}^n$ defined by the zeros of the degree $d$ homogeneous polynomial $F \in \C[x_0,\dots,x_n].$ We denote by $A^n(X)$ the space of rational $n$-forms in $\mathbb{P}^n$ with poles along $X$ and say that $X$ is the \textit{polar locus} of such forms. Alternatively, we can define $A^n(X)$ as the regular $n$-forms in $\mathbb{P}^n \setminus X.$ By \cite[Corollary 2.11]{G69}, every $\omega \in A^n(X)$ can be written as $$\omega=x_0\cdots x_n\frac{Q(x_0,\dots,x_n)}{F(x_0,\dots,x_n)^k} \Omega_1,$$ where $k \geq 0,$ $Q \in \C[x_0,\dots,x_n]$ is homogeneous of degree $k\deg(F)-(n+1)$ and $$\Omega_1=\frac{1}{x_0\cdots x_n}\sum_{i=0}^n(-1)^ix_i \text{d}x_0 \wedge \cdots \wedge \text{d}x_{i-1} \wedge \text{d}x_{i+1} \wedge \cdots \wedge \text{d}x_n.$$

Let $[\mathbb{P}^{n-1} \cdot X]$ denote the homology class of a hyperplane section of $X$ and let $\eta \in H^{2}(X,\C)$ denote its Poincaré dual. We define \begin{align*}
H^{n-1}_{\text{prim}}(X,\C)=\{[\sigma] \in H^{n-1}(X,\C) : \sigma \cdot \eta =0\},
\end{align*}
which can be geometrically interpreted as the classes in $H^{n-1}(X,\C)$ that are orthogonal to the hyperplane class. An analogous definition holds for $H^{n-k,k-1}_{\text{prim}}(X).$ 

The algebraic de Rham cohomology group $\mathcal{H}^n(X)$ is defined by $$\mathcal{H}^n(X)=\frac{A^n(X)}{\text{d}A^{n-1}(X)}.$$ Given any $\omega \in \mathcal{H}^n(X),$ there is a unique class $\text{Res}(\omega) \in H^{n-1}(X,\C)$ such that for $\gamma$ an $(n-1)$-cycle and $T(\gamma)$ a tube around $\gamma,$ we have $$\frac{1}{2\pi i}\int_{T(\gamma)}\omega=\int_{\gamma} \text{Res}(\omega).$$ This defines the \textit{residue map} $$\text{Res}: \mathcal{H}^n(X) \to H^{n-1}(X,\C).$$ For example, the holomorphic form $\omega \in H^{n-1,0}(X)$ satisfies $$\text{Res}\left(\frac{x_0\cdots x_n\Omega_1}{F_{\psi}}\right)=\omega.$$

Thanks to a result of Griffiths, the residue map can be used to compute the primitive cohomology of a hypersurface as follows
\begin{prop}\cite[(8.6)]{G69} 
The residue map induces an injective ring homomorphism whose image is the primitive cohomology $H^{n-k,k-1}_{\text{prim}}(X)$ \begin{align*}
\left(\frac{\C[x_0,\dots,x_n]}{J(F)}\right)_{k\deg(F)-(n+1)} \to H^{n-k,k-1}(X).    
\end{align*} 
\end{prop}

\begin{ex}

If $X$ is a quartic K3 surface in $\mathbb{P}^3,$ one has \begin{align}\label{eq:eqiso}
H^{2,0}_{\text{prim}}(X)=\left(\frac{\C[x_0,\dots,x_n]}{J(F)}\right)_{0}=\C, \quad 
H^{1,1}_{\text{prim}}(X)=\left(\frac{\C[x_0,\dots,x_n]}{J(F)}\right)_{4},\quad
H^{0,2}_{\text{prim}}(X)=\left(\frac{\C[x_0,\dots,x_n]}{J(F)}\right)_{8}=\C.
\end{align}
    
\end{ex}

Let $X_{\psi}$ be a pencil of hypersurfaces in $\mathbb{P}^n$ in the parameter $\psi,$ defined by homogeneous polynomials $F_{\psi} \in \C[x_0,\dots,x_n].$ Fix bases $\gamma_j \in H_{n-1}(X_{\psi},\C), j=1,\dots,\dim_{\C}H_{n-1}(X_{\psi},\C)$ and $$\Omega_{\psi,i} \in H^{n-1}(X_{\psi},\C), \quad i=1,\dots,\dim_{\C}H^{n-1}(X_{\psi},\C)$$ in such a way that it also provides a basis for the Hodge decomposition of $H^{n-1}(X_{\psi},\C).$ 

Define the \textit{period integrals} as $$\int_{\gamma_j}\Omega_{\psi,i},$$ for $i,j=1,\dots,\dim_{\C}H^{n-1}(X_{\psi},\C).$ 

For each pair $i,j,$ using the residue map defined above, we find $k \in \mathbb{N}$ and $Q_i \in \C[x_0,\dots,x_n]_{k\deg(F)-(n+1)}$ such that $$\int_{\gamma_j}\Omega_{\psi,i}=\int_{T(\gamma_j)}\frac{x_0\cdots x_n Q_i}{F_{\psi}^k}\Omega_1,$$ where $T(\gamma_j)$ is a tube around $\gamma_j.$ Our goal is to understand how these integrals vary along the family, so we differentiate them with respect to $\psi$ \begin{align*}
\frac{\text{d}}{\text{d}\psi}\int_{T(\gamma_j)}\frac{x_0 \cdots x_n Q_i}{F_{\psi}^k}\Omega_1=-k\int_{T(\gamma_j)}\frac{x_0\cdots x_n Q_i}{F_{\psi}^{k+1}}\frac{\text{d}F}{\text{d}\psi}\Omega_1,   
\end{align*} which originates a new $(n-1)$-form. Since $H^{n-1}(X_{\psi},\C)$ is finite dimensional, if we differentiate the period at most $\dim_{\C}H^{n-1}(X_{\psi},\C)$ times with respect to $\psi,$ we will obtain a linear relation between the $(n-1)$-forms that arise from the process. This linear relation is called the \emph{Picard-Fuchs equation} of the period $\int_{\gamma_j}\Omega_{\psi,i}.$

We will use the following notation. For a fixed cycle $\gamma$, consider the period integrals \begin{align}\label{eq:eqeq}
(v_0,\dots,v_n):=\int_{T(\gamma)}\frac{x^{\bm{v}}}{F_{\psi}^{k(\bm{v})}}\Omega_1,  
\end{align} where $x^v:=x_0^{v_0}\cdots x_n^{v_n}$ and $k(\bm{v}):=\frac{1}{n+1}\sum_{i=0}^nv_i.$ For example, the holomorphic period is given by $$(1,\dots,1)=\int_{T(\gamma)}\frac{x_0 \cdots x_n}{F_{\psi}^{}}\Omega_1.$$ 

\subsection{Hypergeometric series and periods of Delsarte pencils}

In many cases of geometric and arithmetic interest, the Picard-Fuchs equations associated to a pencil are \emph{hypergeometric}, a special type of ordinary differential equation with solutions given by \emph{hypergeometric series}. When this happens, the periods can be written in terms of such series, and one can study them explicitly. As we observe in \S ~\ref{subs: gahrs}, for Delsarte pencils of K3 surfaces and Calabi-Yau varieties, the Picard-Fuchs equations associated to the holomorphic form are always hypergeometric. We describe a method for identifying other hypergeometric Picard-Fuchs equations in \S ~\ref{subs: ASmethod}.

We begin this section by defining hypergeometric differential operators and their formal solutions, called hypergeometric series. 

Let $d$ be a natural number and $\bm{\alpha}=\{\alpha_1,\dots,\alpha_d\},\bm{\beta}=\{\beta_1,\dots,\beta_d\}$ two multisets of rational numbers. We define the \emph{hypergeometric series} \begin{equation}\label{eq:eqhyp}
    _dF_{d-1}(\bm{\alpha},\bm{\beta} \mid x)=\sum_{n=0}^{\infty}\dfrac{(\alpha_1)_n\cdots(\alpha_d)_n}{(\beta_1)_n\cdots(\beta_d)_n}x^n.
\end{equation}
Let $\Theta=x\tfrac{\text{d}}{\text{d}x}$. We also define the \emph{hypergeometric differential operator } $$D(\bm{\alpha},\bm{\beta} \mid x)=(\Theta+\beta_1-1)\cdots (\Theta+\beta_d-1)-x(\Theta+\alpha_1)\cdots (\Theta+\alpha_d).$$ Moreover, a differential equation will be called \textit{hypergeometric} if it is of the form $D(\bm{\alpha},\bm{\beta} \mid x) \cdot f=0,$ where $f: \C \to \C.$  

If $\beta_j$ are distinct elements in $(0,1]$ and $\beta_1=1,$ then a solution for the hypergeometric differential operator $D(\bm{\alpha},\bm{\beta} \mid x)$ at the origin is given by $_dF_{d-1}(\bm{\alpha},\bm{\beta} \mid x).$ The other solutions are $$x^{1-\beta_j}{_dF_{d-1}}(\alpha_1+1-\beta_j,\dots, \alpha_d+1-\beta_j ; \beta_1+1-\beta_j,\dots,\beta_d+1-\beta_j \mid x),$$ for $j=2,\dots,d.$ 

\subsubsection{The Picard-Fuchs equation of the holomorphic form}\label{subs: gahrs}
Thanks to the following result of G\"ahrs, we know that the Picard-Fuchs equation associated to the holomorphic period of any invertible pencil is hypergeometric, and we can determine it only by knowing the dual weights of the family.

\begin{theorem}[\cite{G11}]\label{thm: gahrs}
Let $X_{A,\psi}$ be an invertible pencil of hypersurfaces in $\mathbb{P}^n$ with dual weights $q_0,\dots,q_n.$ Then let $$\bm{\alpha}=\left\{\frac{j}{d^T}: j=0,\dots,d^T-1\right\} \text{ and } \bm{\beta}=\left\{\frac{j}{q_i}: i=0,\dots,n; j=0,\dots,q_i-1\right\}.$$
Then the Picard-Fuchs equation associated to the holomorphic forms of this family is hypergeometric as follows \begin{align*}
\prod_{\beta_{ij}\in \bm{\beta} \setminus I}(\Theta-\beta_{ij})-z \prod_{\alpha_j \in \bm{\alpha}\setminus I}(\Theta+\alpha_j)=0,   
\end{align*} where $z:=\prod_{i}q_i^{-q_i}\psi^{-d^T},$ $\Theta=-(d^T)\frac{\text{d}}{\text{d}\psi}$ and $I:=\bm{\alpha}\cap \bm{\beta}.$
\end{theorem} 
Notice that $\bm{\alpha}$ does not have repetitions, so the intersection of hypergeometric parameters invoked in Theorem~\ref{thm: gahrs} is well defined.

Although Theorem~\ref{thm: gahrs} gives an explicit formula for the Picard-Fuchs equation associated to the holomorphic form, it does not say anything about the periods that do not satisfy that ODE. However, recently in \cite{AS23}, Adolphson and Sperber developed a new method to compute hypergeometric Picard-Fuchs equations associated with a monomial basis. This method gives an explicit description of the matrix of fundamental solutions associated with a monomial basis and consequently yields information about the remaining hypergeometric periods.

\subsubsection{The remaining periods and the Adolphson-Sperber method}\label{subs: ASmethod}
Let $\mathcal{A}=\{\bm{a}_j\}_{j=1}^m \subseteq \Z^n$ be a linearly independent set in $\R^n$ and suppose that there is a point $\bm{a}_0 \in \Z^n$ in the interior of the convex hull of $\mathcal{A}$. Equivalently, we may suppose $$\ell_0\bm{a}_0=\sum_{j=1}^m \ell_j\bm{a}_j,$$ where $\ell_0,\dots,\ell_m$ are coprime integers such that $\ell_0=\sum_{j=0}^n \ell_j.$ Let \begin{align}\label{eq: flambda}
f_{\lambda}=\sum_{j=1}^m \ell_jX^{\bm{a}_j}-\ell_0\lambda X^{\bm{a}_0} \in \Z[X_0^{\pm },\dots,X_n^{\pm}],   
\end{align} where $X^{\bm{a}_j}=X_0^{a_{0j}}\cdots X_n^{a_{nj}}$ and $\lambda$ is a parameter.

Let $V \subset \R^n$ be the subspace of dimension $m$ generated by $\mathcal{A}$ and $V_{\Z}=V \cap \Z^n.$ Let $C(\mathcal{A})$ be the real cone generated by $\mathcal{A}$ and $M=V_{\Z}\cap C(\mathcal{A}).$ Consider also $S \subset \C[\lambda][X_0^{\pm},\dots,X_n^{\pm}]$ be the $\C[\lambda]$-algebra generated by the monomials $\{X^u: u \in M\}.$

\begin{remark}  
As mentioned in \cite[Equation (1.3)]{AS23}, for invertible polynomials the $\ell_j,$ $j=1,\dots,m$ are the dual weights and $\ell_0=d^T.$ 
\end{remark}

Consider the half-open parallelepiped defined by $\mathcal{A}$, \begin{align*}
P(\mathcal{A})=\left\{\sum_{j=1}^m c_j\bm{a}_j \mid 0 \leq c_j < 1 \text{ for } j=1,\dots,m\right\}.    
\end{align*}

Let $\mathcal{B}:=V_{\Z}\cap P(\mathcal{A})$.

\begin{prop}
\label{hypergeometric}
Let $\bm{b} \in \mathcal{B}$ and write $\bm{b}=\sum_{j=1}^m v_j\bm{a}_j,$ with $v_j \in [0,1) \cap \Q$ for $j=1,\dots,m.$ The $(\bm{b},\bm{b})$ entry of the fundamental solution matrix is the hypergeometric series \begin{align}\label{eq: eqAS}
\sum_{s=0}^{\infty}\frac{\left(\frac{v_j}{\ell_j}\right)_s\left(\frac{v_j+1}{\ell_j}\right)_s\cdots \left(\frac{v_j+\ell_j-1}{\ell_j}\right)_s}{\left(\frac{1}{\ell_0}\right)_s\left(\frac{2}{\ell_0}\right)_s\cdots \left(\frac{\ell_0}{\ell_0}\right)_s}\lambda^{s \ell_0}.    
\end{align}    
\end{prop}

\begin{remark}\label{rem: cov}
The solutions given in Proposition \ref{hypergeometric} are solutions at the origin. To obtain the hypergeometric parameters that match our finite field computations, we perform a change of variables (\cite[\S 10]{AS23}) that transforms solutions at the origin into solutions at infinity. We explain it below for general multisets $\bm{\alpha},\bm{\beta}$ such that $|\bm{\alpha}|=|\bm{\beta}|=d$:

\begin{enumerate}
    \item Choose $\alpha_1$ to be the smallest element of $\bm{\alpha}.$ 
    \item Define $\tilde{\bm{\alpha}}=\{1-\beta_i+\alpha_1: i=1,\dots,d\}$ and $\tilde{\bm{\beta}}=\{\alpha_1-\alpha_i+1 : i=1,\dots,d\}.$
   
    \item If $H(x)$ satisfies $D(\tilde{\bm{\alpha}},\tilde{\bm{\beta}} \mid x),$ then $t^{-d^T\alpha_1}H(t^{-d^T})$ satisfies $D(d^T\bm{\alpha},d^T\bm{\beta} \mid t^{d^T}).$
\end{enumerate}    
\end{remark}

As an application of their work, in \cite[\S 11]{AS23}, Adolphson and Sperber look at the de Rham cohomology of projective hypersurfaces and prove Theorem \ref{T:ASPF} for one-parameter monomial deformations of Fermat hypersurfaces. The key ingredient in the proof is a result of Katz that also holds (by the related work of Griffiths in \cite{G69}) for any generalized Delsarte pencil of projective hypersurfaces whose general member is smooth: the central fiber need not be of Fermat type. Therefore, Theorem \ref{T:ASPF} also holds in general as follows:

\begin{theorem}[\protect{\cite[\S 11]{AS23}}]\label{T:ASPF}
Suppose we have a generalized Delsarte pencil of projective hypersurfaces $X_\lambda \subset \mathbb{P}^n_{\C}$ defined by homogeneous degree $d$ polynomials $f_{\lambda} \in \Z[\lambda][X_0,\dots,X_n],$ as in \eqref{eq: flambda}, whose general member is smooth. Let $\mathcal{C}$ be the subspace of $V_{\Z}$ given by elements $\bm{c}$ whose entries are positive and sum to an integer multiple of $d$. The set of monomials corresponding to the lattice points $\mathcal{B} \cap \mathcal{C}$ forms a basis for a subspace $H^{n-1}_\text{hyp}(X_\lambda)$ of middle cohomology. All of these basis elements satisfy hypergeometric Picard-Fuchs equations.
\end{theorem}

\begin{remark}
The normalization of the defining polynomial adopted in \cite{AS23} (see \eqref{eq: flambda}) differs from ours (see \eqref{eq: pencil}) by a linear change of variables on $X_0,\dots,X_n$ and the deformation parameter $\lambda$ that does not affect the values of the parameters $\bm{\alpha}$ and $\bm{\beta}.$ We present the change of variables for $\Csf_4$ and $\Csf_2\Fsf_2$ here.

\label{ex: covc4} For the family $\Csf_4,$ our defining polynomial is $x_0^3x_1+x_1^3x_2+x_2^3x_3+x_3^4-27\psi x_0x_1x_2x_3.$ In Adolphson and Sperber's normalization, it is $9X_0^3X_1+6X_1^3X_2+7X_2^3X_3+5X_3^4-27\lambda X_0X_1X_2X_3.$ Thus the linear change of variables is given by the maps \begin{align*}
X_0 \mapsto 2^{-\frac{1}{9}}3^{\frac{5}{9}}5^{-\frac{1}{108}}7^{\frac{1}{27}}x_0 \qquad 
X_1 \mapsto 2^{\frac{1}{3}}3^{\frac{1}{3}}5^{\frac{1}{36}}7^{-\frac{1}{9}}x_1 \qquad 
X_2 \mapsto 5^{-\frac{1}{12}}7^{\frac{1}{3}}x_2 \qquad 
X_3 \mapsto 5^{\frac{1}{4}}x_3 \qquad 
\lambda \mapsto 2^{\frac{2}{9}}3^{\frac{8}{9}}5^{\frac{5}{27}}7^{\frac{7}{27}}\psi.
\end{align*} 

\label{ex: covc2f2}
For the family $\Csf_2\Fsf_2,$ our defining polynomial is $x_0^3x_1+x_1^4+x_2^4+x_3^4-12\psi x_0x_1x_2x_3$, and Adolphson and Sperber's is $4X_0^3X_1+2X_1^4+3X_2^4+3X_3^4-12\lambda X_0X_1X_2X_3.$ Therefore, we make the following change of variables \begin{align*}
X_0 \mapsto 2^{-\frac{7}{12}}x_0 \qquad X_1 \mapsto 2^{-\frac{1}{4}}x_1 \qquad X_2 \mapsto 3^{-\frac{1}{4}}x_2 \qquad X_3 \mapsto 3^{-\frac{1}{4}}x_3 \qquad \lambda \mapsto 2^{\frac{5}{6}}3^{\frac{1}{2}}\psi.
\end{align*} 

Although the change of variables does not change $\bm{\alpha}$ and $\bm{\beta},$ it does change the parameter $x$ in the hypergeometric series. The scalar multiplying $\psi$ is important for this, as we will see in the next section.    
\end{remark}

\subsection{Picard-Fuchs equations for Delsarte K3 surfaces}
\label{subs: pfK3}
Now we apply the results discussed earlier in the section to our five pencils of Delsarte K3 surfaces. In all of these cases, summing the degree of the holomorphic Picard-Fuchs equation with the dimension of $W_{\psi}$ as in Table \ref{table: kloost}, we obtain exactly $\#(\mathcal{B}\cap \mathcal{C})$. As we will prove in Theorem~\ref{T:hypergeometricMatch}, the hypergeometric parameters obtained here coincide with the parameters from the finite field sums in \S~\ref{S:counting}. 

We begin with a detailed discussion of the families $\Csf_4$ and $\Csf_2\Fsf_2$.

\subsubsection{The family $\Csf_4$}\label{ex:ASC4}
This family is determined by the matrix 

\[\left(\begin{array}{rrrr}
3 & 1 & 0 & 0 \\
0 & 3 & 1 & 0 \\
0 & 0 & 3 & 1 \\
0 & 0 & 0 & 4
\end{array}\right).\]

Let $\mathcal{A}$ be generated by the rows of this matrix. Using SageMath \cite{sage}, we compute that $\mathcal{B} \cap \mathcal{C}$ consists of the following list of 18 lattice points:\begin{align*}
&\{(1, 1, 1, 1), (1, 1, 2, 4), (1, 1, 3, 3), (1, 2, 1, 4), (1, 2, 2, 3), (1, 2, 3, 2), (1, 3, 1, 3), (1, 3, 2, 2), (1, 3, 3, 1),\\&(2, 1, 1, 4), (2, 1, 2, 3), (2, 1, 3, 2), (2, 2, 1, 3), (2, 2, 2, 2), (2, 2, 3, 1), (2, 3, 1, 2), (2, 3, 2, 1), (2, 3, 3, 4)\}.    
\end{align*}

Plugging the point (1, 1, 1, 1) into equation \eqref{eq: eqAS}, we obtain the following sets of parameters \begin{align*}
\bm{\alpha}&=\left[\tfrac{1}{27}, \tfrac{1}{27}, \tfrac{1}{27}, \tfrac{34}{189}, \tfrac{11}{54}, \tfrac{32}{135}, \tfrac{61}{189}, \tfrac{10}{27}, \tfrac{59}{135}, \tfrac{88}{189}, \tfrac{29}{54}, \tfrac{115}{189}, \tfrac{86}{135}, \tfrac{19}{27}, \tfrac{142}{189}, \tfrac{113}{135}, \tfrac{47}{54}, \tfrac{169}{189}\right] \\
\bm{\beta}&=\left[\tfrac{2}{27}, \tfrac{1}{9}, \tfrac{5}{27}, \tfrac{2}{9}, \tfrac{8}{27}, \tfrac{1}{3}, \tfrac{11}{27}, \tfrac{4}{9}, \tfrac{14}{27}, \tfrac{5}{9}, \tfrac{17}{27}, \tfrac{2}{3}, \tfrac{20}{27}, \tfrac{7}{9}, \tfrac{23}{27}, \tfrac{8}{9}, \tfrac{26}{27}, 1\right].    
\end{align*}

Choosing $\alpha_1=\tfrac{1}{27}$ and performing the change of variables described in Remark \ref{rem: cov}, we get the new pair \begin{align*}
&\bm{\alpha}_{\heartsuit_0}=\left[\tfrac{26}{27}, \tfrac{25}{27}, \tfrac{23}{27}, \tfrac{22}{27}, \tfrac{20}{27}, \tfrac{19}{27}, \tfrac{17}{27}, \tfrac{16}{27}, \tfrac{14}{27}, \tfrac{13}{27}, \tfrac{11}{27}, \tfrac{10}{27}, \tfrac{8}{27}, \tfrac{7}{27}, \tfrac{5}{27}, \tfrac{4}{27}, \tfrac{2}{27}, \tfrac{1}{27}\right]\\&\bm{\beta}_{\heartsuit_0}=\left[1, 1, 1, \tfrac{6}{7}, \tfrac{5}{6}, \tfrac{4}{5}, \tfrac{5}{7}, \tfrac{2}{3}, \tfrac{3}{5}, \tfrac{4}{7}, \tfrac{1}{2}, \tfrac{3}{7}, \tfrac{2}{5}, \tfrac{1}{3}, \tfrac{2}{7}, \tfrac{1}{5}, \tfrac{1}{6}, \tfrac{1}{7}\right], 
\end{align*} which is associated to the following hypergeometric series
$$\lambda^{-1}{_{18}F_{17}(\bm{\alpha}_{\heartsuit_0},\bm{\beta}_{\heartsuit_0} \mid \lambda^{-27})}.$$ 

Moreover, we recall that $\lambda$ should be replaced by $2^{\frac{2}{9}}3^{\frac{8}{9}}5^{\frac{5}{27}}7^{\frac{7}{27}}\psi,$ so the associated hypergeometric series is $$2^{-\frac{2}{9}}3^{-\frac{8}{9}}5^{-\frac{5}{27}}7^{-\frac{7}{27}}\psi^{-1}{_{18}F_{17}(\bm{\alpha}_{\heartsuit_0},\bm{\beta}_{\heartsuit_0} \mid t_{\heartsuit_0})}.$$

On the other hand, using Theorem \ref{thm: gahrs}, we conclude that the Picard-Fuchs equation satisfied by the holomorphic periods is given by the differential operator $D(\bm{\alpha}_{\heartsuit_0},\bm{\beta}_{\heartsuit_0} \mid t_{\heartsuit_0}).$ Moreover, we claim that all the other periods that satisfy hypergeometric differential equations are also solutions for the same operator. Indeed, if $\omega$ denotes the holomorphic form, then $D(\bm{\alpha}_{\heartsuit_0},\bm{\beta}_{\heartsuit_0} \mid t_{\heartsuit_0})\cdot \omega=0$ by G\"ahrs. Moreover, $$0=\Theta^i D(\bm{\alpha}_{\heartsuit_0},\bm{\beta}_{\heartsuit_0} \mid t_{\heartsuit_0})\cdot \omega=D(\bm{\alpha}_{\heartsuit_0},\bm{\beta}_{\heartsuit_0} \mid t_{\heartsuit_0})\cdot \Theta^i\omega,$$ for every $i=1,\dots,18.$ By Proposition \cite[Proposition 4.1.11]{ZF18} the differential equation $D(\bm{\alpha}_{\heartsuit_0},\bm{\beta}_{\heartsuit_0} \mid t_{\heartsuit_0})\cdot f=0$ is irreducible. Therefore, it follows that it is the Picard-Fuchs equation associated to $\omega$ and its seventeen derivatives. 

The family $\Csf_4$ demonstrates that Theorem~\ref{T:ASPF} may produce a basis of primitive cohomology elements satisfying hypergeometric Picard-Fuchs equations that is strictly smaller than the dimension of $H^{n-1}(X)$. The remaining elements in a full basis for primitive cohomology need not have hypergeometric Picard-Fuchs equations. 

Still, for $\Csf_4,$ $H^{2}_\text{hyp}(X_\psi)$ is spanned by the holomorphic form and its derivatives. In general, the group of symmetries $G$ of the invertible pencil decomposes $H^{n-1}_\text{hyp}(X_\psi)$ into distinct orthogonal subspaces based on the different characters of the group action on monomials.

\subsubsection{The family $\Csf_2 \Fsf_2$}\label{ex:ASC2F2}
Consider the family $\Csf_2 \Fsf_2$ given by $x_0^3x_1+x_1^4+x_2^4+x_3^4 - 12 \psi x_0 x_1 x_2 x_3$. The group of symmetries of this family is $\Z/4\Z$. We may view the symmetry group as generated by $x_2 \mapsto i x_2$, $x_3 \mapsto -i x_3$. Under this action, $\mathcal{B} \cap \mathcal{C}$ decomposes into four subsets:

\begin{enumerate}[(a)]
\item \{(1, 1, 1, 1), (1, 3, 2, 2), (1, 1, 3, 3), (2, 4, 1, 1), (2, 2, 2, 2), (2, 4, 3, 3)\}
\item \{(1, 4, 2, 1), (1, 2, 3, 2), (2, 3, 2, 1), (2, 1, 3, 2)\}
\item \{(1, 3, 1, 3), (1, 3, 3, 1), (2, 2, 1, 3), (2, 2, 3, 1)\}
\item \{(1, 4, 1, 2), (1, 2, 2, 3), (2, 3, 1, 2), (2, 1, 2, 3)\}.
\end{enumerate}

Again, we plug in a representative from each class of monomials, and obtain Picard-Fuchs equations using the methods described in \cite{AS23}. The result are hypergeometric differential equations with parameters as follows:

\begin{enumerate}[(a)]
\item $\bm{\alpha}=\left[\frac{11}{12}, \frac{5}{6}, \frac{7}{12}, \frac{5}{12}, \frac{1}{6}, \frac{1}{12}\right], \bm{\beta}= \left[1, 1, 1, \frac{2}{3}, \frac{1}{2}, \frac{1}{3}\right]$
\item $\bm{\alpha}=\left[\frac{11}{12}, \frac{5}{6}, \frac{7}{12}, \frac{1}{6}\right], \bm{\beta}= \left[1, \frac{7}{8}, \frac{3}{4}, \frac{3}{8}\right]$
\item $\bm{\alpha}=\left[\frac{11}{12}, \frac{7}{12}, \frac{5}{12}, \frac{1}{12}\right], \bm{\beta}= \left[1, \frac{3}{4}, \frac{1}{2}, \frac{1}{4}\right]$
\item $\bm{\alpha}=\left[\frac{5}{6}, \frac{5}{12}, \frac{1}{6}, \frac{1}{12}\right], \bm{\beta}=\left[1, \frac{5}{8}, \frac{1}{4}, \frac{1}{8}\right]$.
\end{enumerate}

For instance, plugging the monomial $(1,4,1,2)$ into \eqref{eq: eqAS}, we first obtain the parameters \begin{align*}
\bm{\alpha}=\left[\tfrac{1}{12}, \tfrac{11}{24}, \tfrac{5}{6}, \tfrac{23}{24}\right], \bm{\beta}=\left[\tfrac{1}{4}, \tfrac{2}{3}, \tfrac{11}{12}, 1\right]. 
\end{align*}
We choose $\alpha_1=\tfrac{1}{12}$ and apply the change of variables defined in \ref{rem: cov} to obtain the new pair 
\begin{align*}
\bm{\alpha}_{\clubsuit_2}=\left[\tfrac{5}{6}, \tfrac{5}{12}, \tfrac{1}{6}, \tfrac{1}{12}\right], \bm{\beta}_{\clubsuit_2}=\left[1, \tfrac{5}{8}, \tfrac{1}{4}, \tfrac{1}{8}\right]    
\end{align*} associated to the hypergeometric series $\lambda^{-1}{_4F_3}(\bm{\alpha}_{\clubsuit_2},\bm{\beta}_{\clubsuit_2} \mid \lambda^{-12}).$ Recall that $\lambda$ should be replaced by $2^{\frac{5}{6}}3^{\frac{1}{2}}\psi$ and we obtain $2^{-\frac{5}{6}}3^{-\frac{1}{2}}\psi^{-1}{_4F_3}(\bm{\alpha}_{\clubsuit_2},\bm{\beta}_{\clubsuit_2} \mid t_{\clubsuit_2}).$ 

For the monomial $(2,3,2,1),$ we initially obtain \begin{align*}
\bm{\alpha}=\left[\tfrac{1}{6}, \tfrac{7}{24}, \tfrac{5}{12}, \tfrac{19}{24}\right], \bm{\beta}=\left[\tfrac{1}{4}, \tfrac{1}{3}, \tfrac{7}{12}, 1\right].
\end{align*} We apply the change of variables defined in \ref{rem: cov} and get \begin{align*}
\bm{\alpha}_{\clubsuit_3}=\left[\tfrac{11}{12}, \tfrac{5}{6}, \tfrac{7}{12}, \tfrac{1}{6}\right], \bm{\beta}_{\clubsuit_3}=\left[1, \tfrac{7}{8}, \tfrac{3}{4}, \tfrac{3}{8}\right],    
\end{align*} associated to the hypergeometric function $2^{-\frac{5}{3}}3^{-1}\psi^{-2}{_4F_3(\bm{\alpha}_{\clubsuit_3},\bm{\beta}_{\clubsuit_3} \mid t_{\clubsuit_3})}.$ 

We do a similar computation for the other two clusters to obtain the corresponding hypergeometric parameters. Notice that for both families, $\Csf_4$ and $\Csf_2\Fsf_2,$ adding the order of the Picard-Fuchs equation associated to the holomorphic form and $\dim W_{\psi},$ as in Table \ref{table: kloost}, we obtain $\#(\mathcal{B}\cap \mathcal{C}).$ 

\begin{comment}
If we plug (1,4,1,2) into A-S we get $$\lambda^{-1} {_4F_3(\bm{\alpha}_{\clubsuit_1},\bm{\beta}_{\clubsuit_1} \mid \lambda^{-12})}=3^{-\frac{1}{2}}2^{-\frac{5}{6}}\psi^{-1} {_4F_3(\bm{\alpha}_{\clubsuit_1},\bm{\beta}_{\clubsuit_1} \mid 3^{-6}2^{-10}\psi^{-12})}.$$ For the monomial (2,3,1,2) we get $$3^{-\frac{1}{2}}2^{-\frac{5}{6}}\psi^{-2}{_4F_3(\bm{\alpha}_{\clubsuit_3},\bm{\beta}_{\clubsuit_3} \mid 2^{-10}3^{-6}\psi^{-12})}.$$    
\end{comment}

\begin{prop}
Let $\bullet \in \{\Csf_4,\Csf_2\Fsf_2,\Csf_3\Fsf_1,\Csf_2\Lsf_2,\Csf_2\Csf_2\}$ denote one of our five families, with group of symmetries $H=H_{\bullet}.$ Then $H$ acts on the $\C$-vector space $H^2_{\text{hyp}}(X_{\psi}),$ giving a representation $H \to \text{GL}(H^2_{\text{hyp}}(X_{\psi})).$ Since $H$ is abelian, we have a decomposition $$H^2_{\text{hyp}}(X_{\psi})=\bigoplus_{\chi}V_{\chi},$$ where $\chi:H \to \C^{\times}$ is a character and $V_{\chi}$ is the eigenspace associated to $\chi.$ The holomorphic piece corresponds to the trivial character, and the remaining pieces form a subspace of dimension $\dim W_{\psi}.$
\end{prop}

\subsubsection{The remaining pencils}
Now we briefly compute the solutions for the Picard-Fuchs equations of the remaining families.
\subsubsection{The family $\Csf_3\Fsf_1$}
This family is given by the matrix 
\begin{align*}
\begin{pmatrix*}
3 & 1 & 0 & 0 \\
0 & 3 & 1 & 0 \\
0 & 0 & 4 & 0 \\
0 & 0 & 0 & 4
\end{pmatrix*}.    
\end{align*}   

The symmetry group of this family is trivial and therefore the decomposition of $\mathcal{B}\cap \mathcal{C}$ is also trivial as follows \begin{align*}
&\{(1, 1, 1, 1), (1, 1, 4, 2), (1, 1, 3, 3), (1, 2, 4, 1), (1, 2, 3, 2), (1, 2, 2, 3),(1, 3, 3, 1), (1, 3, 2, 2), (1, 3, 1, 3),\\& (2, 1, 4, 1), (2, 1, 3, 2), (2, 1, 2, 3), (2, 2, 3, 1), (2, 2, 2, 2), (2, 2, 1, 3), (2, 3, 2, 1), (2, 3, 1, 2), (2, 3, 4, 3)\}.
\end{align*}
Plugging (1,1,1,1) into \eqref{rem: cov}, we get the hypergeometric parameters, which coincide with those obtained from the finite field sums
\begin{align*}
\bm{\alpha}_{\spadesuit_0}&=\left[\tfrac{35}{36}, \tfrac{17}{18}, \tfrac{31}{36}, \tfrac{29}{36}, \tfrac{13}{18}, \tfrac{25}{36}, \tfrac{23}{36}, \tfrac{11}{18}, \tfrac{19}{36}, \tfrac{17}{36}, \tfrac{7}{18}, \tfrac{13}{36}, \tfrac{11}{36}, \tfrac{5}{18}, \tfrac{7}{36}, \tfrac{5}{36}, \tfrac{1}{18}, \tfrac{1}{36}\right],\\ \bm{\beta}_{\spadesuit_0}&=\left[1, 1, 1, \tfrac{7}{8}, \tfrac{6}{7}, \tfrac{3}{4}, \tfrac{5}{7}, \tfrac{2}{3}, \tfrac{5}{8}, \tfrac{4}{7}, \tfrac{1}{2}, \tfrac{3}{7}, \tfrac{3}{8}, \tfrac{1}{3}, \tfrac{2}{7}, \tfrac{1}{4}, \tfrac{1}{7}, \tfrac{1}{8}\right], 
\end{align*}
associated to the hypergeometric series $2^{-\frac{4}{3}}3^{-\frac{5}{6}}7^{-\frac{7}{36}}\psi^{-1}{_{18}F_{17}}(\bm{\alpha}_{\spadesuit_0},\bm{\beta}_{\spadesuit_0} \mid t_{\spadesuit_0}).$

\subsubsection{The family $\Csf_2\Lsf_2$}
This family is given by the matrix 
\begin{align*}
\begin{pmatrix*}
3 & 1 & 0 & 0 \\
0 & 4 & 0 & 0 \\
0 & 0 & 3 & 1 \\
0 & 0 & 1 & 3
\end{pmatrix*}.    
\end{align*}  
The symmetry group of this family is $\Z/2\Z,$ generated by $x_2 \mapsto -x_2$ and $x_3 \mapsto -x_3.$ Under the action of this group, $\mathcal{B} \cap \mathcal{C}$ decomposes as follows 
\begin{itemize}
    \item[(a)] \{(1, 1, 1, 1), (1, 3, 2, 2), (1, 1, 3, 3), (2, 4, 1, 1), (2, 2, 2, 2), (2, 4, 3, 3)\},
    \item[(b)] \{(1, 4, 1, 2), (1, 4, 2, 1), (1, 2, 2, 3), (1, 2, 3, 2), (2, 3, 1, 2), (2, 3, 2, 1), (2, 1, 2, 3), (2, 1, 3, 2)\}.
\end{itemize}
\noindent Plugging $(1,1,1,1)$ and $(2,3,1,2)$ into equation \eqref{eq: eqAS}, respectively, and applying the change of variables described in Remark \ref{rem: cov}, we get the following pairs of hypergeometric parameters and their associated hypergeometric series
\begin{itemize}
    \item[(a)] $\bm{\alpha}_{\clubsuit_0}=\left[\tfrac{1}{12},\tfrac{1}{6},\tfrac{5}{12},\tfrac{7}{12},\tfrac{5}{6},\tfrac{11}{12}\right],\bm{\beta}_{\clubsuit_0}=\left[\tfrac{1}{2},\tfrac{1}{3},\tfrac{2}{3},1,1,1\right], 2^{-\frac{5}{6}}3^{-\frac{1}{2}}\psi^{-1}{_6F_5(\bm{\alpha}_{\clubsuit_0},\bm{\beta}_{\clubsuit_0} \mid t_{\clubsuit_0})}.$ 
    \item[(b)]  $\bm{\alpha}_{\clubsuit_4}=\left[\frac{23}{24}, \frac{19}{24}, \frac{17}{24}, \frac{13}{24}, \frac{11}{24}, \frac{7}{24}, \frac{5}{24}, \frac{1}{24}\right], \bm{\beta}_{\clubsuit_4}=\left[1, \frac{5}{6}, \frac{3}{4}, \frac{2}{3}, \frac{1}{2}, \frac{1}{3}, \frac{1}{4}, \frac{1}{6}\right], 2^{-\frac{5}{12}}3^{-\frac{1}{4}}\psi^{-\tfrac{1}{2}}{_8F_7(\bm{\alpha}_{\clubsuit_4},\bm{\beta}_{\clubsuit_4} \mid t_{\clubsuit_4})}.$
\end{itemize} 

Here following the notation in Remark \ref{rem: cov}, we have $l=12$ and $\alpha_1=\tfrac{1}{24}.$ 

\subsubsection{The family $\Csf_2\Csf_2$}
This family is given by the matrix 
\begin{align*}
\begin{pmatrix*}
3 & 1 & 0 & 0 \\
0 & 4 & 0 & 0 \\
0 & 0 & 3 & 1 \\
0 & 0 & 0 & 4
\end{pmatrix*}.    
\end{align*} 
The symmetry group of this family is $\Z/6\Z,$ generated by $x_0 \mapsto e^{\frac{4\pi i}{3}}x_0,x_1 \mapsto x_1, x_2 \mapsto e^{\frac{5\pi i}{3}}x_2, x_3 \mapsto -x_3.$ The action of this group decomposes $\mathcal{B}\cap \mathcal{C}$ into 6 subsets as follows \begin{itemize}
\item[(a)] \{(1, 1, 1, 1),(1, 3, 1, 3),(2, 2, 2, 2),(2, 4, 2, 4)\}, 
\item[(b)] \{(1, 2, 1, 4),(1, 4, 1, 2),(2, 1, 2, 3),(2, 3, 2, 1)\},
\item[(c)] \{(2, 2, 1, 3),(2, 4, 1, 1)\}, 
\item[(d)] \{(1, 1, 2, 4),(1, 3, 2, 2)\},
\item[(e)] \{(1, 2, 2, 3),(1, 4, 2, 1)\},
\item[(f)] \{(2, 1, 1, 4),(2, 3, 1, 2)\}. 
\end{itemize}
Plugging (1, 1, 1, 1), (2, 1, 2, 3), (1, 1, 2, 4), (1, 2, 2, 3), (2, 1, 1, 4) and (2, 2, 1, 3), respectively in \eqref{eq: eqAS} and performing the change of variables described in Remark \ref{rem: cov}, we obtain the following triples of hypergeometric parameters:
\begin{itemize}
    \item[(a)] $\bm{\alpha}_{\symking_0}=\left[\frac{5}{6}, \frac{2}{3}, \frac{1}{3}, \frac{1}{6}\right], \bm{\beta}_{\symking_0}=\left[1, 1, 1, \frac{1}{2}\right],t_{\symking_0}$ 
    \item[(b)] $\bm{\alpha}_{\symking_1}=\left[\frac{11}{12}, \frac{7}{12}, \frac{5}{12}, \frac{1}{12}\right], \bm{\beta}_{\symking_1}=\left[1, \frac{3}{4}, \frac{1}{2}, \frac{1}{4}\right],t_{\symking_1}$ 
    \item[(c)] $\bm{\alpha}_{\symking_2}=\left[\frac{5}{6}, \frac{1}{3}\right], \bm{\beta}_{\symking_2}=\left[1, \frac{2}{3}\right],t_{\symking_2}$ 
    \item[(d)] $\bm{\alpha}_{\symking_3}=\left[\frac{2}{3}, \frac{1}{6}\right], \bm{\beta}_{\symking_3}=\left[1, \frac{1}{3}\right],t_{\symking_3}$ 
    \item[(e)] $\bm{\alpha}_{\symking_4}=\left[\frac{11}{12}, \frac{5}{12}\right], \bm{\beta}_{\symking_4}=\left[1, \frac{5}{6}\right],t_{\symking_4}$ 
    \item[(f)] $\bm{\alpha}_{\symking_5}=\left[\frac{7}{12}, \frac{1}{12}\right], \bm{\beta}_{\symking_5}=\left[1, \frac{1}{6}\right],t_{\symking_5}.$ 
\end{itemize}

\bigskip

\begin{thm}\label{T:hypPeriodList}
\begin{itemize}
    \item[(a)] The family $\Csf_4$ has \begin{align*}
    \text{18 periods associated to } {_{18}F_{17}(\bm{\alpha}_{\heartsuit_0},\bm{\beta}_{\heartsuit_0} \mid t_{\heartsuit_0})}. 
    \end{align*}
    \item[(b)] The family $\Csf_2\Fsf_2$ has \begin{align*}
\text{6 periods associated to }& {_{6}F_{5}(\bm{\alpha}_{\clubsuit_0},\bm{\beta}_{\clubsuit_0} \mid t_{\clubsuit_0})},\\
\text{4 periods associated to each } {_{4}F_{3}(\bm{\alpha}_{\clubsuit_1},\bm{\beta}_{\clubsuit_1} \mid t_{\clubsuit_1})},& \hspace{0.1cm} {_{4}F_{3}(\bm{\alpha}_{\clubsuit_2}, \bm{\beta}_{\clubsuit_2} \mid t_{\clubsuit_2})} \text{ and }{_{4}F_{3}(\bm{\alpha}_{\clubsuit_3}, \bm{\beta}_{\clubsuit_3} \mid t_{\clubsuit_3})}.
\end{align*}
    \item[(c)] The family $\Csf_3\Fsf_1$ has \begin{align*}
\text{18 periods associated to } {_{18}F_{17}(\bm{\alpha}_{\spadesuit_0},\bm{\beta}_{\spadesuit_0} \mid t_{\spadesuit_0})}.
\end{align*}
\item[(d)] The family $\Csf_2\Lsf_2$ has \begin{align*}
\text{6 periods associated to } {_{6}F_{5}(\bm{\alpha}_{\clubsuit_0},\bm{\beta}_{\clubsuit_0} \mid t_{\clubsuit_0})}.\\
\text{8 periods associated to } {_{8}F_{7}(\bm{\alpha}_{\clubsuit_4},\bm{\beta}_{\clubsuit_4} \mid t_{\clubsuit_4})}.
\end{align*}
\item[(e)] The family $\Csf_2\Csf_2$ has \begin{align*}
\text{4 periods associated to each } &{_{4}F_{3}(\bm{\alpha}_{\symking_0},\bm{\beta}_{\symking_0} \mid t_{\symking_0})},{_{4}F_{3}(\bm{\alpha}_{\symking_1},\bm{\beta}_{\symking_1} \mid t_{\symking_1})}\\
\text{2 periods associated to each } &{_{2}F_{1}(\bm{\alpha}_{\symking_2},\bm{\beta}_{\symking_2} \mid t_{\symking_2})},{_{2}F_{1}(\bm{\alpha}_{\symking_3},\bm{\beta}_{\symking_3} \mid t_{\symking_3})},\\&{_{2}F_{1}}(\bm{\alpha}_{\symking_4},\bm{\beta}_{\symking_4} \mid t_{\symking_4}), {_{2}F_{1}}(\bm{\alpha}_{\symking_5},\bm{\beta}_{\symking_5} \mid t_{\symking_5}).
\end{align*}
\end{itemize}
\end{thm}
\begin{proof}
The existence of the period matching the hypergeometric series comes from the computations in the previous sections. 
\end{proof}

\section{Hypergeometric sums over finite fields}\label{S:finiteHyp}

As periods are associated with hypergeometric series over the complex numbers, hypergeometric sums are associated with point counts over finite fields. One of our goals is to explicitly detect this connection, so here we start by defining hypergeometric sums and exploring some of their properties that will be useful for our computations later. 

Let $p$ be an odd prime and $q=p^u$ for some natural number $u.$ We denote $\q:=q-1$ and observe that because $p$ is odd, then $\q$ is always even. Fix $\omega$ a generator of the cyclic multiplicative character group of $\F_q^{\times}$ and $\Theta$ a non-trivial additive character.

\begin{definition}\label{def:gaussSum}
Given $m \in \Z,$ we define the \textit{Gauss sum} to be $$g(m)=\sum_{x \in \F_q^{\times}}\omega^m(x)\Theta(x) \in \C^{\times}.$$     
\end{definition}

\begin{remark}\label{rmk: gausssum}
If $\theta$ is another non-trivial additive character, then $\theta=\Theta_k: x \mapsto \Theta(kx),$ for some $k \in \F_q^{\times}.$ The Gauss sums defined by $\theta$ will differ from the Gauss sum defined in Definition~\ref{def:gaussSum} in the following way: $$g_k(m)=\omega(k)^{-m}g(m).$$ 
\end{remark}

\begin{lemma}[Lemma 3.1.3, \cite{HD20}] \label{hd}
We have the following relations
\begin{itemize}
    \item[\textbf{(a)}] $g(0)=-1.$
    \item[\textbf{(b)}] $g(m)g(-m)=(-1)^mq$ for every $m \not \equiv 0 \Mod \q$ and in particular $g(\frac{\q}{2})=(-1)^{\frac{\q}{2}}q.$
    \item[\textbf{(c)}]\textbf{(Hasse-Davenport)} Let $N\mid\q$ with $N>0,$ we have $g(Nm)=-\omega(N)^{Nm}\prod_{j=0}^{N-1}\dfrac{g(m+j\q/N)}{g(j\q/N)}.$
\end{itemize}
\end{lemma}
\begin{proof}
See Theorem 1.1.4 in \cite{BW98} and Theorem 3.7.3 in \cite{CO07}.
\end{proof}

\begin{definition}
Given $m,l \in \Z,$ we define the \emph{Jacobi sum} of $m$ and $l$ as $$J_q(m,l)=\sum_{x \in \F_q^{\times}}\omega^m(x)\omega^l(1-x).$$    
\end{definition}

\begin{lemma}\label{lemma: JacobiGauss}
If $m+l \not\equiv 0 \Mod{\q},$ then $$J_q(m,l)=\frac{g(m)g(l)}{g(m+l)}=\frac{(-1)^{m+l}}{q}g(m)g(l)g(-(m+l)).$$  
\end{lemma}
\begin{proof}
Generalizing \cite[Theorem 8.3.1.(d)]{IR13} from the prime case to the prime power case.   
\end{proof}

\begin{remark}\label{rmk: notationJac}
In Section \ref{S:counting} (e.g Lemma \ref{lemma: lemmac41}), we write the point counts in terms of a polynomial in $q,$ hypergeometric functions and Jacobi sums. In that situation, to shorten the notation, for $a,b \in \Q,$ we write $$J(a,b):=J(a\q,b\q).$$   
\end{remark}

Let $\bm{\alpha}=\{\alpha_1,\dots,\alpha_d\}$ and $\bm{\beta}=\{\beta_1,\dots,\beta_d\}$ be multisets in $\Q.$ Suppose also that $\alpha_i-\beta_j \not \in \Z$ for every $i,j=1,\dots,d.$ 

\begin{definition}\label{def: def}
    The \textit{field of definition} of $\bm{\alpha}, \bm{\beta},$ denoted by $K_{\bm{\alpha}, \bm{\beta}}$ is the field generated over $\Q$ by the coefficients of the polynomials $$F(x)=\prod_{i=1}^d(x-e^{2\pi i \alpha_i}) \qquad \text{ and } \qquad G(x)=\prod_{i=1}^d(x-e^{2\pi i \beta_i}).$$ We say that $\bm{\alpha}, \bm{\beta}$ are \textit{defined over $\Q$} if $K_{\bm{\alpha},\bm{\beta}}=\Q.$ We say that $q$ is \textit{good} for $\bm{\alpha},\bm{\beta}$ if $\gcd(q,\lcd(\bm{\alpha} \cup \bm{\beta}))=1,$ where $\lcd(\bm{\alpha} \cup \bm{\beta})$ denotes the least common denominator of the elements of the multiset $\bm{\alpha} \cup \bm{\beta}.$
\end{definition}

Suppose that $\bm{\alpha}, \bm{\beta}$ are defined over $\Q,$ then $F$ and $G$ have rational coefficients and can be written as a product of cyclotomic polynomials. Consequently, it is possible to write $$\dfrac{F}{G}=\dfrac{\prod_{j=1}^{r}x^{p_j}-1}{\prod_{j=1}^{s}x^{q_j}-1} \quad \text{ where }  \quad \{p_1,\dots,p_r\} \cap \{q_1,\dots,q_s\}=\varnothing.$$ 
\begin{definition}[Defined over $\mathbb{Q}$,\cite{BCM15}]
\label{def: def2}
Suppose that $\bm{\alpha},\bm{\beta}$ are defined over $\Q$ and $q$ is good for $\bm{\alpha},\bm{\beta}.$ Let $$D(x)=\gcd\left(\prod_{j=1}^{r}x^{p_j}-1,\quad \prod_{j=1}^{s}x^{q_j}-1\right), \quad\epsilon=(-1)^{\sum_{j=1}^sq_j},\quad M=\prod_{j=1}^{r}p_j^{p_j}\prod_{j=1}^{s}q_j^{-q_j}$$ and consider the function \begin{align*}
s: \Z/\q\Z &\to \Z \\  
m &\mapsto \text{multiplicity of } e^{\frac{2\pi i m}{\q}} \text{ in } D(x).
\end{align*}

For $t \in \F_q^{\times},$ define the \textit{finite field hypergeometric sum} $$H_q(\bm{\alpha},\bm{\beta}\mid t)=\dfrac{(-1)^{r+s}}{1-q}\sum_{m=0}^{q-2}q^{-s(0)+s(m)}\bigg(\prod_{i=1}^{r}g(p_im)\prod_{i=1}^{s}g(-q_im)\bigg)\omega(\epsilon M^{-1} t)^m.$$
\end{definition} 

There have been many different definitions of hypergeometric functions over a finite field over the years. The definition above, for instance, coincides with McCarthy's definition when the parameters satisfy $\q \alpha_i, \q \beta_j \in \Z$ for every $i,j$ (see \autocite{BCM15}). A more general definition that allows some parameters that are not defined over $\Q$ is the \textit{splittable} hypergeometric function, which we do not define here, but refer to \autocite[Definition 3.2.7]{HD20}. These versions all rely on the fact that Gauss sums are finite field analogues of the Gamma function, which are the basis for the definition of classical hypergeometric functions. 

\begin{remark} There is another family of analogous definitions for hypergeometric functions over a finite field, originated by Greene \cite{Gre}, which uses Jacobi sums instead of Gauss sums. This version has some advantages depending on the context; for example, McCarthy used Greene's finite field hypergeometric function to study periods of elliptic curves in \cite{Mac}.  
In particular, Jacobi sums are amenable to computer calculations -- though thanks to the Gross-Koblitz formula \cite{GK79}, Gauss sums are actually quite computationally feasible, and in fact are what we use to check our point counting computations in Section 5.  
\end{remark}

As discussed in the introduction, the sums we deal with in this paper are not expressible in terms of the definitions above, but rather require an even more general one. This is found in the work of Asem Abdelraouf and which is based on the concept of \emph{gamma triples} below.

\begin{definition}[Definition 2.8, \cite{AA25}]
    A \emph{gamma triple} is a triple $(\bm{\gamma}, \bm{\delta}, N)$ such that 
    \begin{enumerate}
        \item $\bm{\gamma} \in \Z^{d+2}$ is a vector such that $\sum_{j=1}^{d+2}\bm{\gamma}_j=0$, and $\gcd(\bm{\gamma}_1,\dots,\bm{\gamma}_{d+2})=1$;
        \item $\bm{\delta}\in\Z^{d+2}$ is any vector.
        \item $N$ is a positive integer. 
    \end{enumerate}
\end{definition}

\begin{definition}\label{def: quot}
    To a gamma triple $(\bm{\gamma},\bm{\delta}, N)$ we associate two vectors $\bm{\alpha}, \bm{\beta} \in \Q^n$, $n\geq0$, by writing
\begin{align}\label{alg: algnew}
\frac{\prod_{j=1}^n(T-e^{2\pi i\alpha_j})}{\prod_{j=1}^n(T-e^{2\pi i \beta_j})}=\frac{\prod_{\gamma_i<0}T^{-\gamma_i}-\zeta_N^{\delta_i}}{\prod_{\gamma_i>0}T^{\gamma_i}-\zeta_N^{-\delta_i}},
\end{align}
    where the greatest common divisor of the denominator and the numerator of the left-hand side is 1. The parameters $\alpha,\beta$ are called the hypergeometric parameters associated to the triple $(\bm{\gamma},\bm{\delta},N)$. We say that the pair $\alpha,\beta$ is defined over $\Q(\zeta_N)$. 
\end{definition}

Let $D_{\bm{\delta}}(T)$ be the greatest common divisor of the numerator and denominator of the right-hand side of the equation above. We define $s_{\bm{\delta}}(m)$ to be the multiplicity of $e^{2\pi im/\q}$ in $D_{\bm{\delta}}(T),$ $\bm{\gamma}^{\bm{\gamma}}:=\prod_{j=1}^{d+2}\gamma_j^{\gamma_j}$ and

\begin{definition}[Definition 2.17,\cite{AA25}]\label{AA25}
Let $(\bm{\bm{\gamma}},\bm{\bm{\delta}},N)$ be a gamma triple. Suppose that $\q$ is divisible by $N$ and let $t \in \F_q^{\times}.$ We define \begin{align}
F_q(\bm{\gamma}, \bm{\delta},N\mid t)=\frac{1}{1-q}\sum_{m=0}^{q-2}\bigg(\prod_{i=1}^{|\bm{\gamma}|}\frac{g(-\gamma_i m+\frac{\delta_i}{N}\q)}{g(\frac{\delta_i}{N}\q)}q^{s_{\bm{\delta}}(-m)-s_{\bm{\delta}}(0)}\bigg)\omega(\bm{\gamma}^{\bm{\gamma}}t)^m.    
\end{align}
\end{definition}

The definition above does not depend on the additive character $\Theta$.  

Indeed, by Remark \ref{rmk: gausssum}, we have
\begin{align*}
    \prod_{i=1}^{|\bm{\gamma}|}\frac{g_k(-\gamma_im+\delta_i\tfrac{\q}{N})}{g_k(\delta_i\tfrac{\q}{N})}&=\prod_{i=1}^{|\bm{\gamma}|}\frac{\omega(k)^{-\gamma_im-\delta_i\tfrac{\q}{N}}g(-\gamma_im+\delta_i\tfrac{\q}{N})}{\omega(k)^{-\delta_i\tfrac{\q}{N}}g(\delta_i\tfrac{\q}{N})}=\omega(k)^{-m\sum_{i=1}^{|\bm{\gamma}|}\gamma_i}\prod_{i=1}^{|\bm{\gamma}|}\frac{g(-\gamma_im+\delta_i\tfrac{\q}{N})}{g(\delta_i\tfrac{\q}{N})}. 
    \end{align*} Since $\sum_{i=1}^{|\bm{\gamma}|}\gamma_i=0,$ we conclude the proof.

Moreover, if $\bm{\delta}=0$ and $N=1$ and the pair $\bm{\alpha},\bm{\beta}$ satisfies equation \eqref{alg: algnew}, then $\bm{\alpha},\bm{\beta}$ is defined over $\Q$, and this definition recovers Definition \ref{def: def2}. 

It follows directly from the definition that $F_q(\bm{\gamma},\bm{\delta},N\mid t) \in \Q(\zeta_{\q},\zeta_p).$ Moreover, since $F_q$ does not depend on the choice of the additive character, as pointed out above, we have $F_q(\bm{\gamma},\bm{\delta},N\mid t) \in \Q(\zeta_{\q}).$ Finally, one can also verify that  the hypergeometric function descends to $\Q(\zeta_N)$, that is, $$F_q(\bm{\gamma},\bm{\delta},N \mid t) \in \Q(\zeta_N).$$ 
(We thank Asem Abdelraouf for describing how this follows from \cite{AA25} in private communication.)

In the remainder of this section, we prove some properties of $F_q$ that will be used in \S \ref{sub: subL}. For what follows, recall that the group $\text{Gal}(\Q(\zeta_{\q}) \mid \Q) \simeq (\Z/\q\Z)^{\times}$ acts by $\sigma_k(\zeta_{\q})=\zeta_{\q}^k$ where $k \in (\Z/\q\Z)^{\times}.$ 

\begin{lemma}\label{lemma: sdelta}
Let $k\delta = (k \delta_1, \hdots, k\delta_{d+2})$, and $k \in (\mathbb{Z} / q^\times \mathbb{Z})^\times$. Then  $s_{\bm{\delta}}(m)=s_{k\bm{\delta}}(km).$    
\end{lemma}

\begin{proof}
By definition, $s_{k\delta}(km)$ is the multiplicity of $\zeta_{q^\times}^{km}$ as a root of the greatest common divisor of $\prod_{\gamma_i < 0} T^{-\gamma_i} - \zeta_{q^\times}^{k\delta_i}$ and $\prod_{\gamma_i > 0} T^{\gamma_i} - \zeta_{q^\times}^{k\delta_i}.$ Letting $\sigma_{k^{-1}}$ be the Galois element taking $\zeta_{q^\times}$ to $\zeta_{q^\times}^{k^{-1}}$, where $k^{-1}$ is the inverse of $k$ in $(\mathbb{Z} / q^\times \mathbb{Z})^\times$, and using the fact that the Galois action preserves greatest common divisors, we see that this is equal to the multiplicity of $\sigma_{k^{-1}}(\zeta_{q^\times}^{km}) = \zeta_{q^\times}^m$ as a root of the g.c.d of $\sigma_{k^{-1}}(\prod_{\gamma_i < 0} T^{-\gamma_i} - \zeta_{q^\times}^{k\delta_i})$ and $\sigma_{k^{-1}}(\prod_{\gamma_i > 0} T^{\gamma_i} - \zeta_{q^\times}^{k\delta_i}),$ which is $s_{\delta}(m)$.
\end{proof}
\begin{comment}
\begin{proof}
We prove the lemma by show the following equivalent statement: $$e^{-\frac{2\pi i m}{\q}} \text{ is a root of } D_{\bm{\delta}} \text{ if and only if } e^{-\frac{2\pi i k m}{\q}} \text{ is a root of } D_{k\bm{\delta}}.$$   

If $e^{-\frac{2\pi i m}{\q}}$ is a root of $D_{\bm{\delta}},$ then there exists $i$ such that $e^{-\frac{2\pi i m \gamma_i}{\q}}=\zeta_N^{-\delta_i}.$ Then clearly $e^{-\frac{2\pi i k m \gamma_i}{\q}}=\zeta_N^{-k\delta_i}.$ 

On the other hand, if $e^{-\frac{2\pi i k m \gamma_i}{\q}}$ is a root of $D_{k\bm{\delta}},$ then there exists $i$ such that $e^{-\frac{2\pi i m k \gamma_i}{\q}}=\zeta_N^{-k\delta_i}.$ Since $k \in (\Z/\q\Z)^{\times},$ then there exists $a \in \Z$ such that $a k \equiv 1 \Mod{\q}.$ This means that $ak=1+b\q$ for some $b \in \Z.$ Therefore, \begin{align*}
e^{-\frac{2\pi i m \gamma_i}{\q}}=e^{-\frac{2 \pi i m}{\q}(1+b\q)\gamma_i}=e^{-\frac{2 \pi i m ak}{\q}\gamma_i}=\left(e^{-\frac{2 \pi i m}{\q}\gamma_i}\right)^a=e^{-\frac{2\pi i a k \delta_i}{N}}=(\zeta_N^{-k\delta_i})^a. 
\end{align*}
Now notice that since $N \mid \q,$ then $\gcd(N,k)=1$ and if we write $\q=lN,$ then $ak=1+b\q=1+blN.$ Therefore, \begin{align*}
e^{-\frac{2\pi i m \gamma_i}{\q}}=(\zeta_N^{-k\delta_i})^a=e^{-\frac{2 \pi i (1+blN)\delta_i}{N}}=e^{-\frac{2\pi i \delta_i}{N}}=\zeta_N^{-\delta_i}.    
\end{align*}
\end{proof}
\end{comment}
\begin{lemma}\label{lemma: galoisaction}
Let $k \in (\Z/\q\Z)^{\times},$ then $\sigma_k(F_q(\bm{\gamma},\bm{\delta},N \mid t))=F_q(\bm{\gamma},k\bm{\delta},N \mid t).$
\end{lemma}
\begin{proof}
A simple computation shows that $\sigma_k(\omega(x))=\omega^k(x)$ and $\sigma_k(g(m))=g(km).$ Therefore, by Lemma \ref{lemma: sdelta}  \begin{align*}
    \sigma_k(F_q(\bm{\gamma},\bm{\delta},N \mid t))&=-\tfrac{1}{\q}\sum_{m=0}^{q-2}\prod_{i=1}^{\mid\bm{\gamma}\mid}\frac{\sigma_k(g(-\gamma_i m+\delta_i\tfrac{\q}{N}))}{\sigma_k(g(\delta_i\frac{\q}{N}))}q^{s_{\bm{\delta}}(-m)-s_{\bm{\delta}}(0)}\sigma_k(\omega(\bm{\gamma}^{\bm{\gamma}}t))^m\\&=-\tfrac{1}{\q}\sum_{m=0}^{q-2}\prod_{i=1}^{\mid\bm{\gamma}\mid}\frac{g(-k\gamma_i m+k\delta_i\tfrac{\q}{N})}{g(k\delta_i\frac{\q}{N})}q^{s_{\bm{k\delta}}(-km)-s_{\bm{k\delta}}(0)}\omega(\bm{\gamma}^{\bm{\gamma}}t)^{km}.   
    \end{align*}
    Making the change of variable $m \mapsto km,$ we get \begin{align*}
    \sigma_k(F_q(\bm{\gamma},\bm{\delta},N \mid t))=-\tfrac{1}{\q}\sum_{m=0}^{q-2}\prod_{i=1}^{\mid\bm{\gamma}\mid}\frac{g(-\gamma_i m+k\delta_i\tfrac{\q}{N})}{g(k\delta_i\frac{\q}{N})}q^{s_{\bm{k\delta}}(-m)-s_{\bm{k\delta}}(0)}\omega(\bm{\gamma}^{\bm{\gamma}}t)^{m}=F_q(\bm{\gamma},k\bm{\delta},N \mid t).    
    \end{align*}  
\end{proof}

\begin{lemma}\label{lm: Fp}
$F_q(\bm{\gamma},p\bm{\delta},N \mid t)=F_q(\bm{\gamma},\bm{\delta},N \mid t^p).$   
\end{lemma}
\begin{proof}
Observe that using the notation of Remark \ref{rmk: gausssum}, we have \begin{equation}\label{eq: gp}
g_k(pm)=g_{k^p}(m), 
\end{equation} for every $m.$ Since Definition \ref{AA25} does not depend on the choice of additive character, we may suppose that $k=1.$ Therefore, \begin{align*}
F_q(\bm{\gamma},p\bm{\delta},N \mid t)&=-\tfrac{1}{\q}\sum_{m=0}^{q-2}\prod_{i=1}^{|\bm{\gamma}|}\frac{g(-\gamma_im+\tfrac{p\delta_i\q}{N})}{g(\tfrac{p\delta_i\q}{N})}q^{s_{p\bm{\delta}}(-m)-s_{p\bm{\delta}}(0)}\omega(\bm{\gamma^{\gamma}}t)^m\\&=-\tfrac{1}{\q}\sum_{m=0}^{q-2}\prod_{i=1}^{|\bm{\gamma}|}\frac{g(-\gamma_ipm+\tfrac{p\delta_i\q}{N})}{g(\tfrac{p\delta_i\q}{N})}q^{s_{p\bm{\delta}}(-pm)-s_{p\bm{\delta}}(0)}\omega(\bm{\gamma^{\gamma}}t)^{pm}\\&=-\tfrac{1}{\q}\sum_{m=0}^{q-2}\prod_{i=1}^{|\bm{\gamma}|}\frac{g(p(-\gamma_im+\tfrac{\delta_i\q}{N}))}{g(p\tfrac{\delta_i\q}{N})}q^{s_{p\bm{\delta}}(-pm)-s_{p\bm{\delta}}(0)}\omega(\bm{\gamma^{\gamma}}t)^{pm}\\&\stackrel{\eqref{eq: gp}}{=}-\tfrac{1}{\q}\sum_{m=0}^{q-2}\prod_{i=1}^{|\bm{\gamma}|}\frac{g(-\gamma_im+\tfrac{\delta_i\q}{N})}{g(\tfrac{\delta_i\q}{N})}q^{s_{p\bm{\delta}}(-pm)-s_{p\bm{\delta}}(0)}\omega(\bm{\gamma^{\gamma}}t)^{pm}\\&\stackrel{\text{Lemma } \ref{lemma: sdelta}}{=}-\tfrac{1}{\q}\sum_{m=0}^{q-2}\prod_{i=1}^{|\bm{\gamma}|}\frac{g(-\gamma_im+\tfrac{\delta_i\q}{N})}{g(\tfrac{\delta_i\q}{N})}q^{s_{\bm{\delta}}(-m)-s_{\bm{\delta}}(0)}\omega(\bm{\gamma}^{\bm{\gamma}}t^p)^m\\&=F_q(\bm{\gamma},\bm{\delta},N \mid t^p).
\end{align*}     
\end{proof}

\section{Number of points}\label{S:numPoints}\label{S:pointcounts} 
\label{S:counting}
The goal of this section is to compute the number of points of the five Delsarte pencils. We start with a classical formula due to Koblitz that allows one to express the point counts as products of Gauss sums. Then we explain the general step-by-step on how to apply the formula. Finally, we compute the point counts and match them with hypergeometric sums for each of the pencils, but with greater focus on $\Csf_4$ and $\Csf_2\Fsf_2.$

\subsection{Koblitz's formula}

Let $X$ be the hypersurface defined by the polynomial $$\sum_{i=1}^ra_ix_0^{v_{0i}}\cdots x_n^{v_{ni}} \in \F_q[x_0,\dots,x_n],$$ where $a_i \neq 0$ for at least one $i \in \{1,\dots,r\}.$ Let $G=(\mu_{\q})^r/\Delta,$ where $\mu_{\q}$ is the group of $\q$-roots of unity in $\C^{\times}$ and $\Delta$ is the diagonal of $(\mu_{q^\times})^r$.

Suppose that $\q$ does not divide any of the $v_{ij},$ $i \in \{1,\dots,r\}, j \in \{0,\dots,n\}.$ We have a bijective correspondence between multiplicative characters of $(\mu_{\q})^r/\Delta$ and the set $S$ of solutions of the following linear system: \begin{align}\label{sys:sys1}
\sum_{i=1}^rv_{ij}s_i \equiv 0 \Mod \q \text{ and } \sum_{i=1}^rs_i \equiv 0 \Mod \q \text{ for every }j=0,\dots,n,    
\end{align} given by taking a solution $s = (s_1, \dots, s_r)$ to the character $\chi_s$, defined by $\chi_s(x_1, \dots, x_r) = \omega(\prod_{=1}^r x_i^{s_i})$. 

We define $\omega(a)^{-s}=\omega(a_1^{-s_1}\cdots a_r^{-s_r})$ and \begin{align}\label{c:c}
 c_s&=\dfrac{(\q)^{n-r+1}}{q}\prod_{i=1}^rg(s_i), \text{ for } s \in S \setminus \{(0,\dots,0)\} \quad \\ \quad 
 c_0&=c_{(0,\dots,0)}=(\q)^{n-r+1}\dfrac{(\q)^{r-1}-(-1)^{r-1}}{q}.
\end{align}

Let $U$ denote the intersection of $X$ with the $n$-dimensional torus $\mathbb{G}_m$ in $\mathbb{P}^n.$ Then, Koblitz's Theorem \cite[Theorem 1]{K83} can be formulated as \begin{thm}\label{thm:Kob}
$$\#U(\F_q)=\sum_{s \in S}\omega(a)^{-s}c_s.$$     
\end{thm}
\begin{proof}
See \cite[Theorem 3.3.3]{HD20}{}.    
\end{proof} 

In the next section, we apply Theorem \ref{thm:Kob} to our five families to obtain a formula for their number of points over a finite field $\mathbb{F}_q.$

\subsection{Step by step}
\label{S: step}
In this section we explain how to apply Koblitz's formula (Theorem \ref{thm:Kob}) to obtain explicit point counts for each of our five families.

\subsubsection{Step 1: Computing the characters}
\label{subs: charac}

Given one of the invertible pencils, we first consider the matrix $R$ associated to it as described in \eqref{sys:sys1}. We compute the Smith Normal Form of $R$ which consists of a diagonal matrix $\text{SNF}(R)$ and two unimodular matrices $P,Q$ such that $PRQ=\text{SNF}(R).$ Still following the notation of the previous section, we write $Y=Q^{-1}(s_1,\dots,s_r)^T.$

Now the system \begin{align}\label{alg: algsnf}
    \text{SNF}(R)Y \equiv 0 \Mod{\q}
\end{align} is easily solved as $\text{SNF}(R)$ is a diagonal matrix. Moreover, \eqref{alg: algsnf} can be rewritten as $$PR(s_1,\dots,s_r)^T \equiv 0 \Mod{\q}$$ and we can finally extract the solutions $(s_1,\dots,s_r)^T$ of the original system in \eqref{sys:sys1}. 

\subsubsection{Step 2: Decomposing the set $S$}
\label{subs: decomp}
The set of solutions of \eqref{sys:sys1}, denoted by $S,$ can be decomposed into a finite disjoint union $$S=\bigcup_{i \in I}S_i,$$ where $S_i$ are subsets of $S$ whose elements depend on a unique parameter taking values in $\Z/\q\Z.$ 

\subsubsection{Step 3: Computing the point counts}
Koblitz's formula gives the number of points on the intersection of a hypersurface in $\mathbb{P}^n$ with $\mathbb{G}_m$. We now specialize to the case $n=3$ and a quartic hypersurface $X.$ We first compute the number of points on the intersection of the quartic $X \subset \mathbb{P}^3$ with $\mathbb{G}_m,$ and then use the formula several times to obtain the number of points on the intersection of $X$ with the smaller-dimensional tori of $\mathbb{P}^3$. Finally, we use the decomposition of $S$ mentioned above to break the point counts into smaller pieces $$\#U(\F_q)=\sum_{i \in I}\sum_{s \in S_i}\omega(a)^{-s}c_s.$$

\subsection{Explicit computations}
In this Section we treat $\Csf_4$ and $\Csf_2\Fsf_2$ in detail and briefly show the results for the other three pencils. The pencil $\Csf_4$ is a `warm-up' example and a good candidate to start with. Similarly to $\Csf_3\Fsf_1,$ it has only one hypergeometric function associated with the point counts, which simplifies the computations greatly. The family $\Csf_2\Fsf_2$ illustrates well the use of our new techniques. Along with $\Csf_2\Csf_2,$ it differs from the other three families as its point counts cannot be described only in terms of hypergeometric sums defined over $\Q.$ Throughout this section, $\psi \in \F_q$ is the deformation parameter as defined in \eqref{eq: pencil}.

\subsubsection[]{The family $\Csf_4$} We recall that $\Csf_4$ is the pencil defined by the one-parameter deformation of the invertible polynomial $x_0^3x_1+x_1^3x_2+x_2^3x_3+x_3^4.$ Following \textit{Step 1} described in Section \ref{subs: charac}, the system in \eqref{sys:sys1} is given by the following equation:
\begin{align*}
\begin{pmatrix}
3 & 0 & 0 &0& 1\\
1 & 3 & 0 &0 & 1\\
0 & 1 & 3 &0 & 1\\
0 & 0 & 1 &4 & 1\\
1 & 1 & 1 & 1 &1
\end{pmatrix}\begin{pmatrix}
s_1 \\
s_2 \\
s_3 \\
s_4 \\
s_5
\end{pmatrix}\equiv 0 \Mod \q.
\end{align*} 
Notice that the matrix on the left-hand side consists of the transpose $A^T$ of the matrix $A$ associated to the monomials in the pencil augmented by a row of 1s. The row of 1s corresponds to the second condition in \eqref{sys:sys1}, which asks for the sum of the $s_i$ to be equivalent to 0 modulo $q^\times$.

We obtain the solutions $$S=\{(-9,-6,-7,-5,27)m : m \in \Z/\q\Z\},$$ and since the set depends on a unique parameter, it is already the decomposition mentioned in \textit{Step 2}. To proceed with \textit{Step 3}, we prove two lemmas. In Lemma \ref{lemma: lemmac41} we find the point counts on the maximal torus, and in Lemma \ref{lemma: lemmac4} we find the point counts on the smaller tori. Summing the results of Lemmas \ref{lemma: lemmac41} and \ref{lemma: lemmac4} leads to the complete point count in Proposition \ref{prop: propc4}.

\begin{lemma}\label{lemma: lemmac41}
\begin{itemize}
    \item[(a)] If $q\not \equiv 1 \Mod{3},$ then $\#U_{\psi}(\F_q)=q^2-3q+3+H_q(\bm{\alpha}_{\heartsuit_0},\bm{\beta}_{\heartsuit_0}|t_{\heartsuit_0}).$
    \item[(b)] If $q \equiv 1 \Mod{3}$ and $q \not\equiv 1 \Mod{9},$ then $\#U_{\psi}(\F_q)=q^2-3q+5+H_q(\bm{\alpha}_{\heartsuit_0},\bm{\beta}_{\heartsuit_0}|t_{\heartsuit_0}).$
    \item[(c)] If $q \equiv 1 \Mod{9},$ then \begin{align*}
        \#U_{\psi}(\F_q)&=q^2-3q+5+H_q(\bm{\alpha}_{\heartsuit_0},\bm{\beta}_{\heartsuit_0}|t_{\heartsuit_0})-\sum_{i \in (\Z/9\Z)^{\times}}J_q(\tfrac{i}{9},\tfrac{2i}{9}).
    \end{align*} 
\end{itemize}    
\end{lemma}
\begin{proof}
Let $m \in \Z/\q\Z.$ The coefficients in (\ref{c:c}) are given by \begin{align*}
c_{(0,0,0,0,0)}&=\dfrac{q^3-4q^2+6q-4}{\q}.\\
c_{(-9m,-6m,-7m,-5m,27m)}&=\frac{1}{q\q}g(-9m)g(-6m)g(-7m)g(-5m)g(27m), \text{ for } m \neq 0.
\end{align*}

By Theorem \ref{thm:Kob}, 

\begin{equation}\label{eq:torus}
\#U_{\psi}(\F_q)=\sum_{s \in S}\omega(a)^{-s}c_s=\dfrac{q^3-4q^2+6q-4}{\q}+\frac{1}{q\q}\sum_{m=1}^{q-2}g(-9m)g(-6m)g(-7m)g(-5m)g(27m)\omega(-27\psi)^{-27m}. 
\end{equation}
On the other hand, using Definition \ref{def: def2}, we can find hypergeometric parameters corresponding to the same point count in the following way. First, consider the quotient of polynomials determined by $S$ as follows: $$\frac{x^{27}-1}{(x^9-1)(x^6-1)(x^7-1)(x^5-1)}.$$ Notice that the gcd between the numerator and denominator is $D(x)=x^9-1.$ Let  $$\Tilde{\mathcal{B}}:=\{m \in \Q \cap [0,1]: D(e^{2\pi i m/\q})=0\}=\Big\{0,\tfrac{\q}{3},\tfrac{2\q}{3},\tfrac{\q}{9},\tfrac{2\q}{9},\tfrac{4\q}{9},\tfrac{5\q}{9},\tfrac{7\q}{9},\tfrac{8\q}{9}\Big\}.$$ Moreover, as $D(x)$ is separable over $\C$, we have $s(m)=1$ for all $m \in \Tilde{\mathcal{B}}$ . Also note that $M=\frac{27^{27}}{9^96^67^75^5}$ and $\epsilon=-1.$   We proceed once again with the three cases based on $q$ modulo $3$ and $9$.

\begin{itemize}
    \item[(a)] If $q \not\equiv 1 \Mod{3},$ then $\mathcal{B}:=\{m \in \Z/\q\Z: s(m)>0\}=\{0\}$ and we have 
\begin{align}\label{eq: hp}
    H_q(\bm{\alpha}_{\heartsuit_0},\bm{\beta}_{\heartsuit_0}|t_{\heartsuit_0})&=\frac{1}{\q}\sum_{m=0}^{q-2}q^{-1+s(m)}g(-9m)g(-6m)g(-7m)g(-5m)g(27m)\omega(-27\psi)^{-27m}\\&=\frac{1}{\q}(g(0)^5+\frac{1}{q}\sum_{m=1}^{q-2}g(-9m)g(-6m)g(-7m)g(-5m)g(27m)\omega(-27\psi)^{-27m})\notag\\&=-\frac{1}{\q}+\frac{1}{q\q}\sum_{m=1}^{q-2}g(-9m)g(-6m)g(-7m)g(-5m)g(27m)\omega(-27\psi)^{-27m}\notag.
\end{align} 
    Comparing \eqref{eq:torus} and \eqref{eq: hp}, one can see that $$\#U_{\psi}(\F_q)=q^2-3q+3+H_q(\bm{\alpha}_{\heartsuit_0},\bm{\beta}_{\heartsuit_0}|t_{\heartsuit_0}).$$
    \item[(b)] If $q \equiv 1 \Mod{3}$ and $q \not \equiv 1 \Mod{9},$ then $\mathcal{B}=\Big\{0,\frac{\q}{3},\frac{2\q}{3}\Big\}.$ The special coefficients are \begin{align*}
    c_{\frac{\q}{3}}=c_{\frac{2\q}{3}}=\frac{1}{q\q }g(0)^3g(\tfrac{\q}{3})g(\tfrac{2\q}{3})=\frac{(-1)^{\frac{\q}{3}}}{\q}=\frac{-1}{\q}.   
    \end{align*} Therefore, \eqref{eq:torus} takes the form \begin{align}\label{eq: torus2}
    \sum_{s \in S}\omega(a)^{-s}=\frac{q^3-4q^2+6q-4}{\q}-\frac{2q}{q\q}+\frac{1}{q\q}\sum_{m \notin \mathcal{B}}g(-9m)g(-6m)g(-7m)g(-5m)g(27m)\omega(-27\psi)^{-27m}. 
    \end{align}
The hypergeometric sum is given by \begin{align}\label{eq: hp2}
H_q(\bm{\alpha}_{\heartsuit_0},\bm{\beta}_{\heartsuit_0}|t_{\heartsuit_0})&=\frac{1}{\q}(g(0)^5+2g(0)^3g(\tfrac{\q}{3})g(\tfrac{2\q}{3}))\\&+\frac{1}{q\q}\sum_{m \notin \mathcal{B}}g(-9m)g(-6m)g(-7m)g(-5m)g(27m)\omega(-27\psi)^{-27m} \notag\\&=-\frac{1}{\q}-\frac{2}{\q}(-1)^{\frac{\q}{3}}q+\frac{1}{q\q}\sum_{m \notin \mathcal{B}}g(-9m)g(-6m)g(-7m)g(-5m)g(27m)\omega(-27\psi)^{-27m}\notag\\&=-\frac{1}{\q}-\frac{2q}{\q}+\frac{1}{q\q}\sum_{m \notin \mathcal{B}}g(-9m)g(-6m)g(-7m)g(-5m)g(27m)\omega(-27\psi)^{-27m}\notag.
\end{align}

Comparing \eqref{eq:torus} and \eqref{eq: hp2} we get \begin{align*}
\#U_{\psi}(\F_q)=q^2-3q+5+H_q(\bm{\alpha}_{\heartsuit_0},\bm{\beta}_{\heartsuit_0}|t_{\heartsuit_0}). 
\end{align*}
\item[(c)] If $q \equiv 1 \Mod{9},$ then $\mathcal{B}=\Big\{0,\tfrac{\q}{3},\tfrac{2\q}{3},\tfrac{\q}{9},\tfrac{2\q}{9},\tfrac{4\q}{9},\tfrac{5\q}{9},\tfrac{7\q}{9},\tfrac{8\q}{9}\Big\}$ and in an analogous way as done for the two previous items, one obtains that \begin{align}
\#U_{\psi}(\F_q)&=q^2-3q+5+H_q(\bm{\alpha}_{\heartsuit_0},\bm{\beta}_{\heartsuit_0}|t_{\heartsuit_0})\\&-\frac{1}{q}\left(g(\tfrac{2\q}{3})g(\tfrac{7\q}{9})g(\tfrac{5\q}{9})+g(\tfrac{\q}{3})g(\tfrac{5\q}{9})g(\tfrac{\q}{9})+g(\tfrac{2\q}{3})g(\tfrac{2\q}{9})g(\tfrac{\q}{9})\right)\notag\\&-\frac{1}{q}\left(g(\tfrac{\q}{3})g(\tfrac{7\q}{9})g(\tfrac{8\q}{9})+g(\tfrac{2\q}{3})g(\tfrac{4\q}{9})g(\tfrac{8\q}{9})+g(\tfrac{\q}{3})g(\tfrac{2\q}{9})g(\tfrac{4\q}{9})\right)\notag.    
\end{align}    

Recalling Remark \ref{rmk: notationJac}, we write the last expression in terms of Jacobi sums. By Lemma \ref{lemma: JacobiGauss}, for every $i \in (\Z/9\Z)^{\times},$ we have \begin{align*}
J_q(\tfrac{i}{9},\tfrac{2i}{9})=\tfrac{1}{q}g(\tfrac{i\q}{9})g(\tfrac{2i\q}{9})g(\tfrac{-i\q}{3}).    
\end{align*} For instance, if $i=1$ we get $J_q(\tfrac{1}{9},\tfrac{2}{9})=\tfrac{1}{q}g(\tfrac{\q}{9})g(\tfrac{2\q}{9})g(\tfrac{2\q}{3}).$ Doing the same for every element of $(\Z/9\Z)^{\times},$ we prove the result.  
\end{itemize} 
\end{proof}
Now that we have the point counts for the maximal torus of $C_4$, we treat the smaller tori. 
\begin{lemma}\label{lemma: lemmac4}
\begin{itemize}
    \item[(a)] If $q \not\equiv 1 \Mod{3},$ then $\#X_{\psi}(\F_q)-\#U_{\psi}(\F_q)=5q-2.$
    \item[(b)] If $q \equiv 1 \Mod{3}$ and $q \not\equiv 1 \Mod{9},$ then $\#X_{\psi}(\F_q)-\#U_{\psi}(\F_q)=7q-4.$
    \item[(c)] If $q \equiv 1 \Mod 9,$ then \begin{align*}
    \#X_{\psi}(\F_q)-\#U_{\psi}(\F_q)=7q-4+\sum_{i \in (\Z/9\Z)^{\times}}J_q(\tfrac{i}{9},\tfrac{2i}{9}).  
    \end{align*} 
\end{itemize}    
\end{lemma}

\begin{proof}
\begin{itemize}
\item[(a)] Suppose that $q \not \equiv 1 \Mod{3}.$ 

\bigskip 

\textbf{One coordinate is zero and all the others are non-zero} 

\bigskip 

The first case is $x_0=0.$ The equation becomes $x_1^3x_2+x_2^3x_3+x_3^4=0$ and using \eqref{sys:sys1} is associated to the system
\begin{align*}
\begin{pmatrix}
3 & 0 & 0 \\
1 & 3 & 0 \\
0 & 1 & 4 \\
1 & 1 & 1 
\end{pmatrix}\begin{pmatrix}
s_1 \\
s_2 \\
s_3
\end{pmatrix}\equiv 0 \Mod \q.
\end{align*} This system has a unique solution given by $(0,0,0)$ and therefore the number of points is $q-2.$ The other three cases differ from the first one, but are similar between them and contribute each with $q-1$ points. To illustrate, consider $x_1=0,$ then the equation becomes $x_2^3x_3+x_3^4=0$ and we have the following system \begin{align*}
\begin{pmatrix}
 0 & 0 \\
 3 & 0 \\
 1 & 4 \\
 1 & 1 
\end{pmatrix}\begin{pmatrix}
s_1 \\
s_2 
\end{pmatrix}\equiv 0 \Mod \q
\end{align*} which has $(0,0)$ as its unique solution and therefore the number of points is $q-1.$ 

\bigskip

\textbf{Two coordinates are zero and the other two are non-zero} 

\bigskip 

We only have solutions in two of the cases, namely $x_0=x_1=0$ and $x_1=x_3=0.$ In the first case, the equation becomes $x_2^3x_3+x_3^4=0$ and is associated to the matrix \begin{align*}
\begin{pmatrix}
 3 & 0 \\
 1 & 4 \\
 1 & 1 
\end{pmatrix}\begin{pmatrix}
s_1 \\
s_2 
\end{pmatrix}\equiv 0 \Mod{\q}.
\end{align*} The only solution for this system is $(0,0)$ and therefore the number of points is $c_{(0,0)}=1.$ In the second case, the equation is trivial and therefore every element $(a:b) \in \mathbb{P}^1 \setminus \{(0:1),(1:0)\}$ is a solution, so there are $q-1$ points.
\bigskip

\textbf{Three coordinates are zero and the other one is non-zero} 

\bigskip

\noindent We have three points satisfying this condition: $(0:0:1:0),(0:1:0:0) \text{ and }(1:0:0:0).$

\bigskip

We conclude that $$\#X_{\psi}(\F_q)-\#U_{\psi}(\F_q)=5q-2.$$

\bigskip 

\item[(b)] Suppose $q \equiv 1 \Mod{3}$ and $q \not \equiv 1 \Mod{9}.$ 

\bigskip

\textbf{One coordinate is zero and all the others are non-zero}

\bigskip

\noindent The first case is $x_0=0,$ which has equation $x_1^3x_2+x_2^3x_3+x_3^4=0.$ It is associated to the system \begin{align*}
\begin{pmatrix}
3 & 0 & 0 \\
1 & 3 & 0 \\
0 & 1 & 4 \\
1 & 1 & 1 
\end{pmatrix}\begin{pmatrix}
s_1 \\
s_2 \\
s_3
\end{pmatrix}\equiv 0 \Mod \q
\end{align*} which has solutions given by $\{(0,0,0),\tfrac{\q}{3}(0,1,2),\tfrac{\q}{3}(0,2,1)\}.$ The number of points is then given by $$c_{(0,0,0)}+c_{\tfrac{\q}{3}(0,1,2)}+c_{\tfrac{\q}{3}(0,2,1)}=q-4.$$ The second case is $x_1=0$ and the corresponding equation is $x_2^3x_3+x_3^4=0.$ So the system is \begin{align*}
\begin{pmatrix}
 0 & 0 \\
 3 & 0 \\
 1 & 4 \\
 1 & 1 
\end{pmatrix}\begin{pmatrix}
s_1 \\
s_2 
\end{pmatrix}\equiv 0 \Mod \q
\end{align*} and has solutions $S=\{(0,0),\tfrac{\q}{3}(1,2),\tfrac{\q}{3}(2,1)\}$ giving then that the number of points is $$c_{(0,0)}+c_{\tfrac{\q}{3}(1,2)}+c_{\tfrac{\q}{3}(2,1)}=3q-3.$$ The other two cases, namely $x_2=0$ and $x_3=0$ are similar to item (a) and give each $q-1$ points. 

\bigskip

\textbf{Two coordinates are zero and the other two are non-zero} 

\bigskip 

As in item (a), only two of the cases have solutions. The first one is $x_0=x_1=0.$ The equation becomes $x_2^3x_3+x_3^4=0$ and the system \begin{align*}
\begin{pmatrix}
 3 & 0 \\
 1 & 4 \\
 1 & 1 
\end{pmatrix}\begin{pmatrix}
s_1 \\
s_2 
\end{pmatrix}\equiv 0 \Mod \q
\end{align*} has solutions $\{(0,0),\tfrac{\q}{3}(1,2),\tfrac{\q}{3}(2,1)\}.$ Therefore, the number of points is $$c_{(0,0)}+c_{\tfrac{\q}{3}(1,2)}+c_{\tfrac{\q}{3}(2,1)}=3.$$ The second case that has solutions is $x_1=x_3=0$ which again has trivial equation and therefore yields $q-1$ points.

\bigskip

\textbf{Three coordinates are zero and the other one is non-zero.} 

\bigskip 

\noindent As in item (a), we have the same 3 points. 

\bigskip

Therefore, $$\#X_{\psi}(\F_q)-\#U_{\psi}(\F_q)=7q-4.$$

\bigskip

\item[(c)] Suppose that $q \equiv 1 \Mod{9}.$

\bigskip

\textbf{One coordinate is zero and all the others are non-zero} 

\bigskip 
    
\noindent The first case is $x_0=0.$ The equation becomes $x_1^3x_2+x_2^3x_3+x_3^4=0$ and is associated with the system \begin{align*}
\begin{pmatrix}
3 & 0 & 0 \\
1 & 3 & 0 \\
0 & 1 & 4 \\
1 & 1 & 1 
\end{pmatrix}\begin{pmatrix}
s_1 \\
s_2 \\
s_3
\end{pmatrix}\equiv 0 \Mod{\q}.
\end{align*} whose solutions are \begin{align*}
S=\{(0,0,0),\tfrac{\q}{3}(0,1,2),\tfrac{\q}{3}(0,2,1),\tfrac{\q}{9}(3,5,1),\tfrac{\q}{9}(6,1,2),\tfrac{\q}{9}(3,2,4),\tfrac{\q}{9}(6,7,5),\tfrac{\q}{9}(3,8,7),\tfrac{\q}{9}(6,4,8)\}    
\end{align*} and the the number of points is \begin{align*}
&q-4+\frac{1}{q}\left(g(\tfrac{2\q}{3})g(\tfrac{7\q}{9})g(\tfrac{5\q}{9})+g(\tfrac{\q}{3})g(\tfrac{5\q}{9})g(\tfrac{\q}{9})+g(\tfrac{2\q}{3})g(\tfrac{2\q}{9})g(\tfrac{\q}{9})\right)\notag\\&+\frac{1}{q}\left(g(\tfrac{\q}{3})g(\tfrac{7\q}{9})g(\tfrac{8\q}{9})+g(\tfrac{2\q}{3})g(\tfrac{4\q}{9})g(\tfrac{8\q}{9})+g(\tfrac{\q}{3})g(\tfrac{2\q}{9})g(\tfrac{4\q}{9})\right).    
\end{align*} The second case is $x_1=0$ and has equation $x_2^3x_3+x_3^4=0.$ The associate system is \begin{align*}
\begin{pmatrix}
 0 & 0 \\
 3 & 0 \\
 1 & 4 \\
 1 & 1 
\end{pmatrix}\begin{pmatrix}
s_1 \\
s_2 
\end{pmatrix}\equiv 0 \Mod{\q},
\end{align*} with solutions $S=\{(0,0),\tfrac{\q}{3}(1,2),\tfrac{\q}{3}(2,1)\}.$ Thus, the number of points is $c_{(0,0)}+c_{\tfrac{\q}{3}(1,2)}+c_{\tfrac{\q}{3}(2,1)}=3q-3.$ 

The last two cases are similar and give the same number of points. We consider only one of the cases, namely $x_2=0.$ The equation is $x_0^3x_1+x_3^4=0$ and has associated system of the form \begin{align*}
\begin{pmatrix}
 3 & 0 \\
 1 & 0 \\
 0 & 4 \\
 1 & 1 
\end{pmatrix}\begin{pmatrix}
s_1 \\
s_2 
\end{pmatrix}\equiv 0 \Mod{\q},
\end{align*} and only $(0,0)$ as solution and therefore the number of points is $q-1.$ 

\bigskip

\textbf{Two coordinates are zero and the other two are non-zero} 

\bigskip 

Only two cases have solutions. The first one is $x_0=x_1=0$ whose equation has the form $x_2^3x_3+x_3^4=0$ and the associated system is given by \begin{align*}
\begin{pmatrix}
 3 & 0 \\
 1 & 4 \\
 1 & 1 
\end{pmatrix}\begin{pmatrix}
s_1 \\
s_2 
\end{pmatrix}\equiv 0 \Mod \q
\end{align*} Its solutions are $S=\{(0,0),\tfrac{\q}{3}(1,2),\tfrac{\q}{3}(2,1)\}$ and the number of points is $c_{(0,0)}+c_{\tfrac{\q}{3}(1,2)}+c_{\tfrac{\q}{3}(2,1)}=3.$ The second case is $x_1=x_3=0,$ whose equation is trivial and therefore every $(a:b) \in \mathbb{P}^1 \setminus \{(0:1),(1:0)\}$ is a solution. Thus, there are $q-1$ points.

\bigskip

\textbf{Three coordinates are zero and the other one is non-zero} 

\bigskip

\noindent 
There are three points satisfying this condition, namely $(0:0:1:0),(0:1:0:0) \text{ and } (1:0:0:0).$

\bigskip 

Therefore, we obtain the formula 
\begin{align*}
    \#X_{\psi}(\F_q)-\#U_{\psi}(\F_q)=7q-4&+\frac{1}{q}\left(g(\tfrac{2\q}{3})g(\tfrac{7\q}{9})g(\tfrac{5\q}{9})+g(\tfrac{\q}{3})g(\tfrac{5\q}{9})g(\tfrac{\q}{9})+g(\tfrac{2\q}{3})g(\tfrac{2\q}{9})g(\tfrac{\q}{9})\right)\notag\\&+\frac{1}{q}\left(g(\tfrac{\q}{3})g(\tfrac{7\q}{9})g(\tfrac{8\q}{9})+g(\tfrac{2\q}{3})g(\tfrac{4\q}{9})g(\tfrac{8\q}{9})+g(\tfrac{\q}{3})g(\tfrac{2\q}{9})g(\tfrac{4\q}{9})\right)\\&=7q-4+\sum_{i \in (\Z/9\Z)^{\times}}J_q(\tfrac{i}{9},\tfrac{2i}{9}),  
\end{align*} where the conversion between products of Gauss sums and Jacobi sums is obtained exactly as in Lemma \ref{lemma: lemmac41}.
\end{itemize}
\end{proof}

For a congruence relation $a \equiv b \Mod{c},$ we denote $$\delta[a\equiv b \Mod c]=\begin{cases}
1, \text{ if } a\equiv b \Mod c\\
0, \text{ otherwise}
\end{cases}.$$ With this notation and combining Lemma \ref{lemma: lemmac41} and Lemma \ref{lemma: lemmac4}, we prove the following result

\begin{prop}\label{prop: propc4} The pair $\bm{\alpha}_{\heartsuit_0},\bm{\beta}_{\heartsuit_0}$ is defined over $\Q$ and
$$\#X_{\psi}(\F_q)=q^2+2q+1+2q\delta[q\equiv 1 \Mod 3]+H_q(\bm{\alpha}_{\heartsuit_0},\bm{\beta}_{\heartsuit_0} \mid t_{\heartsuit_0}).$$ 
\end{prop}

\subsubsection[]{The family $\Csf_2\Fsf_2$} 
Recall that $\Csf_2\Fsf_2$ is defined as the one-parameter deformation of the invertible polynomial $x_0^3x_1+x_1^4+x_2^4+x_3^4.$ Following \textit{Step 1} and \textit{2} described in Section~\ref{subs: charac}, we show that the set of solutions of \eqref{sys:sys1} decomposes as follows.

\bigskip 

If $q \equiv 1 \Mod 4,$ then we have a disjoint union $S=S_0 \cup S_1 \cup S_2 \cup S_3,$ where 

\bigskip
    
\begin{itemize}
    \item[(i)] $S_0=\Bigg\{(4,2,3,3,-12)l:l \in \Z/\q\Z\Bigg\}$
    \item[(ii)] $S_1=\Bigg\{(4,2,3,3,-12)l+\dfrac{\q}{2}(0,1,0,1,0):l \in \Z/\q\Z\Bigg\}$
    \item[(iii)] $S_2=\Bigg\{(4,2,3,3,-12)l+\dfrac{\q}{4}(0,3,0,1,0):l \in \Z/\q\Z\Bigg\}$
    \item[(iv)] $S_3=\Bigg\{(4,2,3,3,-12)l+\dfrac{\q}{4}(0,1,0,3,0): l \in \Z/\q\Z\Bigg\}$
\end{itemize}  

\bigskip 

If $q \equiv 3 \Mod 4,$ then we have a disjoint union $S=S_0 \cup S_1,$ where

\bigskip 

\begin{itemize}
    \item[(i)] $S_0=\Bigg\{(4,2,3,3,-12)l:l \in \Z/\q\Z\Bigg\}$
    \item[(ii)] $S_1=\Bigg\{(4,2,3,3,-12)l+\dfrac{\q}{2}(0,1,0,1,0):l \in \Z/\q\Z\Bigg\}$
\end{itemize}

\bigskip 

Here \textit{Step 3} (see \ref{S: step}) consists of Lemmas \ref{lm: thelemma}, \ref{lemma: lemma2}, \ref{lemma:club2}, and \ref{lemma:club3} whereas Lemmas \ref{lem: junk1} and \ref{lemma:char} are used in the proof of Lemma \ref{lemma: lemma2}.

\begin{lemma}\label{lm: thelemma} The pair $\bm{\alpha}_{\clubsuit_0},\bm{\beta}_{\clubsuit_0}$ is defined over $\Q$ and 
\begin{itemize}
    \item [\textbf{(a)}] If $q \not \equiv 1 \Mod{3}$ and $q \not \equiv 1 \Mod{4},$ then $$\sum_{s \in S_0}\omega(a)^{-s}c_s=q^2-3q+2+H_q(\bm{\alpha}_{\clubsuit_0},\bm{\beta}_{\clubsuit_0}\mid t_{\clubsuit_0}).$$
    \item [\textbf{(b)}] If $q \not \equiv 1 \Mod{3}$ and $q \equiv 1 \Mod{4},$ then $$\sum_{s \in S_0}\omega(a)^{-s}c_s=q^2-3q+4+H_q(\bm{\alpha}_{\clubsuit_0},\bm{\beta}_{\clubsuit_0}\mid t_{\clubsuit_0})-\left(J_q\left(\tfrac{1}{4},\tfrac{1}{4}\right)+J_q\left(\tfrac{3}{4},\tfrac{3}{4}\right)\right).$$
    \item [\textbf{(c)}] If $q \equiv 1 \Mod{3}$ and $q \not\equiv 1 \Mod{4},$ then $$\sum_{s \in S_0}\omega(a)^{-s}c_s=q^2-3q+4+H_q\left(\bm{\alpha}_{\clubsuit_0},\bm{\beta}_{\clubsuit_0}\mid t_{\clubsuit_0}\right)$$
    \item[\textbf{(d)}] If $q \equiv 1 \Mod{3}$ and $q \equiv 1 \Mod{4},$ then $$\sum_{s \in S_0}\omega(a)^{-s}c_s=q^2-3q+6+H_q\left(\bm{\alpha}_{\clubsuit_0},\bm{\beta}_{\clubsuit_0}\mid t_{\clubsuit_0}\right)-\left(J_q\left(\tfrac{1}{4},\tfrac{1}{4}\right)+J_q\left(\tfrac{3}{4},\tfrac{3}{4}\right)\right).$$
\end{itemize} 
\end{lemma}

\begin{proof}
Let $q_1=4, q_2=2, q_3=q_4=3$ and $p_1=12.$ We have $$\dfrac{x^{12}-1}{(x^4-1)(x^2-1)(x^3-1)^2}=\dfrac{(x-1)^3(x+1)(x-e^{\frac{2\pi i}{3}})(x-e^{\frac{4\pi i}{3}})}{(x-e^{\frac{\pi i}{6}})(x-e^{\frac{\pi i}{3}})(x-e^{\frac{5\pi i}{6}})(x-e^{\frac{7\pi i}{6}})(x-e^{\frac{5\pi i}{3}})(x-e^{\frac{11\pi i}{6}})}$$ and the gcd between numerator and denominator is $$D(x)=(x^4-1)(x-e^{\frac{2\pi i}{3}})(x-e^{\frac{4\pi i}{3}})=(x-1)(x-e^{\frac{\pi i}{2}})(x-e^{\pi i})(x-e^{\frac{3\pi i}{2}})(x-e^{\frac{2\pi i}{3}})(x-e^{\frac{4\pi i}{3}}).$$ Moreover, we have $K_{\bm{\alpha_2},\bm{\beta_2}}=\Q,$ which means $\bm{\alpha_2},\bm{\beta_2}$ are defined over $\Q.$ 

Notice that if $m \in \Z/\q\Z,$ then $s(m)>0$ if and only if $e^{2\pi i \frac{m}{\q}}$ is a root of $D(x)$ and since $D(x)$ is separable over $\C$, $s(m)=1$ in the affirmative case. We separate in 4 cases depending on whether or not $4$ and $3$ divide $\q$ as the numbers $s(m)$ will depend on these conditions. For each of the cases, for $m \in \{0,\dots,q-2\}$ $$s(m)=\begin{cases}
1, \text{ if } m \in \mathcal{B} \\
0, \text{ otherwise}
\end{cases}.$$ 

We also have \begin{align}\label{sum: sum}\sum_{s \in S_0}\omega(a)^{-s}c_s&=\sum_{m=0}^{q-2}\omega(-12\psi)^{-12m}c_s, \text{ where } \\ 
c_{(0,0,0,0,0)}&=\dfrac{q^3-4q^2+6q-4}{\q} \\
c_{(4,2,3,3,-12)m}&=\dfrac{1}{q\q}g(-4m)g(-2m)g(-3m)^2g(12m), \text{ for all }m\neq 0.
\end{align}

We compute below all the special values of $c_s$ as they will be used along the proof. For $m=\frac{\q}{2},$ we get $$c_{(\frac{\q}{2})(0,0,1,1,0)}=\dfrac{-1}{q\q}g(\tfrac{\q}{2})^2=\frac{(-1)^{\frac{\q}{2}+1}}{\q}.$$ 
For $m=\frac{\q}{3},\frac{2\q}{3},$ we have $$c_{(\frac{\q}{3})(1,-1,0,0,0)}=c_{(\frac{\q}{3})(-1,1,0,0,0)}=\dfrac{-1}{q\q}g(\tfrac{\q}{3})g(-\tfrac{\q}{3})=\frac{(-1)^{\frac{\q}{3}+1}}{\q}.$$
And finally, for $m=\frac{\q}{4},\frac{3\q}{4},$ we have $$c_{(\frac{\q}{4})(0,2,-1,-1,0)}=c_{(\frac{\q}{4})(0,2,1,1,0)}=\frac{1}{q\q}g(\tfrac{\q}{2})g(-\tfrac{\q}{4})^2=\frac{(-1)^{\frac{\q}{2}}}{\q}J_q\left(\tfrac{3}{4},\tfrac{3}{4}\right)$$

\begin{itemize}
    \item[\textbf{(a)}] Let $\mathcal{B}=\{0,\frac{\q}{2}\}.$  We have 
\begin{align}\label{eq: eq19}
H_q(\bm{\alpha}_{\clubsuit_0},\bm{\beta}_{\clubsuit_0}\mid t_{\clubsuit_0})&=\dfrac{1}{\q}\sum_{m=0}^{q-2}q^{-1+s(m)}g(-4m)g(-2m)g(-3m)^2g(12m)\omega(-12\psi)^{-12m}\\
    &=1+\dfrac{1}{q\q}\sum_{m \notin \mathcal{B}}g(-4m)g(-2m)g(-3m)^2g(12m)\omega(-12\psi)^{-12m}\notag\end{align} 
\noindent We use the special values of $c_s$ for $m \in \mathcal{B}$ to pull the corresponding terms out of the sum in \eqref{sum: sum}. We also recall that $\frac{\q}{2}$ is odd, so \begin{align*}
    \sum_{s \in S_0}\omega(a)^{-s}c_s&=q^2-3q+3+\dfrac{1}{q\q}\sum_{m \notin \mathcal{B}}g(-4m)g(-2m)g(-3m)^2g(12m)\omega(-12\psi)^{-12m}.
\end{align*}
Isolating the sum in equation \eqref{eq: eq19} and substituting in the equation above we obtain the final result

$$\sum_{s \in S_0}\omega(a)^{-s}c_s=q^2-3q+2+H_q(\bm{\alpha}_{\clubsuit_0},\bm{\beta}_{\clubsuit_0}\mid t_{\clubsuit_0}).$$

\item[\textbf{(b)}] Let $\mathcal{B}=\{0,\frac{\q}{2},\frac{\q}{4},\frac{3\q}{4}\}.$ We have \begin{align}\label{eq: 20}
   H_q(\bm{\alpha}_{\clubsuit_0},\bm{\beta}_{\clubsuit_0}\mid t_{\clubsuit_0})&=\frac{1}{\q}(-1-q+g(\tfrac{\q}{2})(g(\tfrac{\q}{4})^2+g(-\tfrac{\q}{4})^2))\notag\\&+\dfrac{1}{q\q}\sum_{m \notin \mathcal{B}}g(-4m)g(-2m)g(-3m)^2g(12m)\omega(-12\psi)^{-12m}.
\end{align}

We use the special values of $c_s$ for $m \in \mathcal{B}$ to pull the corresponding terms out of the sum in \eqref{sum: sum}. We also recall that $\frac{\q}{2}$ is even, then \begin{align*}
    \sum_{s \in S_0}\omega(a)^{-s}c_s&=\frac{q^3-4q^2+6q-5}{\q}+\frac{g(\tfrac{\q}{2})}{q\q}(g(\tfrac{\q}{4})^2+g(-\tfrac{\q}{4})^2)\\&+\sum_{m \notin \mathcal{B}}g(-4m)g(-2m)g(-3m)^2g(12m)\omega(-12\psi)^{-12m}
\end{align*}

Isolating the sum in equation \eqref{eq: 20} and plugging into the equation above, we obtain $$\sum_{s \in S_0}\omega(a)^{-s}c_s=q^2-3q+4+H_q(\bm{\alpha}_{\clubsuit_0},\bm{\beta}_{\clubsuit_0}\mid t_{\clubsuit_0})-\tfrac{1}{q}g(\tfrac{\q}{2})(g(\tfrac{\q}{4})^2+g(-\tfrac{\q}{4})^2).$$

By Lemma \ref{lemma: JacobiGauss}, we have $$J_q\left(\tfrac{1}{4},\tfrac{1}{4}\right)+J_q\left(\tfrac{3}{4},\tfrac{3}{4}\right)=\tfrac{1}{q}g(\tfrac{\q}{2})(g(\tfrac{\q}{4})^2+g(-\tfrac{\q}{4})^2)$$ and this concludes the proof of \textbf{(b)}.

\item[\textbf{(c)}] Let $\mathcal{B}=\Big\{0,\frac{\q}{2},\frac{\q}{3},\frac{2\q}{3}\Big\}.$ We have \begin{align}\label{eq: 23}
H_q(\bm{\alpha}_{\clubsuit_0},\bm{\beta}_{\clubsuit_0}\mid t_{\clubsuit_0})=-\frac{q+1}{\q}+\frac{1}{q\q}\sum_{m \notin \mathcal{B}}g(-4m)g(-2m)g(-3m)^2g(12m)\omega(-12\psi)^{-12m}
\end{align}

We use the special values of $c_s$ for $m \in \mathcal{B}$ to pull the corresponding terms out of the sum in \eqref{sum: sum}. We also recall that $\frac{\q}{3}$ is even and $\frac{\q}{2}$ is odd, so

\begin{equation}
 \label{eq: eq25}
\sum_{s \in S_0}\omega(a)^{-s}c_s=\dfrac{q^3-4q^2+6q-5}{\q}+\dfrac{1}{q\q}\sum_{m \notin \mathcal{B}}g(-4m)g(-2m)g(-3m)^2g(12m)\omega(-12\psi)^{-12m}   
\end{equation}

Isolating the sum in equation \eqref{eq: 23} and substituting in equation \eqref{eq: eq25}, we obtain \begin{align}
\sum_{s \in S_0}\omega(a)^{-s}c_s&=q^2-3q+4+H_q(\bm{\alpha}_{\clubsuit_0},\bm{\beta}_{\clubsuit_0}\mid t_{\clubsuit_0}).
\end{align}

\item[\textbf{(d)}] Let $\mathcal{B}=\Big\{0,\frac{\q}{2},\frac{\q}{4},\frac{3\q}{4},\frac{\q}{3},\frac{2\q}{3}\Big\}.$ We have 
\begin{align}\label{eq: eq20}
   & H_q(\bm{\alpha}_{\clubsuit_0},\bm{\beta}_{\clubsuit_0}\mid t_{\clubsuit_0})=\dfrac{1}{\q}\sum_{m=0}^{q-2}q^{-1+s(m)}g(-4m)g(-2m)g(-3m)^2g(12m)\omega(-12\psi)^{-12m}\notag\\
    &=\frac{1}{\q}\left(-1-3q+g(\tfrac{\q}{2})(g(\tfrac{\q}{4})^2+g(-\tfrac{\q}{4})^2)+\frac{1}{q}\sum_{m \notin \mathcal{B}}g(-4m)g(-2m)g(-3m)^2g(12m)\omega(-12\psi)^{-12m}\right)
\end{align}
    
We use the special values of $c_s$ for $m \in \mathcal{B}$ to pull the corresponding terms out of the sum in \eqref{sum: sum}. Also recall that $\frac{\q}{2}$ and $\frac{\q}{3}$ are both even, so \begin{align}\label{eq: 21}
    \sum_{s \in S_0}\omega(a)^{-s}c_s&=\frac{q^3-4q^2+6q-7}{\q}+\frac{g(\frac{\q}{2})(g(\frac{\q}{4})^2+g(-\frac{\q}{4})^2)}{q\q}\notag\\&+\frac{1}{q\q}\sum_{m \notin \mathcal{B}}g(-4m)g(-2m)g(-3m)^2g(12m)\omega(-12\psi)^{-12m}.
\end{align}

We isolate the sum in \eqref{eq: 20}, plug into equation \eqref{eq: 21} and convert between Gauss and Jacobi sums (just as in the proof of \textbf{(b)}) to obtain the stated result \begin{align*}
    \sum_{s \in S_0}\omega(a)^{-s}c_s&=q^2-3q+6-\left(J_q\left(\tfrac{1}{4},\tfrac{1}{4}\right)+J_q\left(\tfrac{3}{4},\tfrac{3}{4}\right)\right)+H_q(\bm{\alpha}_{\clubsuit_0},\bm{\beta}_{\clubsuit_0}\mid t_{\clubsuit_0}).
\end{align*}
\end{itemize}
\end{proof}

\begin{lemma}\label{lem: junk1}
If $q \equiv 1 \Mod 3,$ then \begin{align}
g(\tfrac{\q}{3})^3\omega(2)^{-\frac{\q}{3}}&=(-1)^{\frac{\q}{2}}g(\tfrac{\q}{2})g(\tfrac{\q}{6})g(\tfrac{\q}{3})=qJ_q\left(\tfrac{1}{3},\tfrac{1}{2}\right)\label{alg: first}\\
g(\tfrac{2\q}{3})^3\omega(2)^{-\frac{2\q}{3}}&=(-1)^{\frac{\q}{2}}g(\tfrac{\q}{2})g(\tfrac{5\q}{6})g(\tfrac{2\q}{3})=qJ_q\left(\tfrac{2}{3},\tfrac{1}{2}\right)\label{alg: second}
\end{align}    
\end{lemma}
\begin{proof}
We will show only \eqref{alg: first}, since \eqref{alg: second} is completely analogous. Take $N=2$ and $m=\frac{\q}{6}$ in the Hasse-Davenport relation: $$g(\tfrac{\q}{2})g(\tfrac{\q}{3})\omega(2)^{-\frac{\q}{3}}=g(\tfrac{\q}{6})g(\tfrac{2\q}{3}).$$

Multiply both sides by $g(\frac{\q}{2})$ and get $$(-1)^{\frac{\q}{2}}qg(\tfrac{\q}{3})\omega(2)^{-\frac{\q}{3}}=g(\tfrac{\q}{2})g(\tfrac{\q}{6})g(\tfrac{2\q}{3}).$$ 

Now multiply both sides by $g(\tfrac{\q}{3})^2$ to obtain the first equality in \eqref{alg: first}. For the second equality, use Lemma \ref{lemma: JacobiGauss}. 
\end{proof}

\begin{lemma} \label{lemma:char}
\begin{itemize}
    \item[(a)] If $q \equiv 1 \Mod 4,$ then $\omega(2)^{\frac{\q}{2}}=(-1)^{\frac{\q}{4}}.$
    \item[(b)] If $q \equiv 1 \Mod 8,$ then $\omega(2)^{\frac{\q}{4}}g(\tfrac{3\q}{4})^2g(\tfrac{\q}{2})=g(\tfrac{3\q}{4})g(\tfrac{3\q}{8})g(\tfrac{7\q}{8})=qJ_q(\tfrac{3}{4}, \tfrac{3}{4}).$
\end{itemize}   
\end{lemma}
\begin{proof}
\begin{itemize}
    \item[(a)] It follows from a simple application of Hasse-Davenport with $N=2$ and $m=\frac{\q}{4}.$ 
    \item[(b)] It follows from the Hasse-Davenport relation with $N=2,$ $m=\frac{\q}{8}$ and then use item (a).
\end{itemize}
\end{proof}

\begin{lemma}\label{lemma: lemma2} The pair $\bm{\alpha}_{\clubsuit_1},\bm{\beta}_{\clubsuit_1}$ is defined over $\Q$ and  
\begin{itemize}
    \item[(a)] If $q \not\equiv 1 \Mod{4}$ and $q \not\equiv 1 \Mod{3},$ then $$\sum_{s \in S_1}\omega(a)^{-s}c_s=-2-qH_q(\bm{\alpha}_{\clubsuit_1},\bm{\beta}_{\clubsuit_1}\mid t_{\clubsuit_1}).$$
    \item[(b)] If $q \not \equiv 1 \Mod{4}$ and $q \equiv 1 \Mod{3},$ then $$\sum_{s \in S_1}\omega(a)^{-s}c_s=-2-qH_q(\bm{\alpha}_{\clubsuit_1},\bm{\beta}_{\clubsuit_1}\mid t_{\clubsuit_1})+2\left(J_q\left(\tfrac{1}{2},\tfrac{1}{3}\right)+J_q\left(\tfrac{1}{2},\tfrac{2}{3}\right)\right).$$ 
    \item[(c)] If $q \equiv 1 \Mod{4}$ and $q \not \equiv 1 \Mod{3},$ then $$\sum_{s \in S_1}\omega(a)^{-s}c_s=2(1+(-1)^{\frac{\q}{4}})+qH_q(\bm{\alpha}_{\clubsuit_1},\bm{\beta}_{\clubsuit_1}\mid t_{\clubsuit_1}).$$
    \item[(d)] If $q \equiv 1 \Mod{4}$ and $q \equiv 1 \Mod{3},$ then $$\sum_{s \in S_1}\omega(a)^{-s}c_s=2(1+(-1)^{\frac{\q}{4}})+qH_q(\bm{\alpha}_{\clubsuit_1},\bm{\beta}_{\clubsuit_1}\mid t_{\clubsuit_1})-2\left(J_q\left(\tfrac{1}{2},\tfrac{1}{3}\right)+J_q\left(\tfrac{1}{2},\tfrac{2}{3}\right)\right).$$
\end{itemize}    
\end{lemma}
\begin{proof} The coefficients in \eqref{c:c} are
\begin{equation} \label{eq: eq16}
        c_{(-4m,-2m+\frac{\q}{2},-3m,-3m+\frac{\q}{2},12m)}=\dfrac{1}{q\q}g(-4m)g(-2m+\tfrac{\q}{2})g(-3m)g(-3m+\tfrac{\q}{2})g(12m), 
    \end{equation} $\text{ for every } m \in \{0,\dots,q-2\}.$ We first compute some special values of the $ c_s$s that will appear later when we separate into cases.     
    \begin{table}[H]
        \centering
        \begin{tabular}{cc||cc}
        $m$ & Corresponding $c_s$ & $m$ & Corresponding $c_s$\\
        \hline
        $0,\frac{\q}{2}$   & $\frac{(-1)^{\frac{\q}{2}+1}}{\q}$ & $\frac{\q}{6},\frac{2\q}{3}$   & $\frac{1}{q\q}g(\tfrac{\q}{3})g(\tfrac{\q}{6})g(\tfrac{\q}{2})$\\
        $\frac{\q}{4},\frac{3\q}{4}$ & $\frac{(-1)^{\frac{\q}{4}+1}}{\q}$ & $\frac{\q}{3},\frac{5\q}{6}$     & $\frac{1}{q\q}g(\tfrac{2\q}{3})g(\tfrac{5\q}{6})g(\tfrac{\q}{2})$\\
        \end{tabular}
        \label{tab:my_label}
    \end{table}
    
    Using Hasse-Davenport with $N=2$ and $-3m$ we obtain \begin{align}\label{al: al1}        g(-3m)g(-3m+\tfrac{\q}{2})=\dfrac{g(\frac{\q}{2})}{\omega(2)^{-6m}}g(-6m).
    \end{align}

    Using Hasse-Davenport again with $N=2$ and $-2m$ together with the relation $g(2m)g(-2m)=q,$ for all $m \notin \{0,\tfrac{\q}{2}\}$ we obtain \begin{align}\label{al: al2}
        g(-2m+\tfrac{\q}{2})=\tfrac{g(\tfrac{\q}{2})}{q\omega(2)^{-4m}}g(-4m)g(2m) \text{ for every } m \notin \{0,\tfrac{\q}{2}\}.
    \end{align}

    Substituting \eqref{al: al1} and \eqref{al: al2} in \eqref{eq: eq16} we obtain \begin{align}\label{c: cf}
        c_{(-4m,-2m+\frac{\q}{2},-3m,-3m+\frac{\q}{2},12m)}&=\dfrac{g(\frac{\q}{2})^2}{q^2\q\omega(2)^{-10m}}g(-4m)^2g(-6m)g(2m)g(12m)\\&=\dfrac{(-1)^{\frac{\q}{2}}}{q\q}g(-4m)^2g(-6m)g(2m)g(12m)\omega(2)^{10m} \text{ for every } m \notin \{0,\tfrac{\q}{2}\}.
    \end{align}

Now considering $q_1=4, q_2=4, q_3=6, p_1=2$ and $p_2=12$ and computing the quotient of the corresponding polynomials we obtain $$\dfrac{(x^2-1)(x^{12}-1)}{(x^4-1)^2(x^{6}-1)}=\dfrac{(x-e^{\frac{\pi i}{6} })(x-e^{\frac{5\pi i}{6}})(x-e^{\frac{7\pi i}{6}})(x-e^{\frac{11\pi i}{6}})}{(x-1)(x-e^{\frac{3\pi i}{2}})(x-e^{\frac{\pi i}{2}})(x-e^{\pi i})}.$$ The $\gcd$ of the numerator and denominator is $$D(x)=(x-1)^2(x-e^{\frac{\pi i}{2}})(x-e^{\pi i})^2(x-e^{\frac{3\pi i}{2}})(x-e^{\frac{\pi i}{3}})(x-e^{\frac{2\pi i}{3}})(x-e^{\frac{4\pi i}{3}})(x-e^{\frac{5\pi i}{3}}).$$ 

Moreover $K_{\bm{\alpha}_{\clubsuit_1},\bm{\beta}_{\clubsuit_1}}=\mathbb{Q}$ which means that $\bm{\alpha}_{\clubsuit_1},\bm{\beta}_{\clubsuit_1}$ are defined over $\mathbb{Q}.$ We have that $M=2^43^6$ and $\epsilon=1.$

\begin{align}\label{hip: hip1}
    H_q(\bm{\alpha}_{\clubsuit_1},\bm{\beta}_{\clubsuit_1}\mid t_{\clubsuit_1})&=\frac{(-1)^5}{1-q}\sum_{m=0}^{q-2}q^{-s(0)+s(m)}g(-4m)^2g(-6m)g(2m)g(12m)\omega(2^{-4}3^{-6}t)^m\\&=\frac{1}{q-1}\sum_{m=0}^{q-2}q^{-s(0)+s(m)}g(-4m)^2g(-6m)g(2m)g(12m)\omega(2^{-4}3^{-6}u)^m
\end{align}

Now fix $t=2^{-10}3^{-6}\psi^{-12}$ and substitute in the hypergeometric sum \eqref{hip: hip1}. Notice that it gives us the following \begin{align}\label{hip: hip2}
H_q(\bm{\alpha}_{\clubsuit_1},\bm{\beta}_{\clubsuit_1}\mid t_{\clubsuit_1})=\frac{1}{q-1}\sum_{m=0}^{q-2}q^{-s(0)+s(m)}g(-4m)^2g(-6m)g(2m)g(12m)\omega(2)^{10m}\omega\left(-12\psi\right)^{-12m}    
\end{align}

Now we need to compute the values of the function $s(m),$ $m \in \{0,\dots,q-2\}.$ For this purpose, we consider four cases depending on the modularity of $q-1$ for 3 and 4. 

\begin{itemize}
    \item[(a)] Let $\mathcal{B}=\Big\{0,\frac{\q}{2}\Big\}$ and by looking at $D(x)$ observe that $$s(m)=\begin{cases}
    2 \text{ if } m \in \mathcal{B} \\
    0 \text{ otherwise } 
    \end{cases}.$$ 
Substituting this in \eqref{hip: hip2} we get
    \begin{align*}
H_q(\bm{\alpha}_{\clubsuit_1},\bm{\beta}_{\clubsuit_1}\mid t_{\clubsuit_1})=-\dfrac{2}{\q}+\frac{1}{q^2\q}\sum_{m \notin \mathcal{B}}g(-4m)^2g(-6m)g(2m)g(12m)\omega(2)^{10m}\omega\left(-12\psi\right)^{-12m}    
\end{align*} and multiplying by $q$ we get \begin{align}\label{sum: sum1}
qH_q(\bm{\alpha}_{\clubsuit_1},\bm{\beta}_{\clubsuit_1}\mid t_{\clubsuit_1})=-\dfrac{2q}{\q}+\frac{1}{q\q}\sum_{m \notin \mathcal{B}}g(-4m)^2g(-6m)g(2m)g(12m)\omega(2)^{10m}\omega\left(-12\psi\right)^{-12m}    
\end{align}

Now that we have the hypergeometric sum, we look at the Koblitz formula and use \eqref{c: cf} and the special $c_s$'s to write it \begin{align}\label{kob: kobf}
\sum_{s \in S_1}\omega(a)^{-s}c_s=2\dfrac{(-1)^{\frac{\q}{2}+1}}{\q}+\dfrac{(-1)^{\frac{\q}{2}}}{q\q}\sum_{m \notin \mathcal{B}}g(-4m)^2g(-6m)g(2m)g(12m)\omega(2)^{10m}\omega(-12\psi)^{-12m}    
\end{align} 

Observe that since $4 \nmid \q,$ then $\frac{\q}{2}$ is odd and this implies that $(-1)^{\frac{\q}{2}}=-1.$ Substituting in \eqref{kob: kobf} we obtain \begin{align}\label{sum: sum2}
    \sum_{s \in S_1}\omega(a)^{-s}c_s=\dfrac{2}{\q}-\dfrac{1}{q\q}\sum_{m \notin \mathcal{B}}g(-4m)^2g(-6m)g(2m)g(12m)\omega(2)^{10m}\omega(-12\psi)^{-12m}.
\end{align}

Observe that we have a common sum in \eqref{sum: sum1} and \eqref{sum: sum2}, namely $$S=\dfrac{1}{q\q}\sum_{m \notin \mathcal{B}}g(-4m)^2g(-6m)g(2m)g(12m)\omega(2)^{10m}\omega(-12\psi)^{-12m}.$$ We isolate $S$ in \eqref{sum: sum1} and substitute in \eqref{sum: sum2} which gives us \begin{align}
\sum_{s \in S_1}\omega(a)^{-s}c_s=\frac{2}{\q}-\dfrac{2q}{\q}-qH_q(\bm{\alpha}_{\clubsuit_1},\bm{\beta}_{\clubsuit_1}\mid t_{\clubsuit_1}).    
\end{align}

Therefore, $$\sum_{s \in S_1}\omega(a)^{-s}c_s=\frac{2-2q}{\q}-qH_q(\bm{\alpha}_{\clubsuit_1},\bm{\beta}_{\clubsuit_1}\mid t_{\clubsuit_1})=\frac{2(1-q)}{\q}-qH_q(\bm{\alpha}_{\clubsuit_1},\bm{\beta}_{\clubsuit_1}\mid t_{\clubsuit_1})=-2-qH_q(\bm{\alpha}_{\clubsuit_1},\bm{\beta}_{\clubsuit_1}\mid t_{\clubsuit_1}).$$   

\item[(b)] Let $\mathcal{B}=\Big\{0,\frac{\q}{2},\frac{\q}{6},\frac{\q}{3},\frac{2\q}{3},\frac{5\q}{6}\Big\}.$ The values of $s$ are $s(m)=\begin{cases}
    2, \text{ if } m=0,\frac{\q}{2} \\
    1, \text{ if } m=\frac{\q}{6},\frac{\q}{3},\frac{2\q}{3},\frac{5\q}{6} \\
    0, \text{ otherwise }
    \end{cases}$

Substituting these values in \eqref{hip: hip2} we have \begin{align}
H_q(\bm{\alpha}_{\clubsuit_1},\bm{\beta}_{\clubsuit_1}\mid t_{\clubsuit_1})&=-\frac{2}{\q}+\frac{2}{q\q}(g(\tfrac{\q}{3})^3\omega(2)^{-\frac{\q}{3}}+g(\tfrac{2\q}{3})^3\omega(2)^{-\frac{2\q}{3}})\notag\\&+\frac{1}{q^2\q}\sum_{m \notin \mathcal{B}}g(-4m)^2g(-6m)g(2m)g(12m)\omega(2)^{10m}\omega(-12\psi)^{-12m}.    
\end{align} 
Therefore, \begin{align}\label{hyp: hyp2}
qH_q(\bm{\alpha}_{\clubsuit_1},\bm{\beta}_{\clubsuit_1}\mid t_{\clubsuit_1})&=-\frac{2q}{\q}+\frac{2}{\q}(g(\tfrac{\q}{3})^3\omega(2)^{-\frac{\q}{3}}+g(\tfrac{2\q}{3})^3\omega(2)^{-\frac{2\q}{3}})\notag \\&+\frac{1}{q\q}\sum_{m \notin \mathcal{B}}g(-4m)^2g(-6m)g(2m)g(12m)\omega(2)^{10m}\omega(-12\psi)^{-12m}.  
\end{align}    

The next step is to look at Koblitz's formula and use \eqref{c: cf} and the special $c_s$'s to write it as 
\begin{align}
\sum_{s \in S_1}\omega(a)^{-s}c_s&=\frac{2}{\q}+\frac{2}{q\q}g(\tfrac{\q}{2})(g(\tfrac{\q}{3})g(\tfrac{\q}{6})+g(\tfrac{2\q}{3})g(\tfrac{5\q}{6}))\notag\\&-\frac{1}{q\q}\sum_{m \notin \mathcal{B}}g(-4m)^2g(-6m)g(2m)g(12m)\omega(2)^{10m}\omega(-12\psi)^{-12m}.   
\end{align}

Let $$S=\frac{1}{q\q}\sum_{m \notin \mathcal{B}}g(-4m)^2g(-6m)g(2m)g(12m)\omega(2)^{10m}\omega(-12\psi)^{-12m}$$ which is common between \eqref{hyp: hyp2} and \eqref{sum: sum2}. Isolating $S$ in \eqref{hyp: hyp2} and substituting in \eqref{sum: sum2} we obtain \begin{align}
\sum_{s \in S_1}\omega(a)^{-s}c_s&=\frac{2}{\q}+\frac{2}{q\q}g(\tfrac{\q}{2})(g(\tfrac{\q}{3})g(\tfrac{\q}{6})+g(\tfrac{2\q}{3})g(\tfrac{5\q}{6}))-\frac{2q}{\q}+\\&\frac{2}{\q}(g(\tfrac{\q}{3})^3\omega(2)^{-\frac{\q}{3}}+g(\tfrac{2\q}{3})^3\omega(2)^{-\frac{2\q}{3}})-qH_q(\bm{\alpha}_{\clubsuit_1},\bm{\beta}_{\clubsuit_1}\mid t_{\clubsuit_1}).    
\end{align}

Now applying Lemma \ref{lem: junk1} and rewriting one of the terms, we obtain \begin{align*}
\sum_{s \in S_1}\omega(a)^{-s}c_s&=\frac{-2(q-1)}{\q}-qH_q(\bm{\alpha}_{\clubsuit_1},\bm{\beta}_{\clubsuit_1}\mid t_{\clubsuit_1})-\frac{2q}{q\q}g(\tfrac{\q}{2})(g(\tfrac{\q}{3})g(\tfrac{\q}{6})+g(\tfrac{2\q}{3})g(\tfrac{5\q}{6}))\\&+\frac{2}{q\q}g(\tfrac{\q}{2})(g(\tfrac{\q}{3})g(\tfrac{\q}{6})+g(\tfrac{2\q}{3})g(\tfrac{5\q}{6}))\\&=-2-qH_q(\bm{\alpha}_{\clubsuit_1},\bm{\beta}_{\clubsuit_1}\mid t_{\clubsuit_1})-\frac{2(q-1)}{q\q}g(\tfrac{\q}{2})(g(\tfrac{\q}{3})g(\tfrac{\q}{6})+g(\tfrac{2\q}{3})g(\tfrac{5\q}{6}))\\&=-2-qH_q(\bm{\alpha}_{\clubsuit_1},\bm{\beta}_{\clubsuit_1}\mid t_{\clubsuit_1})-\frac{2}{q}g(\tfrac{\q}{2})(g(\tfrac{\q}{3})g(\tfrac{\q}{6})+g(\tfrac{2\q}{3})g(\tfrac{5\q}{6}))\\&=-2-qH_q(\bm{\alpha}_{\clubsuit_1},\bm{\beta}_{\clubsuit_1}\mid t_{\clubsuit_1})+2\left(J_q\left(\tfrac{1}{2},\tfrac{1}{3}\right)+J_q\left(\tfrac{1}{2},\tfrac{2}{3}\right)\right),
\end{align*} where the last equality follows from Lemma \ref{lemma: JacobiGauss}, together with the fact that $\tfrac{\q}{6}$ is odd since $q \not\equiv 1 \Mod{4}.$
    \item[(c)] Let $\mathcal{B}=\Big\{0,\frac{\q}{4},\frac{\q}{2},\frac{3\q}{4}\Big\}.$ The values of $s$ are $s(m)=\begin{cases}
    2, \text{ if } m=0,\frac{\q}{2} \\
    1, \text{ if } m=\frac{\q}{4},\frac{3\q}{4}\\
    0, \text{ otherwise }
    \end{cases}$ 

    Substituting the previous values in \eqref{hip: hip2}, we get $$H_q(\bm{\alpha}_{\clubsuit_1},\bm{\beta}_{\clubsuit_1}\mid t_{\clubsuit_1})=-\tfrac{2}{\q}(1+\omega(2)^{\frac{\q}{2}})+\tfrac{1}{\q q^2}\sum_{m \notin \mathcal{B}}g(-4m)^2g(-6m)g(2m)g(12m)\omega(2)^{10m}\omega(-12\psi)^{-12m}.$$ By Lemma \ref{lemma:char}, we have 
    \begin{align}\label{hyp: hypp}
    H_q(\bm{\alpha}_{\clubsuit_1},\bm{\beta}_{\clubsuit_1}\mid t_{\clubsuit_1})=-\frac{2}{\q}(1+(-1)^{\frac{\q}{4}})+\frac{1}{\q q^2}\sum_{m \notin \mathcal{B}}g(-4m)^2g(-6m)g(2m)g(12m)\omega(2)^{10m}\omega(-12\psi)^{-12m}.  
    \end{align} On the other hand, from \eqref{eq: eq16} and \eqref{c: cf}, we have \begin{align}\label{kob: kobb}
    \sum_{s \in S_1} \omega(a)^{-s}c_s&=-\frac{2}{\q}(1+(-1)^{\frac{\q}{4}})+\frac{1}{q\q}\sum_{m \notin \mathcal{B}}g(-4m)^2g(-6m)g(2m)g(12m)\omega(2)^{10m}\omega(-12\psi)^{-12m}    
    \end{align} Comparing \eqref{hyp: hypp} and \eqref{kob: kobb}, we obtain the formula $$\sum_{s \in S_1} \omega(a)^{-s}c_s=2(1+(-1)^{\frac{\q}{4}})+qH_q(\bm{\alpha}_{\clubsuit_1},\bm{\beta}_{\clubsuit_1}\mid t_{\clubsuit_1}).$$
    
    \item[(d)] Let $\mathcal{B}=\Big\{0,\frac{\q}{4},\frac{\q}{2},\frac{3\q}{4},\frac{\q}{6},\frac{\q}{3},\frac{2\q}{3},\frac{5\q}{6}\Big\}.$ The values of $s$ are $$s(m)=\begin{cases}
    2, \text{ if } m=0,\frac{\q}{2} \\
    1, \text{ if } m=\frac{\q}{4},\frac{3\q}{4},\frac{\q}{6},\frac{\q}{3},\frac{2\q}{3},\frac{5\q}{6}\\
    0, \text{ otherwise }
    \end{cases}.$$ Using \eqref{hip: hip2}, we get \begin{align*}    H_q(\bm{\alpha}_{\clubsuit_1},\bm{\beta}_{\clubsuit_1}\mid t_{\clubsuit_1})&=-\frac{2}{\q}(1+(-1)^{\frac{\q}{4}})+\frac{2}{\q q}(g(\tfrac{\q}{3})^3\omega(2)^{\frac{2\q}{3}}+g(\tfrac{2\q}{3})^3\omega(2)^{\frac{\q}{3}}))\\&+\frac{1}{\q q^2}\sum_{m \notin \mathcal{B}}g(-4m)^2g(-6m)g(2m)g(12m)\omega(2)^{10m}\omega(-12\psi)^{-12m}.    
    \end{align*}
    
    Applying Lemma \eqref{lem: junk1}, we rewrite 
    \begin{align*}    H_q(\bm{\alpha}_{\clubsuit_1},\bm{\beta}_{\clubsuit_1}\mid t_{\clubsuit_1})&=-\frac{2}{\q}(1+(-1)^{\frac{\q}{4}})+\frac{2}{\q q}g(\tfrac{\q}{2})(g(\tfrac{\q}{6})g(\tfrac{\q}{3})+g(\tfrac{5\q}{6})g(\tfrac{2\q}{3}))\\&+\frac{1}{\q q^2}\sum_{m \notin \mathcal{B}}g(-4m)^2g(-6m)g(2m)g(12m)\omega(2)^{10m}\omega(-12\psi)^{-12m}.    
    \end{align*} 
    On the other hand, we also get \begin{align*}
    \sum_{s \in S_1}\omega(a)^{-s}c_s&=-\tfrac{2}{\q}(1+(-1)^{\frac{\q}{4}})+\tfrac{2}{q\q}g(\tfrac{\q}{2})(g(\tfrac{\q}{3})g(\tfrac{\q}{6})+g(\tfrac{2\q}{3})g(\tfrac{5\q}{6}))\\&+\tfrac{1}{q\q}\sum_{m \notin \mathcal{B}}g(-4m)^2g(-6m)g(2m)g(12m)\omega(2)^{10m}\omega(-12\psi)^{-12m}. 
    \end{align*}
    Therefore, comparing the last two equations we obtain \begin{align*}
    \sum_{s \in S_1}\omega(a)^{-s}c_s&=2(1+(-1)^{\frac{\q}{4}})+qH_q(\bm{\alpha}_{\clubsuit_1},\bm{\beta}_{\clubsuit_1}\mid t_{\clubsuit_1})-\tfrac{2}{q}g(\tfrac{\q}{2})(g(\tfrac{\q}{3})g(\tfrac{\q}{6})+g(\tfrac{2\q}{3})g(\tfrac{5\q}{6}))\\&=2(1+(-1)^{\frac{\q}{4}})+qH_q(\bm{\alpha}_{\clubsuit_1},\bm{\beta}_{\clubsuit_1}\mid t_{\clubsuit_1})-2\left(J_q\left(\tfrac{1}{2},\tfrac{1}{3}\right)+J_q\left(\tfrac{1}{2},\tfrac{2}{3}\right)\right), 
    \end{align*} and the last equality follows from Lemma \ref{lemma: JacobiGauss} and from the fact that $\tfrac{\q}{6}$ is even since $q \equiv 1 \Mod{4}.$
\end{itemize}  
\end{proof}

\begin{lemma}\label{lemma:club2}
Suppose that $q \equiv 1 \Mod{4},$ then the pair $\bm{\alpha}_{\clubsuit_2}=\left\{\tfrac{7}{12},\tfrac{1}{6},\tfrac{11}{12},\tfrac{5}{6}\right\},\bm{\beta}_{\clubsuit_2}=\left\{\tfrac{1}{8},\tfrac{5}{8},\tfrac{1}{4},1\right\}$ is defined over $\Q(i)$ and \begin{itemize}
    \item[(a)] If $q \not \equiv 1 \Mod{3},$ then 
\begin{align*}
    \sum_{s \in S_2}\omega(a)^{-s}c_s&=2(-1)^{\frac{\q}{4}}-\left(J_q\left(\tfrac{1}{4},\tfrac{1}{4}\right)+J_q\left(\tfrac{3}{4},\tfrac{3}{4}\right)\right)\\&+(-1)^{\frac{\q}{4}}qF_q((4,2,3,3,-12,1,-1),(0,-1,0,1,0,0,0),4 \mid t_{\clubsuit_2}).
\end{align*}  
\item[(b)] If $q \equiv 1 \Mod{3},$ then \begin{align*}
\sum_{s \in S_2}\omega(a)^{-s}c_s&=2(-1)^{\frac{\q}{4}}+(-1)^{\frac{\q}{4}}qF_q((4,2,3,3,-12,1,-1),(0,-1,0,1,0,0,0),4 \mid t_{\clubsuit_2})\\&-\left(J_q\left(\tfrac{1}{4},\tfrac{1}{4}\right)+J_q\left(\tfrac{3}{4},\tfrac{3}{4}\right)\right)-\sum_{i \in (\Z/12\Z)^{\times}}J_q(\tfrac{i}{12},\tfrac{i}{4}). \\  
\end{align*} 
\end{itemize}
\end{lemma}
\begin{proof}
Let $N=4.$ By Definition \ref{AA25}, we have \begin{align*}\label{alg: summ}
&F_q((4,2,3,3,-12,1,-1),(0,-1,0,1,0,0,0),4 \mid t_{\clubsuit_2})=-\tfrac{1}{\q}\\&-\tfrac{1}{\q}\sum_{m=1}^{q-2}\tfrac{g(-4m)g(-2m+\frac{3\q}{4})g(-3m)g(-3m+\frac{\q}{4})g(12m)g(-m)g(m)}{g(0)^5g(\frac{\q}{4})g(\frac{3\q}{4})}q^{s_{\bm{\delta}}(-m)-s_{\bm{\delta}}(0)}\omega(-2^{10}3^612^{-12}t)^m.
\end{align*}  
Observe that $g(-m)g(m)=(-1)^mq$ for $m \neq 0 \Mod{\q}$ and $g(\tfrac{\q}{4})g(\tfrac{3\q}{4})=(-1)^{\frac{\q}{4}}q.$ Moreover, we use that for a generator $\omega$ of the multiplicative character group of $\F_q^{\times},$ we have $\omega(-1)=-1$ and therefore $\omega(-1)^m=(-1)^m,$ so we can pull the factor $(-1)^m$ inside of the argument of the character $\omega$ and write  \begin{align}
&(-1)^{\frac{\q}{4}}F_q((4,2,3,3,-12,1,-1),(0,-1,0,1,0,0,0),4 \mid t_{\clubsuit_2})=-\tfrac{(-1)^{\frac{\q}{4}}}{\q}\\&+\tfrac{1}{\q}\sum_{m=1}^{q-2}g(-4m)g(-2m+\tfrac{3\q}{4})g(-3m)g(-3m+\tfrac{\q}{4})g(12m)q^{s_{\bm{\delta}}(-m)-s_{\bm{\delta}}(0)}\omega(-12\psi)^{-12m}. \notag
\end{align} 

Now consider the quotient of polynomials as in Definition \ref{def: quot}, $$
\frac{(x^{12}-1)(x-1)}{(x^4-1)(x^2-\zeta_4)(x^3-1)(x^3-\zeta_4^{-1})(x-1)}.$$ 

The greatest common divisor of the numerator and denominator is the polynomial $D_{\bm{\delta}}$ whose roots are given by $e^{2\pi i a},$ where $a \in \left\{0,0,\tfrac{1}{4},\tfrac{1}{2},\tfrac{3}{4},\tfrac{1}{3},\tfrac{2}{3},\tfrac{7}{12},\tfrac{11}{12}\right\}.$ Therefore, we obtain the pair $$\bm{\alpha}_{\clubsuit_2}=\left\{\tfrac{7}{12},\tfrac{1}{6},\tfrac{11}{12},\tfrac{5}{6}\right\},\bm{\beta}_{\clubsuit_2}=\left\{\tfrac{1}{8},\tfrac{5}{8},\tfrac{1}{4},1\right\}$$ 
which is defined over $\Q(i).$
\begin{itemize}
    \item[(a)] If $q \not \equiv 1 \Mod{3},$ then the values of $m \in \Z/\q\Z$ such that $e^{\frac{-2\pi i m}{\q}}$ is a root of $D_{\bm{\delta}}$ are $$\mathcal{B}_1=\left\{\tfrac{\q}{4},\tfrac{\q}{2},\tfrac{3\q}{4},0\right\}.$$ Notice that all of them have multiplicity one except $m=0,$ which has multiplicity $2,$ so \begin{align*}
    s_{\bm{\delta}}(-m)=\begin{cases}
    0, \text{ if } m \in \Z/\q\Z \setminus \mathcal{B}_1 \\
    1, \text{ if } m \in \mathcal{B}_1\setminus \{0\} \\
    2, \text{ if } m=0
    \end{cases}.    
    \end{align*} 
    We pull out of the sum \eqref{alg: summ} the terms corresponding to $m \in \mathcal{B}$ and obtain \begin{align}\label{alg: c2f2hyps3}
    &(-1)^{\frac{\q}{4}}F_q((4,2,3,3,-12,1,-1),(0,-1,0,1,0,0,0),4 \mid t_{\clubsuit_2})=-\tfrac{2(-1)^{\frac{\q}{4}}}{\q}+\tfrac{g(\tfrac{\q}{2})}{q\q}(g(\tfrac{\q}{4})^2+g(\tfrac{3\q}{4})^2)\notag\\&+\tfrac{1}{q^2\q}\sum_{m \notin \mathcal{B}}g(-4m)g(-2m+\tfrac{3\q}{4})g(-3m)g(-3m+\tfrac{\q}{4})g(12m)\omega(-12\psi)^{-12m}.  \end{align} 
    On the other hand, by Koblitz's formula, one gets \begin{align}\label{alg: c2f2s3}
    &\sum_{s \in S_2}\omega(a)^{-s}c_s=\tfrac{2(-1)^{\frac{\q}{4}+1}}{\q}+\tfrac{1}{q\q}g(\tfrac{\q}{2})(g(\tfrac{\q}{4})^2+g(\tfrac{3\q}{4})^2)\notag\\&+\tfrac{1}{q\q}\sum_{m \notin \mathcal{B}}g(-4m)g(-2m+\tfrac{3\q}{4})g(-3m)g(-3m+\tfrac{\q}{4})g(12m)\omega(-12\psi)^{-12m}.    
    \end{align}
    Comparing \eqref{alg: c2f2s3} and \eqref{alg: c2f2hyps3} and converting between Gauss sums and Jacobi sums as in the proof of items (b) and (d) of Lemma \ref{lm: thelemma}, we obtain \begin{align*}
\sum_{s \in S_2}\omega(a)^{-s}c_s&=2(-1)^{\frac{\q}{4}}-\left(J_q\left(\tfrac{1}{4},\tfrac{1}{4}\right)+J_q\left(\tfrac{3}{4},\tfrac{3}{4}\right)\right)+(-1)^{\frac{\q}{4}}qF_q((4,2,3,3,-12,1,-1),(0,-1,0,1,0,0,0),4 \mid t_{\clubsuit_2}).    
\end{align*}
\item[(b)] If $q \equiv 1 \Mod{3},$ then the values of $m \in \Z/\q\Z$ such that $e^{\frac{-2\pi i m}{\q}}$ is a root of $D_{\bm{\delta}}$ are the elements of $$\mathcal{B}_1=\left\{\tfrac{\q}{4},\tfrac{\q}{2},\tfrac{3\q}{4},\tfrac{\q}{3},\tfrac{2\q}{3},0\right\} \text{ and } \mathcal{B}_2=\left\{\tfrac{\q}{12},\tfrac{5\q}{12}\right\}.$$ Notice that all of them have multiplicity one except $m=0,$ which has multiplicity $2.$ \begin{align*}
    s_{\bm{\delta}}(-m)=\begin{cases}
    0, \text{ if } m \in \Z/\q\Z \setminus \mathcal{B}_1 \\
    1, \text{ if } m \in \mathcal{B}_1 \cup \mathcal{B}_2\setminus \{0\} \\
    2, \text{ if } m=0
    \end{cases}.    
    \end{align*} 
    We pull out of the sum \eqref{alg: summ} the terms corresponding to $m \in \mathcal{B}$ and obtain \begin{align}
    &(-1)^{\frac{\q}{4}}F_q((4,2,3,3,-12,1,-1),(0,-1,0,1,0,0,0),4 \mid t_{\clubsuit_2})=-\tfrac{2(-1)^{\frac{\q}{4}}}{\q}+\tfrac{g(\tfrac{\q}{2})}{q\q}(g(\tfrac{\q}{4})^2+g(\tfrac{3\q}{4})^2)\\&+\tfrac{g(\tfrac{\q}{3})}{q\q}(g(\tfrac{\q}{4})g(\tfrac{5\q}{12})+g(\tfrac{3\q}{4})g(\tfrac{11\q}{12}))\notag+\tfrac{g(\tfrac{2\q}{3})}{q\q}(g(\tfrac{\q}{4})g(\tfrac{\q}{12})+g(\tfrac{3\q}{4})g(\tfrac{7\q}{12}))\notag\\&+\tfrac{1}{q^2\q}\sum_{m \notin \mathcal{B}}g(-4m)g(-2m+\tfrac{3\q}{4})g(-3m)g(-3m+\tfrac{\q}{4})g(12m)\omega(-12\psi)^{-12m}.\notag\end{align} 
    On the other hand, by Koblitz's formula, one gets \begin{align}\label{alg: c2f2s2}
    &\sum_{s \in S_2}\omega(a)^{-s}c_s=\tfrac{2(-1)^{\frac{\q}{4}+1}}{\q}+\tfrac{1}{q\q}g(\tfrac{\q}{2})(g(\tfrac{\q}{4})^2+g(\tfrac{3\q}{4})^2)+\tfrac{g(\tfrac{\q}{3})}{q\q}(g(\tfrac{\q}{4})g(\tfrac{5\q}{12})+g(\tfrac{3\q}{4})g(\tfrac{11\q}{12}))\\&+\tfrac{g(\tfrac{2\q}{3})}{q\q}(g(\tfrac{\q}{4})g(\tfrac{\q}{12})+g(\tfrac{3\q}{4})g(\tfrac{7\q}{12}))\notag+\tfrac{1}{q\q}\sum_{m \notin \mathcal{B}}g(-4m)g(-2m+\tfrac{3\q}{4})g(-3m)g(-3m+\tfrac{\q}{4})g(12m)\omega(-12\psi)^{-12m}.    
    \end{align}
    Comparing \eqref{alg: c2f2s2} and \eqref{alg: c2f2hyps3} and converting between Gauss and Jacobi sums, we obtain \begin{align*}
\sum_{s \in S_2}\omega(a)^{-s}c_s&=2(-1)^{\frac{\q}{4}}+(-1)^{\frac{\q}{4}}qF_q((4,2,3,3,-12,1,-1),(0,-1,0,1,0,0,0),4 \mid t_{\clubsuit_2})\\&-\left(J_q\left(\tfrac{1}{4},\tfrac{1}{4}\right)+J_q\left(\tfrac{3}{4},\tfrac{3}{4}\right)\right)-\sum_{i \in (\Z/12\Z)^{\times}}J_q(\tfrac{i}{12},\tfrac{i}{4}).    
\end{align*}
\end{itemize}
\end{proof}

\begin{lemma}\label{lemma:club3}
Suppose that $q \equiv 1 \Mod{4},$ then the pair $\bm{\alpha}_{\clubsuit_3}=\left\{\tfrac{1}{12},\tfrac{1}{6},\tfrac{5}{12},\tfrac{5}{6}\right\},\bm{\beta}_{\clubsuit_3}=\left\{\tfrac{3}{8},\tfrac{7}{8},\tfrac{3}{4},1\right\}$ is defined over $\Q(i)$ and 
\begin{itemize}
    \item[(a)] If $q \not\equiv 1 \Mod{3},$ then \begin{align*}
    \sum_{s \in S_3}\omega(a)^{-s}c_s&=-\left(J_q\left(\tfrac{1}{4},\tfrac{1}{4}\right)+J_q\left(\tfrac{3}{4},\tfrac{3}{4}\right)\right)+(-1)^{\frac{\q}{4}}(2+qF_q((4,2,3,3,-12,1,-1),(0,1,0,-1,0,0,0),4 \mid t_{\clubsuit_3})).
\end{align*}
\item[(b)] If $q \equiv 1 \Mod{3},$ then \begin{align*}
    \sum_{s \in S_3}\omega(a)^{-s}c_s&=(-1)^{\frac{\q}{4}}(2+qF_q((4,2,3,3,-12,1,-1),(0,1,0,-1,0,0,0),4 \mid t_{\clubsuit_3}))\\&-\left(J_q\left(\tfrac{1}{4},\tfrac{1}{4}\right)+J_q\left(\tfrac{3}{4},\tfrac{3}{4}\right)\right)-\sum_{i \in (\Z/12\Z)^{\times}}J_q(\tfrac{i}{12},\tfrac{i}{4}). 
\end{align*}
\end{itemize}    
\end{lemma}
\begin{proof}
The proof is analogous to the previous Lemma with $(0,1,0,-1,0,0,0)$ instead of $(0,-1,0,-1,0,0,0).$ 
\end{proof}

\bigskip 
Combining Lemmas \ref{lm: thelemma}, \ref{lemma: lemma2}, \ref{lemma:club2}, and \ref{lemma:club3}, we obtain the following point count for $\Csf_2\Fsf_2$ on the maximal torus
\begin{lemma}\label{lemma: lemmainsidec2f2}
\begin{itemize}
    \item[(a)] If $q \not\equiv 1 \Mod{4}$ and $q \not \equiv 1 \Mod{3},$ then \begin{align*}
    \#U_{\psi}(\F_q)&=q^2-3q+H_q(\bm{\alpha}_{\clubsuit_0},\bm{\beta}_{\clubsuit_0}\mid t_{\clubsuit_0})-qH_q(\bm{\alpha}_{\clubsuit_1},\bm{\beta}_{\clubsuit_1}\mid t_{\clubsuit_1}).   
    \end{align*}
    \item[(b)] If $q \not\equiv 1 \Mod{4}$ and $q \equiv 1 \Mod{3},$ then 
\begin{align*}
    \#U_{\psi}(\F_q)&=q^2-3q+2+H_q(\bm{\alpha}_{\clubsuit_0},\bm{\beta}_{\clubsuit_0}\mid t_{\clubsuit_0})-qH_q(\bm{\alpha}_{\clubsuit_1},\bm{\beta}_{\clubsuit_1}\mid t_{\clubsuit_1})+2\left(J_q\left(\tfrac{1}{2},\tfrac{1}{3}\right)+J_q\left(\tfrac{1}{2},\tfrac{2}{3}\right)\right).   
    \end{align*}
    \item[(c)] If $q \equiv 1 \Mod{4}$ and $q \not \equiv 1 \Mod{3},$ then \begin{align*}
    \#U_{\psi}(\F_q)&=q^2-3q+6(1+(-1)^{\frac{\q}{4}})-3\left(J_q\left(\tfrac{1}{4},\tfrac{1}{4}\right)+J_q\left(\tfrac{3}{4},\tfrac{3}{4}\right)\right)+H_q(\bm{\alpha}_{\clubsuit_0},\bm{\beta}_{\clubsuit_0}\mid t_{\clubsuit_0})+qH_q(\bm{\alpha}_{\clubsuit_1},\bm{\beta}_{\clubsuit_1}\mid t_{\clubsuit_1})\\&+(-1)^{\frac{\q}{4}}q(F_q((4,2,3,3,-12,1,-1),(0,-1,0,1,0,0,0),4 \mid t_{\clubsuit_2})\\&+F_q((4,2,3,3,-12,1,-1),(0,1,0,-1,0,0,0),4 \mid t_{\clubsuit_3})).  
    \end{align*}
    \item[(d)] If $q \equiv 1 \Mod{3}$ and $q \equiv 1 \Mod{4},$ then \begin{align*}
    \#U_{\psi}(\F_q)&=q^2-3q+8+6(-1)^{\frac{\q}{4}}-3\left(J_q\left(\tfrac{1}{4},\tfrac{1}{4}\right)+J_q\left(\tfrac{3}{4},\tfrac{3}{4}\right)\right)-2\sum_{i \in (\Z/12\Z)^{\times}}J_q\left(\tfrac{i}{12},\tfrac{i}{4}\right)\\&-2\left(J_q\left(\tfrac{1}{2},\tfrac{1}{3}\right)+J_q\left(\tfrac{1}{2},\tfrac{2}{3}\right)\right)+H_q(\bm{\alpha}_{\clubsuit_0},\bm{\beta}_{\clubsuit_0}\mid t_{\clubsuit_0})+qH_q(\bm{\alpha}_{\clubsuit_1},\bm{\beta}_{\clubsuit_1}\mid t_{\clubsuit_1})\\&+(-1)^{\frac{\q}{4}}q(F_q((4,2,3,3,-12,1,-1),(0,-1,0,1,0,0,0),4 \mid t_{\clubsuit_2})\\&+F_q((4,2,3,3,-12,1,-1),(0,1,0,-1,0,0,0),4 \mid t_{\clubsuit_3})). 
    \end{align*}
\end{itemize}    
\end{lemma}

\bigskip 

Now, we compute the number of points on the smaller tori for $\Csf_2\Fsf_2.$

\begin{lemma}\label{lemma: lemmaoutsidec2f2}
\begin{itemize}
    \item[(a)] If $q \not \equiv 1 \Mod{4}$ and $q \not \equiv 1 \Mod{3},$ then \begin{align*}
    \#X_{\psi}(\F_q)-\#U_{\psi}(\F_q)=3q+1.   
    \end{align*} 
    \item[(b)] If $q\not\equiv 1 \Mod{4}$ and $q \equiv 1 \Mod{3},$ then \begin{align*}
    \#X_{\psi}(\F_q)-\#U_{\psi}(\F_q)=3q-1-2\left(J_q\left(\tfrac{1}{2},\tfrac{1}{3}\right)+J_q\left(\tfrac{1}{2},\tfrac{2}{3}\right)\right).   
    \end{align*}
    \item[(c)] If $q \equiv 1 \Mod{4}$ and $q \not \equiv 1 \Mod{3},$ then \begin{align*}
    \#X_{\psi}(\F_q)-\#U_{\psi}(\F_q)=(5+2(-1)^{\frac{\q}{4}})q-5-6(-1)^{\frac{\q}{4}}+3\left(J_q\left(\tfrac{1}{4},\tfrac{1}{4}\right)+J_q\left(\tfrac{3}{4},\tfrac{3}{4}\right)\right).
\end{align*}
    \item[(d)] If $q \equiv 1 \Mod{4}$ and $q \equiv 1 \Mod{3},$ then \begin{align*}
    \#X_{\psi}(\F_q)-\#U_{\psi}(\F_q)&=(5+2(-1)^{\frac{\q}{4}})q-7-6(-1)^{\frac{\q}{4}}+3\left(J_q\left(\tfrac{1}{4},\tfrac{1}{4}\right)+J_q\left(\tfrac{3}{4},\tfrac{3}{4}\right)\right)\\&+2\left(J_q\left(\tfrac{1}{2},\tfrac{1}{3}\right)+J_q\left(\tfrac{1}{2},\tfrac{2}{3}\right)\right)+\sum_{i \in (\Z/12\Z)^{\times}}J_q(\tfrac{i}{12},\tfrac{i}{4}).   
    \end{align*}
\end{itemize}    
\end{lemma}
\begin{proof}
We compute only cases (a) and (d), since (b) and (c) can be worked out similarly.
\begin{itemize}
\item[(a)] Suppose that $q \not \equiv 1 \Mod{4}$ and $q \not \equiv 1 \Mod{3}.$ 

\bigskip 

\textbf{One coordinate is zero and all the others are non-zero} 

\bigskip 

The first case is $x_0=0,$ whose corresponding equation is $x_1^4+x_2^4+x_3^4=0.$ The system is then \begin{align*}
\begin{pmatrix}
4 & 0 & 0 \\
0 & 4 & 0 \\
0 & 0 & 4 \\
1 & 1 & 1 
\end{pmatrix}\begin{pmatrix}
s_0 \\
s_1 \\
s_2
\end{pmatrix}\equiv 0 \Mod{\q}
\end{align*} and has solutions $\{(0,0,0),\tfrac{\q}{2}(1,0,1),\tfrac{\q}{2}(0,1,1),\tfrac{\q}{2}(1,1,0)\}.$ Therefore, the number of points is $q+1.$ The second case is $x_1=0$ and the equation is $x_2^4+x_3^4=0.$ The corresponding system is \begin{align*}
\begin{pmatrix}
 4 & 0 \\
 0 & 4 \\
 1 & 1 
\end{pmatrix}\begin{pmatrix}
s_0 \\
s_1 
\end{pmatrix}\equiv 0 \Mod{\q}
\end{align*} whose solutions are $\{(0,0),\tfrac{\q}{2}(1,1)\}.$ Consequently, the number of points is $c_{(0,0)}+c_{\tfrac{\q}{2}(1,1)}=0.$ The third and fourth cases are similar, namely $x_2=0$ and $x_3=0,$ and they both give $q-1$ points. For instance, if $x_2=0$, we have equation $x_0^3x_1+x_1^4+x_3^4=0$ and system \begin{align*}
\begin{pmatrix}
 3 & 0 & 0\\
 1 & 4 & 0\\
 0 & 0 & 4 \\
 1 & 1 & 1
\end{pmatrix}\begin{pmatrix}
s_0 \\
s_1 \\
s_2
\end{pmatrix}\equiv 0 \Mod{\q}.
\end{align*}
The system has solutions $\{(0,0,0),\tfrac{\q}{2}(0,1,1)\}$ and hence the number of points is $$c_{(0,0,0)}+c_{\tfrac{\q}{2}(0,1,1)}=q-1.$$ 

\bigskip

\textbf{Two coordinates are zero and the other two are non-zero} 

\bigskip 

There are only two cases where we have solutions. The first one is $x_0=x_1=0$ and the equation becomes $x_2^3x_3+x_3^4=0.$ The corresponding system is \begin{align*}
\begin{pmatrix}
 4 & 0 \\
 0 & 4 \\
 1 & 1 
\end{pmatrix}\begin{pmatrix}
s_0 \\
s_1 
\end{pmatrix}\equiv 0 \Mod{\q}
\end{align*} and has solutions $\{(0,0),\tfrac{\q}{2}(1,1)\}$ and therefore the number of points is $c_{(0,0)}+c_{\tfrac{\q}{2}(1,1)}=0.$ The second case is $x_2=x_3=0$ whose corresponding equation is $x_0^3x_1+x_1^4=0$ and the system is
\begin{align*}
\begin{pmatrix*}
3 & 0 \\
1 & 4 \\
1 & 1 
\end{pmatrix*}\begin{pmatrix*}
s_0 \\
s_1
\end{pmatrix*} \equiv 0 \Mod{\q}. 
\end{align*}

The solution set is $\{(0,0)\}$ and hence the number of points is just 1.

\bigskip 

\textbf{Three coordinates are zero and the other one is non-zero} 

\bigskip 

There is a unique point satisfying this condition, namely $(1:0:0:0).$

\bigskip 

\item[(d)] Suppose that $q \equiv 1 \Mod{4}$ and $q \equiv 1 \Mod{3}.$ 

\bigskip 

\textbf{One coordinate is zero and all the others are non-zero} 

\bigskip 

The first case is $x_0=0$ which has equation $x_1^4+x_2^4+x_3^4=0.$ The system is given by
\begin{align*}
\begin{pmatrix}
4 & 0 & 0 \\
0 & 4 & 0 \\
0 & 0 & 4 \\
1 & 1 & 1 
\end{pmatrix}\begin{pmatrix}
s_0 \\
s_1 \\
s_2
\end{pmatrix}\equiv 0 \Mod{\q}
\end{align*} and its solution set is \begin{align*}
S&=\{(0,0,0),\tfrac{\q}{2}(1,0,1),\tfrac{\q}{2}(0,1,1),\tfrac{\q}{2}(1,1,0),\\&\tfrac{\q}{4}(3,0,1),\tfrac{\q}{4}(1,0,3),\tfrac{\q}{4}(0,1,3),\tfrac{\q}{4}(3,1,0),\\&\tfrac{\q}{4}(1,3,0),\tfrac{\q}{4}(0,3,1),\tfrac{\q}{4}(2,1,1),\tfrac{\q}{4}(1,1,2),\\&\tfrac{\q}{4}(1,2,1),\tfrac{\q}{4}(3,2,3),\tfrac{\q}{4}(2,3,3),\tfrac{\q}{4}(3,3,2)\}.    
\end{align*} Therefore, the number of points is \begin{align*}
q-5-6(-1)^{\frac{\q}{4}}+\tfrac{3}{q}g(\tfrac{\q}{2})(g(\tfrac{\q}{4})^2+g(\tfrac{3\q}{4})^2).    
\end{align*}

The second case is $x_1=0$ whose corresponding equation is $x_2^4+x_3^4=0.$ Its system is \begin{align*}
\begin{pmatrix}
 4 & 0 \\
 0 & 4 \\
 1 & 1 
\end{pmatrix}\begin{pmatrix}
s_0 \\
s_1 
\end{pmatrix}\equiv 0 \Mod{\q}
\end{align*} and has solutions $S=\{(0,0),\tfrac{\q}{2}(1,1),\tfrac{\q}{4}(1,3),\tfrac{\q}{4}(3,1)\},$ so the number of points is $$c_{(0,0)}+c_{\tfrac{\q}{2}(1,1)}=2(1+(-1)^{\frac{\q}{4}})(q-1).$$

The third and fourth cases are symmetric and give the same number of points. So we focus only on one of them, namely $x_2=0.$ It has equation $x_0^3x_1+x_1^4+x_3^4=0$ and system \begin{align*}
\begin{pmatrix}
 3 & 0 & 0\\
 1 & 4 & 0\\
 0 & 0 & 4 \\
 1 & 1 & 1
\end{pmatrix}\begin{pmatrix}
s_0 \\
s_1 \\
s_2
\end{pmatrix}\equiv 0 \Mod{\q}
\end{align*}

The system has solutions \begin{align*}
S&=\{(0,0,0),\tfrac{\q}{2}(0,1,1),\tfrac{\q}{4}(0,3,1),\tfrac{\q}{4}(0,1,3),\tfrac{\q}{12}(4,11,9),\tfrac{\q}{6}(4,5,3),\\&\tfrac{\q}{3}(1,2,0),\tfrac{\q}{12}(8,7,9),\tfrac{\q}{12}(4,5,3),\tfrac{\q}{3}(2,1,0),\tfrac{\q}{6}(2,1,3),\tfrac{\q}{12}(8,1,3)\}    
\end{align*}and consequently the number of points is \begin{align*}
    q-5-2(-1)^{\frac{\q}{4}}&+\tfrac{1}{q}g(\tfrac{2\q}{3})g(\tfrac{5\q}{6})g(\tfrac{\q}{2})+\tfrac{1}{q}g(\tfrac{\q}{3})g(\tfrac{\q}{6})g(\tfrac{\q}{2})+\tfrac{1}{q}g(\tfrac{\q}{3})g(\tfrac{11\q}{12})g(\tfrac{3\q}{4})\\&+\tfrac{1}{q}g(\tfrac{2\q}{3})g(\tfrac{7\q}{12})g(\tfrac{3\q}{4})+\tfrac{1}{q}g(\tfrac{\q}{3})g(\tfrac{\q}{6})g(\tfrac{\q}{4})+\tfrac{1}{q}g(\tfrac{2\q}{3})g(\tfrac{\q}{12})g(\tfrac{\q}{4}).
\end{align*} 

\bigskip

\textbf{Two coordinates are zero and the other two are non-zero} 

\bigskip 

Two of the cases have no solution, namely $x_1=x_2=0$ and $x_1=x_3=0.$ Therefore, we focus on the other cases. The three cases $x_0=x_1=0,$ $x_0=x_2=0$ and $x_0=x_3=0$ are very similar and each yield $2(1+(-1)^{\frac{\q}{4}})$ points. Therefore, we show only the first of them, which is $x_0=x_1=0.$ The equation becomes $x_2^3x_3+x_3^4=0$ and the system is given by \begin{align*}
\begin{pmatrix}
 4 & 0 \\
 0 & 4 \\
 1 & 1 
\end{pmatrix}\begin{pmatrix}
s_0 \\
s_1 
\end{pmatrix}\equiv 0 \Mod{\q}.
\end{align*} The solutions of the system are $\{(0,0),\tfrac{\q}{2}(1,1), \tfrac{\q}{4}(1,3),\tfrac{\q}{4}(3,1)\}$ and hence the number of points is $2(1+(-1)^{\frac{\q}{4}}).$ The last case is $x_2=x_3=0.$ It has equation $x_0^3x_1+x_1^4=0$ and system
\begin{align*}
\begin{pmatrix*}
3 & 0 \\
1 & 4 \\
1 & 1 
\end{pmatrix*}\begin{pmatrix*}
s_0 \\
s_1
\end{pmatrix*} \equiv 0 \Mod{\q}. 
\end{align*} The solution set is $\{(0,0),\tfrac{\q}{3}(2,1),\tfrac{\q}{3}(1,2)\}$ and the number of points is 3.

\bigskip

\textbf{Three coordinates are zero and the other one is non-zero} 

\bigskip 

\noindent Again, the only point that satisfies this condition is $(1:0:0:0).$

\bigskip 

Therefore, \begin{align*}
    \#X_{\psi}(\F_q)-\#U_{\psi}(\F_q)&=(5+2(-1)^{\frac{\q}{4}})q-7-6(-1)^{\frac{\q}{4}}+\tfrac{3}{q}g(\tfrac{\q}{2})(g(\tfrac{\q}{4})^2+g(\tfrac{3\q}{4})^2)\\&+\tfrac{2}{q}g(\tfrac{\q}{3})g(\tfrac{11\q}{12})g(\tfrac{3\q}{4})+\tfrac{2}{q}g(\tfrac{2\q}{3})g(\tfrac{5\q}{6})g(\tfrac{\q}{2})+\tfrac{2}{q}g(\tfrac{2\q}{3})g(\tfrac{7\q}{12})g(\tfrac{3\q}{4})\\&+\tfrac{2}{q}g(\tfrac{\q}{3})g(\tfrac{5\q}{12})g(\tfrac{\q}{4})+\tfrac{2}{q}g(\tfrac{\q}{3})g(\tfrac{\q}{6})g(\tfrac{\q}{2})+\tfrac{2}{q}g(\tfrac{2\q}{3})g(\tfrac{\q}{12})g(\tfrac{\q}{4}).   
    \end{align*}
And using Lemma \ref{lemma: JacobiGauss}, we reformulate the previous as \begin{align*}
    \#X_{\psi}(\F_q)-\#U_{\psi}(\F_q)&=(5+2(-1)^{\frac{\q}{4}})q-7-6(-1)^{\frac{\q}{4}}+2\sum_{i \in (\Z/12\Z)^{\times}}J_q\left(\tfrac{i}{12},\tfrac{i}{4}\right)\\&+3\left(J_q\left(\tfrac{1}{4},\tfrac{1}{4}\right)+J_q\left(\tfrac{3}{4},\tfrac{3}{4}\right)\right)+2\left(J_q\left(\tfrac{1}{2},\tfrac{1}{3}\right)+J_q\left(\tfrac{1}{2},\tfrac{2}{3}\right)\right).   
    \end{align*}
\end{itemize}

Combining Lemmas \ref{lemma: lemmainsidec2f2} and \ref{lemma: lemmaoutsidec2f2}, we obtain the following point count for $\Csf_2\Fsf_2.$ 

\begin{prop}\label{prop: propc2f2}
If $q \not \equiv 1 \Mod{4},$ then \begin{align*}
\#X_{\psi}(\F_q)=q^2+1+H_q(\bm{\alpha}_{\clubsuit_0},\bm{\beta}_{\clubsuit_0}\mid t_{\clubsuit_0})-qH_q(\bm{\alpha}_{\clubsuit_1},\bm{\beta}_{\clubsuit_1}\mid t_{\clubsuit_1}).    
\end{align*}
If $q \equiv 1 \Mod{4},$ then \begin{align*}
\#X_{\psi}(\F_q)&=q^2+2(1+(-1)^{\frac{\q}{4}})q+1+H_q(\bm{\alpha}_{\clubsuit_0},\bm{\beta}_{\clubsuit_0}\mid t_{\clubsuit_0})+qH_q(\bm{\alpha}_{\clubsuit_1},\bm{\beta}_{\clubsuit_1}\mid t_{\clubsuit_1})+\\&(-1)^{\frac{\q}{4}}q(F_q((4,2,3,3,-12,1,-1),(0,-1,0,1,0,0,0)),4 \mid t_{\clubsuit_2})\\&+F_q((4,2,3,3,-12,1,-1),(0,1,0,-1,0,0,0),4 \mid t_{\clubsuit_3})). 
\end{align*}   
\end{prop}

\end{proof}

\subsubsection{The remaining families}

For the pencils $\Csf_3\Fsf_1,\Csf_2\Lsf_2$ and $\Csf_2\Csf_2,$ using similar techniques, we can also obtain formulas for the point counts. An important detail to notice is that for $\Csf_2\Lsf_2$ and $\Csf_2\Csf_2,$ when identifying the hypergeometric sums $H_q(\bm{\alpha}_{\clubsuit_4},\bm{\beta}_{\clubsuit_4}\mid t_{\clubsuit_4})$ and $H_q(\bm{\alpha}_{\symking_1},\bm{\beta}_{\symking_1}\mid t_{\symking_1})$ respectively, we perform repeated applications of the Hasse-Davenport relation and use Lemma \ref{hd} to simplify the expressions. Apart from this observation, the procedure is similar to $\Csf_4$ and $\Csf_2\Fsf_2.$ 

\begin{prop}\label{prop: remainingpc}
Let $\psi \in \F_q,$ then
\begin{align*}
    \#X_{\Csf_3\Fsf_1,\psi}(\F_q)=q^2+q+1+q\begin{cases}
3 &\text{ if } q \equiv 1 \Mod{8}\\
-1 &\text{ if } q\not \equiv 1 \Mod{8} 
\end{cases}+H_q(\bm{\alpha}_{\spadesuit_0},\bm{\beta}_{\spadesuit_0}\mid t_{\spadesuit_0}).
\end{align*}

\begin{align*}
\#X_{\Csf_2\Lsf_2,\psi}(\F_q)=q^2+2q+1&+4q\delta[q \equiv 1 \Mod{3}]+2q\delta[q \equiv 1 \Mod{4}]\\&+H_q(\bm{\alpha}_{\clubsuit_0},\bm{\beta}_{\clubsuit_0}\mid t_{\clubsuit_0})+\omega(-12\psi)^{\frac{\q}{2}}qH_q(\bm{\alpha}_{\clubsuit_4},\bm{\beta}_{\clubsuit_4}\mid t_{\clubsuit_4}).    
\end{align*}
\begin{align*}
\#X_{\Csf_2\Csf_2,\psi}(\F_q)=q^2+6q+1+&H_q(\bm{\alpha}_{\symking_0},\bm{\beta}_{\symking_0}\mid t_{\symking_0})+\omega(-6\psi)^{\frac{\q}{2}}qH_q(\bm{\alpha}_{\symking_1},\bm{\beta}_{\symking_1}\mid t_{\symking_1})\\+\delta[q \equiv 1 \Mod{3}]q\big(&F_q((2,1,2,1,-6),(0,0,1,-1,0),3\mid t_{\symking_3})\\+&F_q((2,1,2,1,-6),(0,0,-1,1,0),3\mid t_{\symking_3})\big)\\+\delta[q \equiv 1 \Mod{3}]\omega(-6\psi)^{\frac{\q}{2}}q\big(&F_q((2,1,2,1,-6),(3,0,1,-1,3),6\mid t_{\symking_3})\\+&F_q((2,1,2,1,-6),(3,0,-1,1,3),6\mid t_{\symking_3})\big).
\end{align*}
\end{prop}

\section{L-series}\label{S:Lseries}

In this section, we use our finite field hypergeometric point count formulas to prove Theorem~\ref{T:hypergeometricMatch}, which shows that the hypergeometric parameters in the finite field formulas match parameters in the Picard-Fuchs equations obtained in \S~\ref{S:PFequations}. We then apply the formulas to prove our \hyperref[T:main]{Main Theorem}. The Main Theorem shows that the $L$-function of each of the five Delsarte pencils decomposes into smaller factors associated to hypergeometric sums and to Dedekind zeta and $L$-functions. The proof follows from a direct manipulation of the explicit formula for the point counts given in \S~\ref{S:counting}.

\subsection{Hypergeometric parameter match}

Before we prove our main results, we come back to the relation between periods and point counts mentioned in the beginning of the introduction. We observe that a similar instance of that relation is manifested in the pencils of Delsarte hypersurfaces studied here. Let $\bullet$ denote one of the five families in Table ~\ref{table: introDelsarteList}. 

\begin{thm}\label{T:hypergeometricMatch}
There is a subspace $H^2_{\rm hyp}(X_{\bullet,\psi})$ of $H^2_{\text{prim}}(X_{\bullet,\psi},\C)$ of dimension $d_{\bullet}$ spanned by forms $\omega_{0,\psi}, \dots, \omega_{m_\bullet,\psi}$ and their derivatives with respect to $\psi$ such that each $\omega_{i,\psi}$ satisfies a hypergeometric Picard-Fuchs differential equation. Here $\omega_{0,\psi}=\omega_{\psi}$ is the holomorphic form of $X_{\bullet,\psi}$ and ${m_\bullet}$ is an integer depending on the family. Let $M_\bullet$ be the integer 2 for $\Csf_2\Lsf_2$, $\Csf_3\Fsf_1$, and $\Csf_4$, 4 for $\Csf_2\Fsf_2$, and 6 for $\Csf_2\Csf_2$. Then we may choose the generators $\omega_{i,\psi}$ such that when $q \equiv 1 \Mod{M_\bullet}$, they are in one-to-one correspondence with finite field hypergeometric summands of the form $H_q(\bm{\alpha},\bm{\beta} \mid t)$ in $\#X_{\bullet,\psi}(\F_q).$ 
More precisely, for each summand $H_q(\bm{\alpha},\bm{\beta} \mid t)$ in $\#X_{\bullet,\psi}(\F_q)$, we have a family of generators $\omega_{\psi,i}$ such that $ \psi \mapsto \int_{\gamma_{\psi}}\omega_{\psi,i}$ is a local solution at $\psi=\infty$ of $D(\bm{\alpha},\bm{\beta} \mid t)$ for some cycle $\gamma_{\psi} \in H_2(X_{\bullet,\psi},\C)$. Here the parameters $\bm{\alpha},\bm{\beta}$ appearing in $H_q(\bm{\alpha},\bm{\beta} \mid t)$ and $D(\bm{\alpha},\bm{\beta} \mid t)$ are the same for each summand-generator pair.
\end{thm}

\begin{proof}
We take $M_\bullet$ to be the order of the group of symmetries of the pencil, as given in Table~\ref{table: classif}, if the group is non-trivial; when the group of symmetries is trivial, we simply impose the condition that $q$ is odd.
Theorem~\ref{T:hypPeriodList} summarizes the hypergeometric parameters for the Picard-Fuchs differential equations. On comparing with Propositions~\ref{prop: propc4}, \ref{prop: propc2f2}, and \ref{prop: remainingpc}, which describe the point counts over finite fields, the desired theorem follows.
\end{proof}

\subsection{$L$-functions of K3 surfaces} 

Following \cite{HD20}, we define $L$-series of K3 surfaces and hypergeometric pairs, and extend the definition to include gamma triples as defined in \cite{AA25}. 
Let $\psi\in\mathbb{Q}\setminus \{0,1\}$ and let $\bullet \in \{\Csf_2\Fsf_2, \Csf_2\Lsf_2,\Csf_2\Csf_2,\Csf_3\Fsf_1,\Csf_4\}$ indicate one of the K3 pencils in Table \ref{table: introDelsarteList}. Recall that $S(\bullet,\psi)$ consists of the set of bad primes in Table \ref{table: introDelsarteList} together with those primes that divide the numerator or denominator of either $t$ or $t-1,$ where $t$ is as in Table \ref{table: alphaBetaSubscripts}.

Let $p\notin S(\bullet,\psi)$; then the zeta function of $X_{\bullet, \psi}$ over $\mathbb{F}_p$ is of the form
 \[
 Z_p(X_{\bullet,\psi},T)=\dfrac{1}{(1-T)(1-pT)(1-p^2T)P_{\bullet,\psi,p}(T)}
 \] where $P_{\bullet,\psi,p}(T)\in 1+T\mathbb{Z}[T]$ and has degree 21. For each pencil $\bullet$ and parameter $\psi$, we may thus define an incomplete $L$-series 
 \begin{equation}
 L_{S(\bullet,\psi)}(X_{\bullet,\psi},s):=\prod_{p\notin S(\bullet,\psi)}P_{\bullet, \psi, p}(p^{-s})^{-1}.
 \end{equation}
The incomplete $L$-series $L_{S(\bullet,\psi)}(X_{\bullet,\psi},s)$ is convergent for $s\in\mathbb{C}$ in a right half-plane.

\bigskip 

\subsection{L-functions of hypergeometric pairs and gamma triples}
\label{sub: subL}
Let $\bm{\alpha},\bm{\beta}$ be multisets of $d$ rational numbers and suppose that $\alpha_i-\beta_j \not \in \mathbb{Z}$ for any pair $i,j \in \{1,\dots,d\}$ and $t \in \mathbb{Q} \setminus \{0,1\}.$ We define $S(\bm{\alpha},\bm{\beta},t)$ to be the set of primes dividing a denominator in $\bm{\alpha} \cup \bm{\beta}$ or the numerator or denominator of $t$ or $t-1,$ which we will call ``bad primes'' for the associated hypergeometric function $H(\bm{\alpha},\bm{\beta}\mid t)$. 

\begin{definition}[\cite{HD20}]\label{def: Lq}
 For a prime power $q$ and a pair $\bm{\alpha},\bm{\beta}$ such that $q$ is good for the pair, we define the formal series
 \[
 L_q(H(\bm{\alpha},\bm{\beta}\mid t),T):=\exp\left(-\sum_{r=1}^{\infty}H_{q^r}(\bm{\alpha},\bm{\beta}\mid t)\dfrac{T^r}{r}\right)\in 1+TK_{\bm{\alpha},\bm{\beta}}[[T]]. 
 \]  

\begin{remark}
  Using Table \ref{table: introDelsarteList} and Table \ref{table: alphaBetaSubscripts}, we notice that the bad primes for the K3 surface are contained in the set of bad primes for the associated hypergeometric functions. Thus, we will refer to the latter when discussing ``bad primes''. 
\end{remark} 

 For a gamma triple $(\bm{\gamma},\bm{\delta}, N),$ we define  
 $$L_q(F(\bm{\gamma},\bm{\delta}, N\mid t),T):=\exp\left(-\sum_{r=1}^{\infty}\delta[q^r \equiv 1 \Mod{N}]F_{q^r}(\bm{\gamma},\bm{\delta}, N\mid t)\dfrac{T^r}{r}\right)\in 1+T\Q(\zeta_N)[[T]].$$

\noindent (Notice that if $N \nmid q^r,$ then $F_{q^r}(\bm{\gamma},\bm{\delta}, N\mid t)$ is not defined, but $\delta[q^m \equiv 1 \Mod{N}]=0.$)

 If $\bm{\alpha},\bm{\beta}$ are defined over $\Q,$ define \begin{align}\label{al: defQ}
L_S(H(\bm{\alpha},\bm{\beta}\mid t),s):=\prod_{p \notin S}L_p(H(\bm{\alpha},\bm{\beta}\mid t),p^{-s})^{-1}.    
 \end{align} 

 \end{definition}
 
 We now follow Section 4 of \cite{HD20} to define $L$-functions for the hypergeometric sum $F_q.$ Let $M=\Q(\zeta_N).$ Observe that this is a particular case of the description in \cite[Section 4]{HD20}, where they allow $M$ to be a more general extension of $\Q.$ Consider the identification $(\Z/N\Z)^{\times} \simeq \text{Gal}(\Q(\zeta_N) \mid \Q)$ given by $k \mapsto \sigma_k$ and $\sigma_k(\zeta_N)=\zeta_N^k.$ In particular, in the notation of \cite[Section 4]{HD20}, $H_M$ in our case is just the trivial group. 

Fix a prime $p \notin S(\bm{\alpha},\bm{\beta},t).$ Denote by $\mathfrak{p}_1,\dots,\mathfrak{p}_r$ the primes above $p$ in $\Z[\zeta_N]$ and $p^f$ the norm of each of those primes (it is well defined because the extension is Galois, so every prime above $p$ has the same norm). If $p \nmid N,$ that is, if the prime $p$ does not ramify, then $f$ is the multiplicative order of $p$ modulo $N$ and $p$ splits into $\varphi(N)/f$ distinct primes, that is, $r=\varphi(N)/f.$ Therefore, the primes $\mathfrak{p}_i$ are indexed by the elements of the group $(\Z/N\Z)^{\times}/\langle p \rangle.$ Bearing in mind this notation, we have the following definition:
 
\begin{definition}\label{def: Lp}
 Define the $L$-factor of the gamma triple $(\bm{\gamma},\bm{\delta},N)$ at $p$ and the incomplete $L$-function respectively as \begin{align}\label{alg: Lpp}
 L_p(F(\bm{\gamma},\bm{\delta},N \mid t),\Q(\zeta_N),T)=\prod_{k_i \in (\Z/N\Z)^{\times}/\langle p\rangle}L_{p^f}(F(\bm{\gamma},k_i\bm{\delta}, N\mid t),T^f) 
 \end{align}  \begin{align}\label{alg: LprodF}
 L_S(F(\bm{\gamma},\bm{\delta},N \mid t),\Q(\zeta_N),s)&=\prod_{p \notin S}L_p(F(\bm{\gamma},\bm{\delta},N \mid t),\Q(\zeta_N),p^{-s})^{-1}.    
 \end{align}  
 \end{definition}

As we will see in the computations, we also need to consider twists by characters. For a prime $p$ and $a\in \F_p^{\times},$ define $\phi_{a}(p)=(\frac{a}{p})$ (that is, the Legendre symbol of $a$ modulo $p$) and extend it multiplicatively for powers of $p.$ For example, when $a=-1,$ we have $\phi_{-1}(p^l)=(\frac{-1}{p})^l=(-1)^{\frac{p^l-1}{2}}.$

\begin{comment}
For a prime $p,$ we we also define $$\phi_{\sqrt{-1}}(p^{l})=(-1)^{\frac{p^l-1}{4}}=\begin{cases}
1, \text{ if } p^l \equiv 1 \Mod{8}\\
-1, \text{otherwise}
\end{cases}.$$    
\end{comment}

For a prime $p$ and a prime ideal $\mathfrak{p} \in \mathrm{Spec}(\Z[i])$ above $p,$ we have $\mathrm{Nm}(\mathfrak{p}) \equiv 1 \Mod{4}.$ Therefore, the following function is well defined and depends only on the norm of $\mathfrak{p}$  $$\phi_{\sqrt{-1}}(\mathfrak{p})=(-1)^{\frac{\mathrm{Nm}(\mathfrak{p})-1}{4}}.$$ 

Now, recall that in Lemma \ref{lemma:char} we proved that for a prime power $q \equiv 1 \Mod{4},$ we have $\omega(2)^{\frac{\q}{2}}=(-1)^{\frac{\q}{4}},$ where $\omega$ is a generator of the character group of $\F_q^{\times}.$ In this context, denoting by $q$ the norm of $p,$ we write $$\phi_{\sqrt{-1}}(q)=(-1)^{\frac{\q}{4}}.$$

Finally, expanding \eqref{al: defQ} as a Dirichlet series, we can define the twisting by a character as done in \cite[(4.1.16)]{HD20}. Indeed, we may write $L_S(H(\bm{\alpha},\bm{\beta}\mid t),s)=\sum_{n \in \Z}\frac{a_n}{n^s}$ and for a Dirichlet character $\chi: \Z \to \C$ define $$L_S(H(\bm{\alpha},\bm{\beta}\mid t),s,\chi)=\sum_{n=1}^{\infty}\frac{a_n\chi(n)}{n^s}.$$ 

We can also expand \eqref{alg: LprodF} as a Dirichlet series in the following way 
    $$\sum_{\substack{\mathfrak{n} \subset \Z[\zeta_N] \\ \mathfrak{n} \neq (0)}}\frac{a_{\mathfrak{n}}}{\text{Nm}(\mathfrak{n})^s}.$$ Given a finite order Dirichlet character $\chi$ over $\Q(\zeta_N),$ we define \begin{align*}L_S(H(\bm{\alpha},\bm{\beta}\mid t),\mathbb{Q}(\zeta_N),s,\chi)=\sum_{\substack{\mathfrak{n} \subset \Z[\zeta_N] \\ \mathfrak{n} \neq (0)}}\frac{\chi(\mathfrak{n}) a_{\mathfrak{n}}}{\mathrm{Nm}(\mathfrak{n})^s}.    
\end{align*} 

\subsection{Computing the L-functions}
Now we compute the incomplete $L$-functions of the five Delsarte pencils according to the definitions given in the previous section. 

\begin{theorem}

\begin{align*}
&L_S(X_{\Csf_4,\psi},s)=\zeta_{S,\Q(i\sqrt{3})}(s-1)\zeta_{S}(s-1)\cdot L_S\left(H(\bm{\alpha}_{\heartsuit_0},\bm{\beta}_{\heartsuit_0}\mid t_{\heartsuit_0}),s\right).
\end{align*} 
\end{theorem}
\begin{proof}
By definition, $L_{S}(X_{\Csf_4,\psi},s)=\prod_{p \notin S}P_{\psi}(p^{-s})^{-1}.$ We also have \begin{align*}
\frac{1}{P_{\psi}(T)}&=(1-T)(1-qT)(1-q^2T)Z(X_{\Csf_4,\psi},T)=(1-T)(1-qT)(1-q^2T)\sum_{m=1}^{\infty}\#X_{\Csf_4,\psi}(\F_{q^m})\frac{T^m}{m}.    
\end{align*} 

We proved that \begin{align*}
\#X_{\Csf_4,\psi}(\F_q)=q^2+2q+1+2q\delta[q\equiv 1 \Mod 3]+H_q(\bm{\alpha}_{\heartsuit_0},\bm{\beta}_{\heartsuit_0}\mid t_{\heartsuit_0}).  
\end{align*}

The factor $(1-T)(1-qT)(1-q^2T),$ which corresponds to the hyperplane section, cancels, and we obtain the following: \begin{align}\label{al: algn}
\frac{1}{P_{\psi}(p^{-s})}=&\exp\left(\sum_{m=1}^{\infty}2\delta[p^m \equiv 1 \Mod 3]\frac{p^{(1-s)m}}{m}\right)\exp\left(\sum_{m=1}^{\infty}\frac{p^{(1-s)m}}{m}\right) \exp\left(\sum_{m=1}^{\infty}H_{p^m}(\bm{\alpha}_{\heartsuit_0},\bm{\beta}_{\heartsuit_0}\mid t_{\heartsuit_0})\frac{p^{-sm}}{m}\right).    
\end{align}

Now we compute $$\exp\left(\sum_{m=1}^{\infty}2\delta[p^m \equiv 1 \Mod 3]\frac{p^{(1-s)m}}{m}\right).$$ 

If $p \equiv 1,7 \Mod {12},$ then $p^m \equiv 1,7 \Mod {12},$ respectively and therefore $p^m \equiv 1 \Mod 3.$ So we have \begin{align*}
\exp\left(\sum_{m=1}^{\infty}2\delta[p^m \equiv 1 \Mod 3]\frac{p^{(1-s)m}}{m}\right)&=\exp\left(-2\sum_{m=1}^{\infty}\frac{p^{(1-s)m}}{m}\right)^{-1}=(1-p^{1-s})^{-2}.
\end{align*}

If $p \equiv 5,11 \Mod{12},$ then $p^{2m} \equiv 1 \Mod 3$ and $p^{2m+1} \not\equiv 1 \Mod 3.$ Therefore, \begin{align*}
\exp\left(\sum_{m=1}^{\infty}2\delta[p^m \equiv 1 \Mod 3]\frac{p^{(1-s)m}}{m}\right)=\exp \left(-\sum_{m=1}^{\infty}\frac{p^{2(1-s)m}}{m} \right)^{-1}=(1-p^{2(1-s)})^{-1}.   
\end{align*}

We conclude that \begin{align*}
\prod_{p \notin S}\exp\left(\sum_{m=1}^{\infty}2p^m\delta[p^m \equiv 1\Mod 3]\frac{p^{-sm}}{m}\right)=\zeta_{S,\Q(\sqrt{-3})}(s-1). 
\end{align*}

One can also easily see that  $$\prod_{p\notin S}\exp\left(\sum_{m=1}^{\infty}\frac{p^{(1-s)m}}{m}\right)=\zeta_S(s-1).$$

We now take the product over $p \notin S$ in equation \eqref{al: algn}, we obtain 

\begin{align*}
&L_S(X_{\Csf_4,\psi},s)=\zeta_{S,\Q(i\sqrt{3})}(s-1)\zeta_{S}(s-1)\cdot L_S\left(H(\bm{\alpha_{\heartsuit_0}},\bm{\beta}_{\heartsuit_0}\mid t_{\heartsuit_0}),s\right).
\end{align*}  
\end{proof}

\begin{thm}
\begin{align*}
L_S(X_{\Csf_2\Fsf_2,\psi},s)=&L_S(\Q(\zeta_8)|\Q,s)\cdot L_S(H(\bm{\alpha}_{\clubsuit_0},\bm{\beta}_{\clubsuit_0}\mid t_{\clubsuit_0}),s)\cdot L_S(H(\bm{\alpha}_{\clubsuit_1},\bm{\beta}_{\clubsuit_1}\mid t_{\clubsuit_1}),s-1,\phi_{-1})\\&L_S(F((4,2,3,3,-12,1,-1),(0,-1,0,1,0,0,0),4 \mid t_{\clubsuit_2}),\Q(i),s-1,\phi_{\sqrt{-1}}).    
\end{align*}   
\end{thm}
\begin{proof}    
By Proposition \ref{prop: propc2f2}, we have  
\begin{align*}
\#X_{\Csf_2\Fsf_2,\psi}(\F_q)&=q^2+q+1+q\begin{cases}
3 \text{ if } q \equiv 1 \Mod{8}\\
-1 \text{ if } q\not \equiv 1 \Mod{8} \end{cases}+H_q(\bm{\alpha}_{\clubsuit_0},\bm{\beta}_{\clubsuit_0}\mid t_{\clubsuit_0})+\phi_{-1}(q)qH_q(\bm{\alpha}_{\clubsuit_1},\bm{\beta}_{\clubsuit_1}\mid t_{\clubsuit_1})\\&+\delta[q \equiv 1 \Mod{4}]\phi_{\sqrt{-1}}(q)q(F_q((4,2,3,3,-12,1,-1),(0,-1,0,1,0,0,0),4 \mid t_{\clubsuit_2})\\&+F_q((4,2,3,3,-12,1,-1),(0,1,0,-1,0,0,0),4 \mid t_{\clubsuit_3})).    
\end{align*}

By definition, $L_{S}(X_{\Csf_2\Fsf_2,\psi},s)=\prod_{p \notin S}P_{\psi}(p^{-s})^{-1}.$ We also have \begin{align*}
\frac{1}{P_{\psi}(T)}&=(1-T)(1-qT)(1-q^2T)Z(X_{\Csf_2\Fsf_2,\psi},T)=(1-T)(1-qT)(1-q^2T)\sum_{m=1}^{\infty}\#X_{\Csf_2\Fsf_2,\psi}(\F_{q^m})\frac{T^m}{m}.    
\end{align*} 

The factor $(1-T)(1-qT)(1-q^2T^2),$ which corresponds to the hyperplane section, cancels, so we only need to handle the remaining terms. The term \begin{align*}
q\begin{cases}
3 &\text{ if } q \equiv 1 \Mod{8} \\
-1 &\text{ if } q\not \equiv 1 \Mod{8}
\end{cases}.    
\end{align*} is handled in \cite[Proposition 4.4.1, equation (4.4.3)]{HD20}, and gives $L_{S}(\Q(\zeta_8)\mid \Q,s-1).$ 

The term $H_p(\bm{\alpha}_{\clubsuit_0},\bm{\beta}_{\clubsuit_0}\mid t_{\clubsuit_0})$ corresponds to $L_{S}(H(\bm{\alpha}_{\clubsuit_0},\bm{\beta}_{\clubsuit_0}\mid t_{\clubsuit_0}),s)$ and $\phi_{-1}(p)pH_p(\bm{\alpha}_{\clubsuit_1},\bm{\beta}_{\clubsuit_1}\mid t_{\clubsuit_1})$ to $L_{S}(H(\bm{\alpha}_{\clubsuit_4},\bm{\beta}_{\clubsuit_4}\mid t_{\clubsuit_4}),s-1,\phi_{-12}\phi_{\psi}).$

All the terms in the sums were handled before, except the last one: 

\begin{align*}
\delta[q \equiv 1 \Mod{4}]\phi_{\sqrt{-1}}(q)q(&F_q((4,2,3,3,-12,1,-1)),(0,-1,0,1,0,0,0),4 \mid t_{\clubsuit_2})\\+&F_q((4,2,3,3,-12,1,-1)),(0,1,0,-1,0,0,0),4 \mid t_{\clubsuit_3})).    
\end{align*}

Notice that this term only appears when $q \equiv 1 \Mod{4},$ so we need to investigate the powers of $p$ according to this. Here we have the extension $M=\Q(\zeta_4)=\Q(i)$ of $\Q.$  

Suppose that $p \equiv 1 \Mod 4,$ then $p$ splits into two in $\Q(i)$ and each of them has norm $p.$ Notice that the group action in this case is given by \begin{align*}
(\Z/4\Z)^{\times}/\langle p \rangle=\Z/2\Z.   
\end{align*} Therefore, \begin{align*}
\exp&\bigg(-\sum_{m=1}^{\infty}\phi_{\sqrt{-1}}(p^m)p^mF_{p^m}((4,2,3,3,-12,1,-1),(0,-1,0,1,0,0),4 \mid t_{\clubsuit_2})\tfrac{p^{-sm}}{m}\\&-\sum_{m=1}^{\infty}\phi_{\sqrt{-1}}(p^m)p^mF_{p^m}((4,2,3,3,-12,1,-1)),(0,1,0,-1,0,0),4 \mid t_{\clubsuit_2})\tfrac{p^{-sm}}{m}\bigg)\\=&L_p\left(F((4,2,3,3,-12,1,-1),(0,-1,0,1,0,0),4 \mid t_{\clubsuit_2}),p^{1-s},\phi_{\sqrt{-1}}\right)\\\cdot &L_p\left(F((4,2,3,3,-12,1,-1),(0,1,0,-1,0,0),4 \mid t_{\clubsuit_2}),p^{1-s},\phi_{\sqrt{-1}}\right)\\=&\prod_{k \in \Z/2\Z}L_p\left(F((4,2,3,3,-12,1,-1),k(0,-1,0,1,0,0),4 \mid t_{\clubsuit_2}),p^{1-s},\phi_{\sqrt{-1}}\right)\\=&L_p(F((4,2,3,3,-12,1,-1),(0,-1,0,1,0,0),4 \mid t_{\clubsuit_2}),\Q(i),p^{1-s},\phi_{\sqrt{-1}}).  
\end{align*} 

Suppose that $p \equiv 3 \Mod 4,$ then $p$ is inert in $\Q(i)$ and has inertia degree equal to $f=2.$ The group action in this case is trivial as \begin{align*}
(\Z/4\Z)^{\times}/\langle p \rangle=(\Z/2\Z)/(\Z/2\Z) \text{ is trivial. }    
\end{align*} Moreover, the odd powers of $p$ are 3 $\Mod{4}$ and the even powers are equivalent to $1 \Mod{4}.$ For $q=p^{2m},$ by Lemma \ref{lm: Fp} we have $$F_q((4,2,3,3,-12,1,-1),(0,-1,0,1,0,0),4 \mid t_{\clubsuit_2})=F_q((4,2,3,3,-12,1,-1),(0,1,0,-1,0,0),4 \mid t_{\clubsuit_2})$$ and here we have $t \in \F_p,$ so $t^p=t.$ Therefore, \begin{align*}
\exp&\bigg(-\sum_{m=1}^{\infty}\phi_{\sqrt{-1}}(p^{2m})p^{2m}F_{p^{2m}}((4,2,3,3,-12,1,-1),(0,-1,0,1,0,0),4 \mid t_{\clubsuit_2})\tfrac{p^{-2sm}}{m}\bigg)\\=&L_{p^2}\left(F((4,2,3,3,-12,1,-1),(0,-1,0,1,0,0),4 \mid t_{\clubsuit_2}),p^{1-s},\phi_{\sqrt{-1}}\right)\\=&L_{p}(F((4,2,3,3,-12,1,-1),(0,-1,0,1,0,0),4 \mid t_{\clubsuit_2}),\Q(i),p^{1-s},\phi_{\sqrt{-1}}).  
\end{align*}

Finally, after taking the product over all $p \notin S,$ we get \begin{align*}
&L_S(F((4,2,3,3,-12,1,-1),(0,-1,0,1,0,0),4 \mid t_{\clubsuit_2}),\Q(i),s-1,\phi_{\sqrt{-1}})\\=&\prod_{p \notin S}L_p(F((4,2,3,3,-12,1,-1),(0,-1,0,1,0,0),4 \mid t_{\clubsuit_2}),\Q(i),p^{1-s},\phi_{\sqrt{-1}}).    
\end{align*}
\end{proof}

\begin{thm}
\begin{align*}
L_S(X_{\Csf_2\Lsf_2,\psi},s)&=\zeta_{S,\Q(i\sqrt{3})}(s-1)^2\zeta_{S,\Q(i)}(s-1)\zeta_S(s-1)\cdot L_S\left(H(\bm{\alpha}_{\clubsuit_0},\bm{\beta}_{\clubsuit_0}\mid t_{\clubsuit_0}),s\right)\\&\cdot L_{S}\left(H\left(\bm{\alpha}_{\clubsuit_4},\bm{\beta}_{\clubsuit_4}\mid t_{\clubsuit_4}\right),s-1,\phi_{-12}\phi_{\psi}\right). 
\end{align*}
\end{thm}
\begin{proof}
By definition, $L_{S}(X_{\Csf_2\Lsf_2,\psi},s)=\prod_{p \notin S}P_{\psi}(p^{-s})^{-1}.$ We also have \begin{align*}
\frac{1}{P_{\psi}(T)}&=(1-T)(1-qT)(1-q^2T)Z(X_{\Csf_2\Lsf_2,\psi},T)=(1-T)(1-qT)(1-q^2T)\sum_{m=1}^{\infty}\#X_{\Csf_2\Lsf_2,\psi}(\F_{q^m})\frac{T^m}{m}.    
\end{align*} 

The factor $(1-T)(1-qT)(1-q^2T^2),$ which corresponds to the hyperplane section, cancels, and we obtain the following: \begin{align}\label{al: algn1}
\frac{1}{P_{\psi}(p^{-s})}=&\exp\left(\sum_{m=1}^{\infty}4\delta[p^m \equiv 1 \Mod 3]\frac{p^{(1-s)m}}{m}\right)\exp\left(2\delta[p^m \equiv 1 \Mod 4]\frac{p^{(1-s)m}}{m}\right)\notag\\&\exp\left(\sum_{m=1}^{\infty}\frac{p^{(1-s)m}}{m}\right) \exp\left(\sum_{m=1}^{\infty}H_{p^m}(\bm{\alpha}_{\clubsuit_0},\bm{\beta}_{\clubsuit_0}\mid t_{\clubsuit_0})\frac{p^{-sm}}{m}\right)\notag\\&\exp\left(\sum_{m=1}^{\infty}\left(\frac{-12\psi}{p^m}\right)H_{p^m}(\bm{\alpha}_{\clubsuit_4},\bm{\beta}_{\clubsuit_4}\mid t_{\clubsuit_4})\frac{p^{(1-s)m}}{m}\right).    
\end{align}

Suppose that $p \equiv 1 \Mod 4,$ then the factor of $\zeta_{\Q(i)}(s-1)$ corresponding to such a prime is given by $(1-p^{1-s})^{-2}.$ Moreover,\begin{align*}
\exp\left(\sum_{m=1}^{\infty}2p^m\delta[p^m \equiv 1\Mod 4]\frac{p^{-sm}}{m}\right)&=\left(\exp\left(-2\sum_{m=1}^{\infty}\frac{p^{-(s-1)m}}{m}\right)\right)^{-1}\\&=(\exp(\log(1-p^{1-s})^2))^{-1}\\&=(1-p^{1-s})^{-2}.  
\end{align*}

Suppose that $p \equiv 3 \Mod 4,$ then $p^{2n} \equiv 1 \Mod 4$ and $p^{2n+1} \equiv 3 \Mod 4.$ We then have \begin{align*}
\exp\left(\sum_{m=1}^{\infty}2p^m\delta[p^m \equiv 1\Mod 4]\frac{p^{-sm}}{m}\right)&=\exp\left(\sum_{n=1}^{\infty}2p^{2n}\frac{p^{-2sn}}{2n}\right)=\exp\left(\sum_{n=1}^{\infty}\frac{p^{2n(1-s)}}{n}\right)\\&=\left(\exp\left(-\sum_{m=1}^{\infty}\frac{p^{-2(s-1)m}}{m}\right)\right)^{-1}\\&=\exp(\log(1-p^{-2(s-1)}))^{-1}\\&=(1-p^{-2(s-1)})^{-1}.
\end{align*}

The factor of $\zeta_{\Q(i)}(s-1)$ corresponding to such type of primes is $1-(p^{1-s})^2.$ Therefore, we have that \begin{align*}
\prod_{p \notin S}\exp\left(\sum_{m=1}^{\infty}2p^m\delta[p^m \equiv 1\Mod 4]\frac{p^{-sm}}{m}\right)=\zeta_{S,\Q(i)}(s-1). 
\end{align*}

Now we compute $$\exp\left(\sum_{m=1}^{\infty}4\delta[p^m \equiv 1 \Mod 3]\frac{p^{(1-s)m}}{m}\right).$$ If $p \equiv 1,7 \Mod {12},$ then $p^m \equiv 1,7 \Mod {12},$ respectively and therefore $p^m \equiv 1 \Mod 3.$ So we have \begin{align*}
\exp\left(\sum_{m=1}^{\infty}4\delta[p^m \equiv 1 \Mod 3]\frac{p^{(1-s)m}}{m}\right)&=\exp\left(-4\sum_{m=1}^{\infty}\frac{p^{(1-s)m}}{m}\right)^{-1}\\&=\exp \left(\log((1-p^{1-s})^4)\right)^{-1}=(1-p^{1-s})^{-4}.
\end{align*}

If $p \equiv 5,11 \Mod{12},$ then $p^{2m} \equiv 1 \Mod 3$ and $p^{2m+1} \not\equiv 1 \Mod 3.$ Therefore, \begin{align*}
\exp\left(\sum_{m=1}^{\infty}4\delta[p^m \equiv 1 \Mod 3]\frac{p^{(1-s)m}}{m}\right)&=\exp \left(\sum_{m=1}^{\infty}\frac{4p^{(1-s)2m}}{2m}\right)=\exp \left(-2\sum_{m=1}^{\infty}\frac{p^{2(1-s)m}}{m} \right)^{-1}\\&=\exp(\log(1-p^{2(1-s)})^2)^{-1}=(1-p^{2(1-s)})^{-2}.   
\end{align*}

We conclude that \begin{align*}
\prod_{p \notin S}\exp\left(\sum_{m=1}^{\infty}4p^m\delta[p^m \equiv 1\Mod 3]\frac{p^{-sm}}{m}\right)=\zeta_{S,\Q(\sqrt{-3})}(s-1). 
\end{align*}

One can also easily see that  $$\prod_{p\notin S}\exp\left(\sum_{m=1}^{\infty}\frac{p^{(1-s)m}}{m}\right)=\zeta_S(s-1).$$

We now taking the product over $p \notin S$ in \eqref{al: algn1}, we obtain 
\begin{align*}
    L_{S}(X_{\Csf_2\Lsf_2,\psi},s)=&\zeta_{S,\Q(\sqrt{3}i)}(s-1)^2\zeta_{S,\Q(i)}(s-1)\zeta_S(s-1)\cdot\\&L_{S}(H(\bm{\alpha}_{\clubsuit_0},\bm{\beta}_{\clubsuit_0}\mid t_{\clubsuit_0}),s)L_{S}(H(\bm{\alpha}_{\clubsuit_4},\bm{\beta}_{\clubsuit_4}\mid t_{\clubsuit_4}),s-1,\phi_{-12}\phi_{\psi}).
\end{align*}  
\end{proof}

\begin{thm}
\begin{align*}
L_S(X_{\Csf_3\Fsf_1,\psi},s)&=L_{S}(\Q(\zeta_8)|\Q,s-1)\cdot L_S\left(H(\bm{\alpha}_{\spadesuit_0},\bm{\beta}_{\spadesuit_0} \mid t_{\spadesuit_0}),s\right). 
\end{align*}
\end{thm}
\begin{proof}
We have that \begin{align*}
    \#X_{\Csf_3\Fsf_1}(\F_q)=q^2+q+1+q\begin{cases}
3 &\text{ if } q \equiv 1 \Mod{8}\\
-1 &\text{ if } q\not \equiv 1 \Mod{8} 
\end{cases}+H_q(\bm{\alpha}_{\spadesuit_0},\bm{\beta}_{\spadesuit_0} \mid t_{\spadesuit_0}).
\end{align*} 

The first summand $q^2+q+1$ cancels as in the previous cases, and the last summand $H_q(\bm{\alpha}_{\spadesuit_0},\bm{\beta}_{\spadesuit_0} \mid t_{\spadesuit_0})$ gives the $L$-function $L_{S}(H(\bm{\alpha}_{\spadesuit_0},\bm{\beta}_{\spadesuit_0} \mid t_{\spadesuit_0}),s).$ Therefore, we only need to look at the remaining term \begin{align*}
q\begin{cases}
3 &\text{ if } q \equiv 1 \Mod{8} \\
-1 &\text{ if } q\not \equiv 1 \Mod{8}
\end{cases}.    
\end{align*} By \cite[Proposition 4.4.1, equation (4.4.3)]{HD20}, we see that this summand gives $L_{S}(\Q(\zeta_8)\mid \Q,s-1).$   
\end{proof}

%\subsubsection{The family $\Csf_2\Csf_2$}
\begin{theorem}
\begin{align*}
L_S(X_{\Csf_2\Csf_2,\psi},s)&=\zeta_S(s-1) \cdot \zeta_{S,\Q(i\sqrt{3})}(s-1)^2 \cdot L_S(H(\bm{\alpha}_{\symking_0},\bm{\beta}_{\symking_0}\mid t_{\symking_0}),s)\\& \cdot L_S(H(\bm{\alpha}_{\symking_1},\bm{\beta}_{\symking_1}\mid t_{\symking_1}),s-1,\phi_{-6\psi}) \\&\cdot L_S(F((2,1,2,1,-6),(0,0,1,-1,0),3\mid t_{\symking_3}),\Q(\sqrt{-3}),s-1)\\&\cdot L_S(F((2,1,2,1,-6),(3,0,-1,1,3),6\mid t_{\symking_3}),\Q(\zeta_6),s-1,\phi_{-6\psi}).
\end{align*}
\end{theorem}
\begin{proof}
We have \begin{align*}
\#X_{\Csf_2\Csf_2,\psi}(\F_q)&=q^2+2q+1+4q\delta[q \equiv 1 \Mod{3}]+H_q(\bm{\alpha}_{\symking_0},\bm{\beta}_{\symking_0}\mid t_{\symking_0})+\phi_{-6\psi}(q)qH_q(\bm{\alpha}_{\symking_1},\bm{\beta}_{\symking_1}\mid t_{\symking_1})\\&+\delta[q \equiv 1 \Mod{3}]q(F_q((2,1,2,1,-6),(0,0,-1,1,0),3\mid t_{\symking_3})\\&+F_q((2,1,2,1,-6),(0,0,1,-1,0),3\mid t_{\symking_3}))\\&+\delta[q \equiv 1 \Mod{3}]\phi_{-6\psi}(q)q(F_q((2,1,2,1,-6),(3,0,-1,1,3),6\mid t_{\symking_3})\\&+F_q((2,1,2,1,-6),(3,0,1,-1,3),6\mid t_{\symking_5})).   
\end{align*} 
The term $q^2+q+1$ cancels as before; the term $4q\delta[q \equiv 1 \Mod{3}]$ was present in the computation for $\Csf_2\Lsf_2$ and has already been identified. Consider the term $$\delta[q \equiv 1 \Mod{3}]q(F_q((2,1,2,1,-6),(0,0,-1,1,0),3\mid t_{\symking_3})+F_q((2,1,2,1,-6),(0,0,1,-1,0),3\mid t_{\symking_3})).$$ The corresponding extension is $\Q(\sqrt{-3}) \mid \Q.$ 

If $p \equiv 1,7 \Mod{12},$ then $p$ is a quadratic residue modulo 12 and consequently it splits into two primes in $\Q(\sqrt{-3}),$ each with norm 1. Moreover, $p^m \equiv 1 \Mod{12}$ for $m$ even and $p^m \equiv 7 \Mod{12}$ for $m$ odd. In any case, we have $p^m \equiv 1 \Mod{3}.$ 

Observe that in this case, $(\Z/3\Z)^{\times}/\langle p \rangle \simeq \Z/2\Z$ because $p \equiv 1 \Mod{3}.$ Therefore,

\begin{align*}
\exp&\bigg(-\sum_{m=1}^{\infty}p^mF_{p^m}((2,1,2,1,-6),(0,0,1,-1,0),3\mid t_{\symking_3}))\tfrac{p^{-sm}}{m}\\&-\sum_{m=1}^{\infty}p^mF_{p^m}((2,1,2,1,-6),(0,0,1,-1,0),3\mid t_{\symking_3})\tfrac{p^{-sm}}{m}\bigg)\\=&L_p\left(F((2,1,2,1,-6),(0,0,1,-1,0),3\mid t_{\symking_3}),p^{1-s}\right)\cdot L_p\left(F((2,1,2,1,-6),(0,0,1,-1,0),3\mid t_{\symking_3}),p^{1-s}\right)\\=&\prod_{k \in \Z/2\Z}L_p\left(F((2,1,2,1,-6),k(0,0,1,-1,0),3\mid t_{\symking_3}),p^{1-s}\right)\\=&L_p(F((2,1,2,1,-6),(0,0,1,-1,0),3\mid t_{\symking_3}),\Q(\sqrt{-3}),p^{1-s}).  
\end{align*} 

If $p \equiv 5,11 \Mod{12},$ then $p$ is not a quadratic residue modulo 12 and consequently $p$ is inert in $\Q(\sqrt{-3}).$ Observe that in this case, $p \equiv 2 \Mod{3},$ so the group 
$(\Z/3\Z)^{\times}/\langle p \rangle=\{1,2\}/\{1,2\}$ is trivial. 

Moreover, $p^m \equiv 1 \Mod{3}$ for $m$ even and $p^m \equiv 2 \Mod{3}$ for $m$ odd. Hence, for $m$ even, it follows from Lemma \ref{lm: Fp} that $$F_{p^m}((2,1,2,1,-6),(0,0,-1,1,0),3\mid t_{\symking_3})=F_{p^m}((2,1,2,1,-6),(0,0,1,-1,0),3\mid t_{\symking_3})$$ and we get the summand as follows 

\begin{align*}
\exp&\bigg(-\sum_{m=1}^{\infty}\epsilon^m p^{2m}F_{p^{2m}}((2,1,2,1,-6),(0,0,1,-1,0),3\mid t_{\symking_3})\tfrac{p^{-2sm}}{m}\bigg)\\=&L_{p^2}\left(F((2,1,2,1,-6),(0,0,1,-1,0),3\mid t_{\symking_3}),p^{1-s}\right)\\=&L_{p}(F((2,1,2,1,-6),(0,0,1,-1,0),3\mid t_{\symking_3}),\Q(\sqrt{-3}),p^{1-s}).  
\end{align*}

Now consider the last term $$\delta[p \equiv 1 \Mod{3}]\phi_{-6\psi}(p)p(F_p((2,1,2,1,-6),(3,0,-1,1,3),6\mid t_{\symking_3})+F_p((2,1,2,1,-6),(3,0,1,-1,3),6\mid t_{\symking_5})).$$ We have $\Q(\zeta_6)=\Q(\sqrt{-3}),$ so the splitting behaviour is the same as in the previous term. Moreover, since we are working only with odd prime numbers, $$p \equiv 1 \Mod{6} \Leftrightarrow p \equiv 1 \Mod{3}.$$ Doing a similar analysis as we did in the previous term, we obtain the last remaining factor $$L_S(F((2,1,2,1,-6),(3,0,-1,1,3),6\mid t_{\symking_3}),\Q(\zeta_6),s-1,\phi_{-6\psi}).$$
\end{proof}

\section*{Data Availability Statement}

Data sharing does not apply to this article as no datasets were generated or analyzed during the current study.

\section*{Declarations}

\subsection*{Funding}

The authors were supported by ICERM through the Collaborate@ICERM program. Gomes Ribeiro acknowledges the support received from the School of Mathematics at the University of Birmingham through a PhD scholarship; she was also supported by the University of Sydney and the University of Birmingham Turing Scheme during her visit to Sydney, when revisions and corrections to this work were made. Orvis acknowledges that he received support through an IBM PhD Fellowship, and from the University of Colorado, Boulder during work on this paper. Gomes Ribeiro, Salerno and Whitcher were supported by the ICTP during the Winter School on Number Theory and Physics. 

\subsection*{Conflict of interest}

The authors have no conflict of interest to declare.

\section*{References}
\printbibliography[heading=none]
\appendix 
\end{document}